%% file: main.tex
\documentclass[reqno,a4paper]{amsart}
\usepackage{setspace}
\usepackage{color}
\usepackage[colorlinks=true, allcolors=blue,backref]{hyperref}
\usepackage{amsmath, amssymb, amsthm}
\usepackage{mathrsfs}
\usepackage{mathtools}
\usepackage[noabbrev,capitalize,nameinlink]{cleveref}
\crefname{equation}{}{}
\usepackage{fullpage}
\usepackage[noadjust]{cite}
\usepackage{graphics}
\usepackage{pifont}
\usepackage{yfonts}
\usepackage{tikz}
\usepackage{bbm}
\usepackage[T1]{fontenc}

\usepackage{environ}
\usepackage{framed}
\usepackage{url}
\usepackage[linesnumbered,ruled,vlined]{algorithm2e}
\usepackage[noend]{algpseudocode}
\usepackage[labelfont=bf]{caption}
\usepackage{cite}
\usepackage{framed}
\usepackage[framemethod=tikz]{mdframed}
\usepackage{appendix}
\usepackage{diagbox}
\usepackage{subcaption}
\usepackage{graphicx}
\usepackage[textsize=tiny]{todonotes}
\usepackage{tcolorbox}
\allowdisplaybreaks[1]
\usepackage{stmaryrd}
\usepackage{dsfont}
\usepackage{paralist}

\usepackage{thmtools}
\usepackage{thm-restate}

\crefformat{enumi}{#2#1#3}
\crefrangeformat{enumi}{#3#1#4 to~#5#2#6}
\crefmultiformat{enumi}{#2#1#3}
{ and~#2#1#3}{, #2#1#3}{ and~#2#1#3}

\DeclareSymbolFont{symbolsC}{U}{pxsyc}{m}{n}
\SetSymbolFont{symbolsC}{bold}{U}{pxsyc}{bx}{n}
\DeclareFontSubstitution{U}{pxsyc}{m}{n}
\DeclareMathSymbol{\medcircle}{\mathbin}{symbolsC}{7}

\crefname{algocf}{Algorithm}{Algorithms}

\AtBeginEnvironment{appendices}{\crefalias{section}{appendix}} %appendices

\usepackage[color,final]{showkeys} %add in 'final' into parameter to remove showkeys

\colorlet{refkey}{orange!20}
\colorlet{labelkey}{blue!30}

\crefname{algocf}{Algorithm}{Algorithms}

\numberwithin{equation}{section}
\newtheorem{theorem}{Theorem}[section]
\newtheorem{lemma}[theorem]{Lemma}
\newtheorem{proposition}[theorem]{Proposition}

\newtheorem{claim}[theorem]{Claim}

\newtheorem{conjecture}[theorem]{Conjecture}
\newtheorem{corollary}[theorem]{Corollary}
\theoremstyle{remark}
\newtheorem{remark}[theorem]{Remark}
\theoremstyle{definition}
\newtheorem{definition}[theorem]{Definition}

\newcommand{\mb}{\mathbb}

\newcommand{\mbm}{\mathbbm}
\newcommand{\mc}{\mathcal}

\newcommand{\mr}{\mathrm}

\newcommand{\floor}[1]{\left\lfloor {#1} \right\rfloor}
\newcommand{\ceil}[1]{\left\lceil {#1} \right\rceil}
\newcommand{\abs}[1]{\left\vert {#1}\right\vert}

\DeclareMathOperator{\ex}{ex}

\let\R\relax
\newcommand*{\R}{{\mathds{R}}}

\newcommand{\clique}[2]{{#1}^{(#2)}}
\newcommand{\ramsey}[2]{L_{#1}(#2)}
\newcommand{\countPat}[2]{N_{#2}(#1)}
\newcommand{\countClique}[3]{N_{#2,#3}(#1)}

\DeclareRobustCommand{\stirling}{\genfrac\{\}{0pt}{}}

\newcommand{\block}{{\textfrak{B}}}
\newcommand{\blockA}{{\mathrm{A}}}
\newcommand{\blockB}{{\mathrm{B}}}

\newcommand{\AB}{{\blockA\blockB}}
\newcommand{\BA}{{\blockB\blockA}}

\let\originalleft\left
\let\originalright\right
\renewcommand{\left}{\mathopen{}\mathclose\bgroup\originalleft}
\renewcommand{\right}{\aftergroup\egroup\originalright}

\global\long\def\zj#1{\textcolor{cyan}{\textbf{[ZJ comments:} #1\textbf{]}}}
\global\long\def\ma#1{\textcolor{blue}{\textbf{[MA comments:} #1\textbf{]}}}

 \usepackage{graphics}

\newcommand{\noop}[1]{}

\title{Extremal, enumerative and probabilistic results on ordered hypergraph matchings}

\author[Anastos]{Michael Anastos}

\address{Institute of Science and Technology Austria (ISTA).}
\email{michael.anastos@ist.ac.at}

\author[Jin]{Zhihan Jin}

\address{ETH Z\"urich}
\email{zhihan.jin@math.ethz.ch}

\author[Kwan]{Matthew Kwan}

\address{Institute of Science and Technology Austria (ISTA).}
\email{matthew.kwan@ist.ac.at}

\author[Sudakov]{Benny Sudakov}

\address{ETH Z\"urich}
\email{benjamin.sudakov@math.ethz.ch}
\thanks{
Michael Anastos is supported by the European Union’s Horizon 2020 research and innovation
programme under the Marie Sk\l{}odowska-Curie grant agreement No.\ 101034413.
Matthew Kwan is supported by ERC Starting Grant ``RANDSTRUCT'' No.\ 101076777, also funded by the European Union
\includegraphics[width=4.5mm, height=3mm]{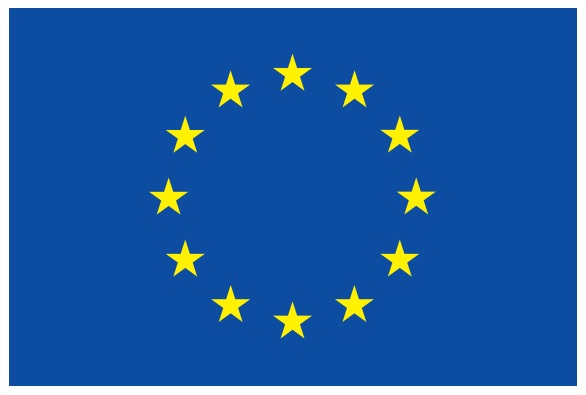}. Zhihan Jin and Benny Sudakov are supported by SNSF grant 200021\_1969.
}

\begin{document}

\maketitle

\newif\ifarxiv
\newif\ifjournal
\arxivtrue
%\journaltrue
\input{introduction}
\input{erdos-szekeres}
\input{random}
\input{counting}
\input{hypergraph-partition}
\input{extremal}

\bibliographystyle{amsplain_initials_nobysame_nomr}
\bibliography{main}

\ifarxiv
\input{appendix}

\fi
\end{document}

%% file: introduction.tex
\begin{abstract}
    An \emph{ordered $r$-matching} is an $r$-uniform hypergraph matching equipped with an ordering on its vertices. These objects can be viewed as natural generalisations of \emph{$r$-dimensional orders}. The theory of ordered 2-matchings is well-developed and has connections and applications to extremal and enumerative combinatorics, probability, and geometry. On the other hand, in the case $r \ge 3$ much less is known, largely due to a lack of powerful bijective tools. Recently, Dudek, Grytczuk and Ruci\'nski made some first steps towards a general theory of ordered $r$-matchings, and in this paper we substantially improve several of their results and introduce some new directions of study. Many intriguing open questions remain.
\end{abstract}

\section{Introduction}
A \emph{matching} is a graph with the property that every vertex is
incident to exactly one edge (equivalently, a matching is a partition of the vertex set into pairs). Given a partition of the vertex set
into two equal-size parts $V_{1},V_{2}$, we say that a matching is \emph{bipartite}
if every edge is between the two parts.

Matchings are fundamental objects in graph theory (and beyond), where
one is usually interested in the existence and identification of matchings
inside larger graphs (see for example the monograph of Lov\'asz and
Plummer~\cite{LP86}). However, matchings can be very interesting objects
in their own right, if one puts an \emph{ordering} on the set of vertices
(for example, we could restrict our attention to matchings on vertex
sets of the form $\{1,\dots,2n\}$, and consider the natural ordering of the integers).

In particular, in an ordered matching there are three different ways
that a pair of distinct edges $e,e'$ can ``interact with each other'', as
follows (write $e[1]<e[2]$ for the two vertices of $e$, write $e'[1]<e'[2]$
for the two vertices of $e'$, and assume without loss of generality
that $e[1]<e'[1]$).
\begin{compactitem}
\item We could have $e[1]<e[2]<e'[1]<e'[2]$ (that is to say, one edge fully
comes before the other one). This configuration is called an \emph{alignment}.
\item We could have $e[1]<e'[1]<e[2]<e'[2]$ (that is to say, the two edges
are interleaved with each other). This configuration is called a \emph{crossing}.
\item We could have $e[1]<e'[1]<e'[2]<e[2]$ (that is to say, one of the
two edges is ``within'' the other). This configuration is called
a \emph{nesting}.
\end{compactitem}
It is a classical result (first
proved by Errera~\cite{Err31}) that the crossing-free matchings on $\{1,\dots,2n\}$ are enumerated by the Catalan numbers $C_{n}$, and an ingenious
bijection (see \cite{Sta}) shows that there are also exactly $C_{n}$
nesting-free matchings on $\{1,\dots,2n\}$. Note that a matching $M$ on $\{1,\dots,2n\}$
is alignment-free if and only if every edge is between $\{1,\dots,n\}$
and $\{n+1,\dots,2n\}$ (i.e., if $M$ is a bipartite matching with
these two parts). Such matchings are in correspondence with permutations
$\sigma\in\mathcal{S}_{n}$, and there are therefore $n!$ of them.

Two of the most important parameters of a permutation $\sigma\in\mathcal{S}_{n}$ are the length $L_{\nearrow}(\sigma)$ of its longest increasing subsequence and the length $L_{\searrow}(\sigma)$ of its longest decreasing subsequence (see for example \cite{Rom15,Sta07} for surveys on the study of these parameters). Two of the highlights in this area are the \emph{Erd\H os--Szekeres theorem}~\cite{ES35}, which says that we always have $L_{\nearrow}(\sigma)\ge \sqrt n$ or $L_{\searrow}(\sigma)\ge \sqrt n$ (i.e., it is not possible to simultaneously avoid long decreasing subsequences and long increasing subsequences), and the \emph{Robinson--Schensted--Knuth correspondence}~\cite{Rob39,Sch61,Knu70} between permutations and Young tableaux, which can be used to enumerate permutations $\sigma\in \mc S_n$ by their values of $L_{\nearrow}(\sigma)$ and $L_{\searrow}(\sigma)$
(and in particular, to study the behaviour of $L_{\nearrow}(\sigma)$ and $L_{\searrow}(\sigma)$ for a \emph{random} permutation $\sigma\in \mc S_n$).

Recalling that permutations $\sigma\in\mathcal{S}_{n}$ are in correspondence with bipartite matchings between $\{1,\dots,n\}$ and $\{n+1,\dots,2n\}$, it turns out that increasing and decreasing subsequences can be described in the language of configurations in matchings: an increasing subsequence corresponds to a set of edges which are pairwise crossing (which we call a \emph{crossing-clique}), and a decreasing subsequence corresponds to a set of edges which are pairwise nesting (which we call a \emph{nesting-clique}).

Extending the huge body of work on increasing and decreasing subsequences, there has been quite some work (see for example \cite{Sta07,CDDSY07,BR01,BJ13,JSW90,KZ06,Kla06,Kra06}), studying nesting-cliques and crossing-cliques (and alignment-cliques, which have the obvious definition) in general matchings. It turns out that many of the techniques that are effective for permutations have natural analogues for general matchings (in particular, there is a variant of the Robinson--Schensted--Knuth correspondence relating matchings to \emph{oscillating tableaux}; see \cite{Sta07}).

Very recently, Dudek, Grytczuk and Ruci\'nski~\cite{DGJR23-r-matching,DGR23-later} made some first steps towards extending the theory of ordered matchings to ordered \emph{hypergraph} matchings (in an $r$-uniform hypergraph matching, every edge contains exactly $r$ vertices, and every vertex is incident to exactly one edge). The jump from graphs to hypergraphs seems to introduce a number of serious difficulties, which should perhaps not be surprising: in much the same way that an ordered matchings generalise permutations, ordered $r$-uniform hypergraph matchings generalise $(r-1)$-tuples of permutations (which are sometimes called ``$r$-dimensional orders'', as they can be described by sets of points in $r$-dimensional space). When $r\ge 3$, there is no known analogue of the Robinson--Schensted--Knuth correspondence for $r$-dimensional orders, and there are a number of longstanding open problems (see for example the survey \cite{Bri93}).

In this paper we substantially extend and improve the results in \cite{DGJR23-r-matching,DGR23-later}, discovering some surprisingly intricate phenomena and moving towards a more complete theory of ordered hypergraph matchings. We also draw attention to a large number of compelling open problems.

\subsection{Ordered hypergraph matchings: basic notions}
We say that a (hyper)graph $H$ is \emph{ordered} if its vertex set $V(H)$ is equipped with a total order (one may think of graphs and hypergraphs which have vertex set $\{1,\dots,N\}$ for some $N$).
\begin{definition}
An ordered hypergraph is said to be an \emph{$r$-matching} if every edge has
exactly $r$ vertices, and every vertex is contained in exactly one
edge. An $r$-matching is said to be \emph{$r$-partite} if, when we divide the vertex set into $r$ contiguous intervals of equal length, every edge of the matching has exactly one vertex in each interval. The \emph{size} of a matching $M$ 
is its number of edges.
\end{definition}

Note that there are 
\begin{equation}
    \frac{(rn)!}{(r!)^{n}n!}\label{eq:count-matchings}
\end{equation}
different $r$-matchings on the vertex set $\{1,\dots,rn\}$. 
The $r$-partite $r$-matchings on $\{1,\dots,rn\}$ are in correspondence with $(r-1)$-tuples of permutations $(\sigma_1,\dots,\sigma_{r-1})\in \mc S_n^{r-1}$, and there are therefore exactly $(n!)^{r-1}$ of them.

\begin{definition}
An \emph{$r$-pattern} is an $r$-matching of size 2 (on the vertex
set $\{1,\dots,2r\}$, say). We can represent an $r$-pattern by a
string of ``$\blockA$''s and ``$\blockB$''s starting with ``$\blockA$'' (where
the vertices from one edge are represented with ``$\blockA$'', and the
vertices of the other edge are represented with ``$\blockB$''). 
\end{definition}
Note that there are exactly $(2\cdot2)!/(2^{2}\cdot2)=3$ different 2-patterns: the alignment (represented by $\blockA\blockA\blockB\blockB$), the crossing (represented by $\AB\AB$) and the nesting (represented by $\AB\BA$). The crossing and nesting are $2$-partite, but the alignment is not (in fact, as we discussed earlier in the introduction, avoidance of the alignment pattern is equivalent to $2$-partiteness).
\begin{definition}
For an $r$-pattern $P$, an $r$-matching $M$ is said to be a \emph{$P$-clique}
if every pair of edges of $M$ are order-isomorphic to $P$.
For any $r$-pattern $P$ and $r$-matching $M$, let $L_{P}(M)$
be the size of the largest $P$-clique in $M$.
\end{definition}

In total, there are $(2r)!/(2(r!)^2)$ different $r$-patterns, but there is a subtlety that occurs only for uniformities $r\ge 3$: for some $r$-patterns $P$ it is simply not possible to have a large $P$-clique.
\begin{definition}
We say that an $r$-pattern $P$ is \emph{collectable} if there are arbitrarily large
$P$-cliques.
\end{definition}

For example, it is easy to see that the 3-pattern $\blockA\blockA\blockB\blockA\blockB\blockB$ is not
collectable. 
In fact, Dudek, Grytczuk and Ruci\'nski~\cite{DGR23-later} showed for every non-collectable pattern, the largest possible clique size is 2. To this end, they gave a characterisation of the collectable
patterns: they are precisely those patterns that are \emph{splittable},
as follows.
\begin{definition}\label{definition: splittable}
A \emph{run} in a pattern $P$ is a sequence of consecutive vertices
in the same edge of $P$. A pattern $P$ is \emph{splittable} if it
can be partitioned into \emph{blocks} of consecutive vertices, each consisting
of two runs of the same length coming from different edges of $P$.
If $P$ is splittable, then this partition is uniquely determined;
we call it the \emph{block partition} of $P$.
\end{definition}

For example, the pattern $\blockA\blockA\blockB\blockB\blockA\blockB\blockB\blockA$ is splittable (the divisions
between the blocks are described by $|\blockA\blockA\blockB\blockB|\blockA\blockB|\blockB\blockA|$ and the block partition has parts $\{1,2,3,4\}$, $\{5,6\}$ and $\{7,8\}$). 
It is easy to check that there are exactly $3^{r-1}$ different collectable $r$-patterns, exactly $2^{r-1}$ of which are $r$-partite (note that the $r$-partite $r$-patterns are precisely the collectable $r$-patterns whose block partition has $r$ blocks).

For every collectable $r$-pattern $P$, there is only one way to form a $P$-clique on a given set of vertices. For example, in the case $r=2$, the only way to form an alignment-clique on $\{1,\dots,2n\}$ is with the edges 
\[\{1,2\},\{3,4\},\dots,\{2n-1,2n\},\]
the only way to form a crossing-clique is with the edges
\[\{1,n+1\},\{2,n+2\},\dots,\{n,2n\},\]
and the only way to form a nesting-clique is with the edges
\[\{1,2n\},\{2,2n-1\},\dots,\{n,n+1\}.\]

We are now ready to present the main results of the paper. To briefly summarise: in \cref{subsec:intro-ramsey} we consider Ramsey--type questions (generalising the Erd\H os--Szekeres theorem), in \cref{subsec:intro-random} we study the size of the largest $P$-clique in a \emph{random} ordered matching, and in \cref{subsec:intro-enumeration} we give estimates for the number of ordered hypergraph matchings avoiding $P$-cliques of a given size. Throughout, we discuss a number of auxiliary results we prove along the way; in particular, in \cref{subsec:intro-extremal} we discuss some contributions to the extremal theory of ordered hypergraphs, which we believe to be of independent interest. In \cref{subsec:further-directions} we present a large number of open problems and directions for further study.

\subsection{Ramsey-type questions}\label{subsec:intro-ramsey}
Recall that the Erd\H os--Szekeres theorem says that every permutation $\sigma\in \mc S_n$ has an increasing or decreasing subsequence of length at least $\sqrt n$ (and it is easy to construct a permutation showing that this is best-possible, by taking a certain type of ``product'' of an increasing sequence of length $\sqrt n$ with a decreasing sequence of length $\sqrt n$). The Erd\H os--Szekeres theorem falls under the umbrella of \emph{Ramsey theory}: no matter how ``disordered'' a permutation is, it must have a long subsequence which is ``completely homogeneous''.

In \cite{DGJR23-r-matching}, Dudek, Grytczuk and Ruci\'nski adapted the Erd\H os--Szekeres theorem to ordered (2-uniform) matchings. They showed that in every 2-matching $M$ of size $n$,
one can find an alignment-clique, a crossing-clique or a nesting-clique of size at least $n^{1/3}$.
They also showed that this result is best-possible, with the same type of product construction used to demonstrate optimality of the Erd\H os--Szekeres theorem.

To discuss the situation for higher-uniformity hypergraphs, we introduce some notation.
\begin{definition}
Let $L(M)=\max_P L_P(M)$ be the size of the largest clique (of any pattern) in $M$, and let $\ramsey r n$ be the minimum value of $L(M)$ among all ordered
$r$-matchings of size $n$.
\end{definition}
In this notation, the aforementioned result of Dudek, Grytczuk and Ruci\'nski~\cite{DGJR23-r-matching} is $\ramsey 2 n=\lceil n^{1/3}\rceil$.

As first observed by Burkill and Mirsky~\cite{BM73} and Kalmanson~\cite{Kal73}, it is actually very easy to iterate the Erd\H os--Szekeres theorem to obtain an optimal bound for \emph{$r$-partite} $r$-matchings: every $r$-partite $r$-matching $M$ of size $n$ has $L(M)\ge n^{1/2^{r-1}}$, and a product construction shows that this is best possible\footnote{In the language of permutations, this can be equivalently formulated as the fact that in any $(r-1)$-tuple of length-$n$ permutations there is a set of at least $n^{1/2^{r-1}}$ indices in which each of our permutations is monotone. This is one of many different ways to generalise the Erd\H os--Szekeres theorem to ``higher dimensions''; for example, see also \cite{ST01,FPSS12,BM14,LS18,BST22,GKS}.}. It is not quite so obvious how to prove a bound on $\ramsey r n$ by iterating the bound $\ramsey 2 n\ge  n^{1/3}$, but in a follow-up paper~\cite{DGR23-later}, Dudek, Grytczuk and Ruci\'nski showed that this is indeed possible if one is willing to give up a constant factor: they proved the general lower bound $\ramsey r n\ge c_r n^{1/3^{r-1}}$ (for some $c_r>0$ depending on $r$).

Na\"ively, this seems to suggest that $\ramsey r n$ is of order $n^{1/3^{r-1}}$ (viewing $r$ as a constant, while $n$ is large). However, it is not clear how to prove a corresponding upper bound using the usual ``product-type'' constructions, because cliques can interact with each other in surprisingly intricate ways. Via product-type constructions, Dudek, Grytczuk and Ruci\'nski were only able to prove the upper bound $\ramsey r n \le n^{1/(2^{r-1}+2)}$.

As our first result, we significantly improve both the lower and upper bounds for $\ramsey r n$, showing that the exponent scales roughly like $1/2^r$.

\begin{theorem}\label{thm:ES}
For $r \ge 2$, we have 
\[ \frac{1}{r-1}\cdot n^{1/((r+1)2^{r-2})} \le \ramsey r n\le \lceil n^{1/(2^r-1)}\rceil.\]
\end{theorem}
The upper bound in \cref{thm:ES} comes from a product-like construction, and the lower bound in our proof of \cref{thm:ES} is proved via an argument that partitions the possible $r$-patterns into roughly $2^r$ subsets with ``poset-like'' structure. It is possible to push this method further, obtaining slightly stronger lower bounds by considering more intricate partitions of the $r$-patterns, and we are actually not sure what the limit of the method is. As an illustration, with some intricate combinatorial analysis we managed to use our method to obtain the correct order of magnitude of $\ramsey 3 n$ and $\ramsey 4 n$.

\begin{theorem}\label{thm:ES-3-4}We have
    \[
      \frac{1}{2}\cdot n^{1/7}\le \ramsey 3 n\le n^{1/7},\qquad
      \frac{1}{4}\cdot n^{1/15}\le \ramsey 4 n\le n^{1/15}.
    \]\label{eq:ES-3-4}
\end{theorem}

In \cref{sec:ramsey-lower}, we describe our strategy to prove lower bounds, and show how to use it to prove the general lower bound on $\ramsey r n$ in \cref{thm:ES} and the sharp lower bound on $\ramsey 3 n$ in \cref{eq:ES-3-4}.  The proof of the sharp lower bound on $\ramsey 4 n$ follows a similar approach but has much more complicated casework;
\ifarxiv
we defer it to \cref{appendix: ES}.
\fi
\ifjournal
the details appear in the appendix of the arXiv version of this paper.
\fi
In \cref{sec:ramsey-upper} we prove the upper bound in \cref{thm:ES} (which also implies the upper bounds in \cref{thm:ES-3-4}).

\subsection{Random matchings}\label{subsec:intro-random}

One of the most notorious problems in the theory of permutations is the \emph{Ulam--Hammersley problem}, to describe the distribution of the longest increasing permutation $L_{\nearrow}(\sigma)$ in a random permutation $\sigma\in \mc S_n$. This problem was famously resolved by Baik, Deift and Johansson~\cite{BDJ99}, but one of the most important milestones along the way was a theorem of Logan and Shepp~\cite{LS77} and Vershik and Kerov~\cite{VK77} (see also the alternative proofs \cite{AD95,Sep98,Joh98,Gro01}), establishing that the expected value of $L_{\nearrow}(\sigma)$ is asymptotically $2\sqrt n$.

These results were extended to random matchings by Baik and Rains~\cite{BR01} (see also \cite{Sta07,BJ13}). Specifically, they
proved that in a random matching $M$ on $\{1,\dots,2n\}$, the expected values of 
$L_{\mathrm{crossing}}(M)$ and $L_{\mathrm{nesting}}(M)$ are both $(\sqrt{2}+o(1))\sqrt{n}$ (as part of a tour-de-force where they found the asymptotic distribution of these quantities).

All this work uses variations on the Robinson--Schensted--Knuth correspondence (or related bijective tools), and does not easily generalise to hypergraphs (or to higher dimensions). For example, if $M$ is a random $r$-partite $r$-matching, then for any $r$-partite $r$-pattern $P$, the random variable $L_P(M)$ has the same distribution as the length of the longest common increasing subsequence in $r-1$ independent random permutations $\sigma_1,\dots,\sigma_{r-1}\in \mc S_n$. This random variable is notoriously hard to study: Bollob\'as and Winkler~\cite{BW88} proved that its expected value is asymptotic to $c_rn^{1/r}$ for some constant $c_r>0$, but the value of $c_r$ is unknown for all $r\ge 3$.

In \cite{DGR23-later} (improving their results in \cite{DGJR23-r-matching}), Dudek, Grytczuk and Ruci\'nski proved that there are constants $c_{r}',c_{r}''\ge0$
such that for any collectable $r$-pattern $P$, and a random $r$-matching
$M$, we have $c_{r}'n^{1/r}\le L_{P}(M)\le c_{r}''n^{1/r}$ whp\footnote{We say that an event holds \emph{with high probability}, or \emph{whp
}for short, if it holds with probability $1-o(1)$. Here and for the
rest of the paper, asymptotics are as $n\to\infty$, unless explicitly
stated otherwise.}. They also conjectured that $L_{P}(M)/n^{1/r}$ converges in probability,
to a constant that only depends on $r$. However, this conjecture is already false for $r=2$: via analysis of a simple renewal process, Justicz, Scheinerman and Winkler~\cite{JSW90} proved that for a random matching $M$ on $\{1,\dots,2n\}$, the size of the largest alignment-clique is $(2/\sqrt{\pi}+o(1))\sqrt{n}$ whp, while the work of Baik and Rains described above implies that the largest crossing-clique and nesting-clique both have size $(\sqrt{2}+o(1))\sqrt{n}$. (This was likely missed by Dudek, Grytczuk and Ruci\'nski on account of the result in \cite{JSW90} being stated in the language of \emph{random interval graphs}, but the equivalence to random matchings is straightforward).

Combining a subadditivity argument with Talagrand's concentration inequality, we are able to show that for a collectable pattern $P$ and a random matching
$M$, the random variable $L_{P}(M)/n^{1/r}$ does converge to a limit. This limit may depend on $P$, but only through the \emph{type} of $P$, as follows.
\begin{definition}
Let $P$ be an $r$-pattern with block partition $J_{1}\cup\dots\cup J_{\ell}$.
The \emph{type} of $P$ is the partition $|J_{1}|/2+\dots+|J_{\ell}|/2$
of $r$ (in the number-theoretic sense).
\end{definition}

For example, $|\blockA\blockA\blockB\blockB|\AB|\BA|$ and $|\AB|\blockA\blockA\blockB\blockB|\AB|$ both have type $2+1+1$.
\begin{theorem}\label{thm:limit-exists}
Fix an $r$-pattern $P$, and let $M$ be a random $r$-matching of size
$n$. Then we have 
\[
\frac{L_{P}(M)}{n^{1/r}}\overset p\to b_{P}
\]
 for some $b_{P}>0$ depending only on the type of $P$.
\end{theorem}

Actually determining the constants $b_{P}$ seems to be a hard problem (in particular, if $P$ is an $r$-partite pattern then our proof shows that $b_P$ can be expressed in terms of the constant $c_r$ in the Bollob\'as--Winkler theorem described above). However, we are able to prove some bounds on the $b_{P}$ in some special cases. In particular, we show that the conjecture of Dudek, Grytczuk and Ruci\'nski is false for all $r\ge 2$. (Recall that the $\Gamma$ function is the analytic continuation of the factorial function).
\begin{proposition}\label{prop:limit-information}
Let the constants $b_{P}$ be as in \cref{thm:limit-exists}.
\begin{enumerate}
\item If $P$ has type $r$ (i.e., if the block partition of $P$ has a
single block), then $b_{P}=1/\Gamma((r+1)/r)$.
\item If $P$ has type $1+\dots+1$ (i.e., if $P$ is $r$-partite), then $b_{P}=c_r(r!)^{1/r}/r>1/\Gamma((r+1)/r)$, where $c_r$ is the constant from the Bollob\'as--Winkler theorem mentioned earlier in this section.
\end{enumerate}
\end{proposition}

We prove \cref{thm:limit-exists,prop:limit-information} in \cref{sec:random}, after proving a certain necessary generalisation of the Bollob\'as--Winkler theorem in \cref{sec:BW-variants}.

\subsection{Enumeration}\label{subsec:intro-enumeration}

For an $r$-pattern $P$, let $\countPat{n}{P}$ denote the number of ordered $r$-matchings on the vertex set $\{1,\dots,rn\}$ which are $P$-free (i.e., no two edges form $P$). 
First notice that if $P$ is not $r$-partite then every $r$-partite matching is $P$-free. 
In particular, $\countPat{n}{P}$ is at least the number of $r$-partite matchings of size $n$, which is exactly $(n!)^{r-1}$. 
In combination with \cref{eq:count-matchings} (and Stirling's approximation) this is already enough to approximate $\countPat{n}{P}$ up to exponential factors (for constant $r$).

\begin{proposition}\label{prop:non-r-partite-count}
Fix a constant $r\in \mb N$. If $P$ is not $r$-partite then
\[\countPat{n}{P}=e^{O_r(n)} n^{(r-1)n}.\]
\end{proposition}
(In this paper, subscripts on asymptotic notation indicate quantities that should be viewed as constants: the constant factor implicit in ``$O_r(n)$'' is allowed to depend on $r$).

The case where $P$ is $r$-partite is more delicate (and the value of $\countPat{n}{P}$ is quite different), but we are able to obtain estimates of similar quality, as follows.

\begin{theorem}
\label{thm:count-basic}
Fix a constant $r\in \mb N$. If $P$ is $r$-partite then
\[\countPat{n}{P}=e^{O_r(n)} n^{(r-1-1/(r-1))n}.\]
\end{theorem}

\begin{remark}\label{rem:computer-check}
    In the case $r=2$ there are only two possibilities for an $r$-partite pattern $P$: a crossing $|\AB|\AB|$ or a nesting $|\AB|\BA|$. As discussed in the introduction, $\countPat{n}{P}$ is \emph{exactly} the same in both cases (in particular, there is a well-known bijection between crossing-free and nesting-free matchings of a given size). However, this does \emph{not} generalise straightforwardly to higher dimensions: a computer search shows that if $P$ is the 3-uniform pattern represented by $|\AB|\AB|\AB|$, then $\countPat 4 P=8626$, whereas if $P$ is the 3-uniform pattern represented by $|\AB|\BA|\BA|$ we have $\countPat 4 P=8630$.
\end{remark}

The lower bound in \cref{thm:count-basic} is a direct consequence of a result due to Brightwell~\cite{Brightwell92} on linear extensions of $r$-dimensional posets, but the upper bound is new (it is obtained by a general connection to extremal ordered hypergraph theory).

We are also interested in enumerating ordered matchings by the size of their largest $P$-clique. We write $\countClique{n}{P}{m}$ for the number of ordered $r$-matchings $M$ on the vertex set $\{1,\dots,rn\}$ which satisfy $L_P(M)<m$ (i.e., they do not contain a $P$-clique of size $m$). 
Note that if $P$ is not $r$-partite, then the same considerations as for \cref{prop:non-r-partite-count} show that $\countClique{n}{P}{m}=e^{O_r(n)} n^{(r-1)n}$ for all $m$ (i.e., varying $m$ can only affect $\countClique{n}{P}{m}$ by a factor of $e^{O_r(n)}$). However, if $P$ is $r$-partite, we get a significant dependence on $m$, as per the following generalisation of \cref{thm:count-basic} (note that \cref{thm:count-basic} corresponds to the case $m=2$).

\begin{theorem}
    \label{thm:count}
    Fix a constant $r\in \mb N$. If $P$ is $r$-partite then for any $2 \le m\le n^{1/r}$ we have 
\[\countClique{n}{P}{m}= e^{O_r(n)}(m-1)^{(r/(r-1))n} n^{(r-1-1/(r-1))n}.\]
\end{theorem}

We remark that the condition $m\le n^{1/r}$ in \cref{thm:count} is not just an artifact of the proof: recall from \cref{subsec:intro-random} that for almost every size-$n$ ordered matching $M$ we have $L_P=O(n^{1/r})$.

We prove the upper bound in \cref{thm:count} via a general lemma (\cref{thm:extremal-to-counting}) which estimates $\countClique{n}{P}{m}$ in terms of a certain \emph{extremal} parameter (namely, the maximum number of edges in an ordered $r$-uniform hypergraph on $\ceil{n/2}$ vertices with $L_P(M)<m$). This type of reduction goes back to Alon and Friedgut~\cite{AF00} (see also \cite{Kla00,Fox13}); in particular, we adapt a proof in a similar high-dimensional situation due to Cibulka and Kyncl~\cite{CK17}. We believe that the relevant extremal parameter is of independent interest, and we discuss it further in \cref{subsec:intro-extremal}.

For the lower bound in \cref{thm:count}, we obtain a new estimate on the number of $(r-1)$-tuples of length-$n$ permutations which have no common increasing subsequence of length $m$ (see \cref{rem:common-increasing}), by ``reverse-engineering'' the techniques in the upper bound. 
This lower bound, and the ideas in its proof, may be of independent interest (in particular, our approach is very different to the probabilistic approach of Brightwell~\cite{Brightwell92} for the case $m=2$).

We prove \cref{thm:count} (which implies \cref{thm:count-basic}) in \cref{sec:counting}.

\subsection{Extremal results}\label{subsec:intro-extremal}
As briefly mentioned in \cref{subsec:intro-enumeration}, in our study of enumerative questions for ordered matchings we encounter some extremal problems for ordered hypergraphs. We believe our extremal results to be of independent interest so we take a moment to discuss them here. First, we define \emph{ordered extremal numbers}, which are variants of the classical \emph{extremal numbers} (or \emph{Tur\'an numbers}) for unordered graphs.

\begin{definition}
Let $G,F$ be ordered $r$-uniform hypergraphs (ordered $r$-graphs, for short). We say $G$ is \emph{$F$-free} if it contains no subgraph isomorphic to $F$ (where the isomorphism must preserve the order of the vertices). Let $\ex_<(n,F)$ denote the maximum number of edges in an $F$-free $n$-vertex ordered $r$-graph.
\end{definition}

In the case $r=2$ (i.e., the case of graphs), ordered extremal numbers have been extensively studied (see the survey by Tardos~\cite{Tardos18} and the references within).
For general $r\ge 2$, much less is known, though there is literature on similar problems for \emph{cyclically} ordered hypergraphs, motivated by geometric considerations (in particular, due to applications to \emph{convex geometric hypergraphs}; see for example \cite{CP92, BRS01, PP03, BKV03, Brass04, ADMOS17, FJKMV20-tight-path-cgg, FJKMV21-cgg, GL21,FMOV22}), and much of this work has implications for ordered extremal numbers $\ex_<(n,F)$.

% Notice that if we partition $\{1,\dots,n\}$ into $r$ contiguous intervals, each of length at least $\floor{n/r}$, and put all edges that has exactly one incident vertex in each of the $r$ intervals, then this graph $G$ contains no $r$-pattern $P$ which is not $r$-partite.
% Thus, $\ex_<(n,P)\ge\floor{n/r}^r$ and $\ex_<(n,P)=\Theta_r(n^r)$ for every $r$-pattern $P$ that is not $r$-partite.
Our first contribution in this direction is that we are able to nail down the exact value $\ex_<(n, P)$ for any $r$-partite $r$-pattern $P$ (this extremal parameter is relevant for \cref{thm:count-basic}). 
For convenience, we define $(x)_+ := \max(x,0)$.
\begin{theorem} \label{theorem: turan number for m=2}
  Let $r, n\ge 1$ and $P$ be any $r$-partite $r$-pattern.
  Then, 
  \[\ex_<(n, P) = \binom{n}{r}-\binom{(n-r)_+}{r}.\]
\end{theorem}
We remark that in the case where $P$ is the $r$-uniform ``generalised crossing'' pattern represented by $|\AB|\AB|\cdots|\AB|$, the result of \cref{theorem: turan number for m=2} was already known, thanks to recent work of F{\"u}redi, Jiang, Kostochka, Mubayi and Verstra{\"e}te~\cite{FJKMV21-cgg}. (It was also known in the case where $P$ is a 2-uniform nesting, as we will discuss after the next theorem).

For an $r$-pattern $P$ and a positive integer $m$, we use $\clique{P}{m}$ to denote the $P$-clique of size $m$ on the vertex set $[rm]$ (so $\clique{P}{2}=P$).
We are not quite able to pin down the values of the ordered extremal numbers $\ex_<(n, \clique{P}{m})$, but we are able to obtain some quite strong bounds (which are an ingredient in our proof of \cref{thm:count}).
\begin{theorem}\label{thm:extremal-clique}
    Let $r,n,m\ge 1$ and let $P$ be an $r$-partite $r$-pattern.
    \begin{enumerate}
        \item In general, we have
        \[\ex_<(n, \clique{P}{m}) \ge \binom{n}{r}-\binom{(n-r(m-1))_+}{r}.
        \]
        \item \smallskip If $n$ is sufficiently large (in terms of $r$) then
        \[
        \ex_<(n, \clique{P}{m})\le O\left(r^2(m-1)\binom n {r-1}\right).
        \]
         \item If $P$ is the "alternating" $r$-pattern represented by $|\AB|\BA|\AB|\BA|\cdots|$, then
         \[\ex_<(n, \clique{P}{m}) = \binom{n}{r}-\binom{(n-r(m-1))_+}{r}.\]
    \end{enumerate}
\end{theorem}
\begin{remark}
    If we view $r$ as a constant, and assume $m=o_r(n)$, then one can check that
    \[\binom{n}{r}-\binom{(n-r(m-1))_+}{r}=(1+o(1))r(m-1)\binom{n}{r-1}.\]
    So, our upper and lower bounds in (1) and (2) differ by a factor of $O(r)$.
\end{remark}
The $r=2$ case of \cref{thm:extremal-clique}(3) (i.e., an exact result in the case where $P$ is the 2-unifom nesting $|\AB|\BA|$) can be derived from the known results on graphs with bounded \emph{queue number} due to Pemmaraju~\cite{Pemmaraju92} and Dujmovi\'c and Wood~\cite{DW04}. We also remark that Capoyleas and Pach~\cite{CP92} previously studied the case where $P=|\AB|\AB|$ is a 2-uniform crossing; in this case they obtained the exact result
\[\ex_<(n,\clique{P}{m})=\binom{n}{2}-\binom{(n-2(m-1))_+}{2}.\]
For general $r\ge 2$, F{\"u}redi, Jiang, Kostochka, Mubayi and Verstra{\"e}te~\cite{FJKMV21-cgg} studied the case where $P$ is the $r$-uniform ``generalised crossing'' pattern represented by $|\AB|\AB|\cdots|\AB|$. We have already mentioned their exact result in the case $m=2$; for all $m$ they proved 
\begin{equation}\label{eq: extremal number for ABAB}
(1-o_{r,m}(1))r(m-1)\binom{n}{r-1}\le \ex_<(n,\clique{P}{m}) \le 2(r-1)(m-1)\binom{n}{r-1}.
\end{equation}
(Recall that subscripts on asymptotic notation indicate quantities that should be viewed as constants, so the ``$o_{r,m}(1)$'' term goes to zero as $n\to\infty$, holding $r,m$ fixed).
Their lower bound construction is the same as our construction for (1) (in the case where $P$ is a generalised crossing), but they only analysed its number of edges asymptotically. 
Their upper bound for this specific case is stronger than our general bound in (2). 

A variety of techniques are involved in the proof of \cref{theorem: turan number for m=2} and the various parts of \cref{thm:extremal-clique}. In particular, for our proof of \cref{thm:extremal-clique}(2) we refine a partitioning lemma for ordered hypergraphs due to F\"uredi, Jiang, Kostochka, Mubayi and  Verstra\"ete~\cite{FJKMV21-partition}. We present this refinement, and discuss it further, in \cref{sec:partitioning}.

The proofs of \cref{theorem: turan number for m=2,thm:extremal-clique} appear in \cref{sec:extremal}.

\subsection{Further directions}\label{subsec:further-directions}
There are a great number of compelling further directions for study.
\subsubsection{Ramsey-type questions} The most obvious open question in this direction is to close the gap between the lower and upper bound in \cref{thm:ES}. It is tempting to guess that the upper bound is sharp, and that $\ramsey r n$ has order of magnitude $n^{1/(2^r-1)}$ for any constant $r$ (this is true for $r\in\{2,3,4\}$). It is even possible that such a bound can be proved via our ``poset partititioning'' method, with judicious choice of the relevant partitions (perhaps for small $r$, appropriate partitions can be found via computer search). 
However, our investigations do not suggest any general structural reason why a bound of this form should hold.

Even in the cases $r\in\{2,3,4\}$ where we know the order of magnitude of $\ramsey r n$, it would be interesting to obtain the exact value (or at least an asymptotic estimate). Also, it is possible to consider ``off-diagonal'' versions of the problem, where one treats different patterns differently. For example, fixing a positive integer $k_P$ for each $r$-pattern $P$, what is the maximum possible size of an $r$-matching which has no $P$-clique of size $k_P$ for any $P$? Actually, this general setting was considered in \cite{DGJR23-r-matching,DGR23-later}, and optimal bounds were obtained in the case $r=2$.
Our methods seem to be suitable for attacking the larger-$r$ case as well, but the relevant casework seems likely to be even more fiddly than for \cref{thm:ES-3-4} (one would need to completely characterise how all the different subsets of patterns can interact with each other).

\subsubsection{Random ordered matchings} The obvious open question in this direction is to determine the values of the constants $b_P$ in \cref{thm:limit-exists}. As we have discused, this is a difficult problem in the case where $P$ is $r$-partite (in which case it amounts to determination of the Bollob\'as--Winkler constant $c_r$), but it may be tractable for certain special $P$.

For example, in the case where $P$ is a 3-pattern which has type $2+1$, the problem of determining $b_{2+1}:=b_P$ seems like it might be of ``intermediate difficulty'' between the Ulam--Hammersley problem (of determining $c_2=2$) and the problem of determining the Bollob\'as--Winkler constant $c_3$. For the Ulam--Hammersley problem, there is a well-known interacting particle process (\emph{Hammersley's interacting particle process}) which captures the limiting behaviour of the longest increasing subsequence of a random permutation, and Aldous and Diaconis~\cite{AD95} managed to give a fairly ``soft'' proof that $c_2=2$ by studying this process. It seems that one can design a variant of Hammersley's process which is suitable for studying $b_{2+1}$, though this process lacks certain symmetries that Aldous and Diaconis used in their analysis.

\subsubsection{Enumeration} The results in \cref{subsec:intro-enumeration} have exponential error terms, and of course it would be interesting to sharpen these. However, there may be some limitations to what is possible to accomplish without determining the limits $b_P$ in \cref{thm:limit-exists}: in the regime where $m$ is of order $n^{1/r}$, studying $\countClique n Pm$ amounts to studying the large deviations of $L_P(M)$ in a random $r$-matching $M$.

Even if exponential error terms cannot be eliminated, it might be interesting to determine their dependence on $r$ (recall that in \cref{thm:count,thm:count-basic} we treat $r$ as a constant).

\subsubsection{Extremal problems}
Given the discussion in \cref{subsec:intro-extremal}, it is natural to conjecture that the lower bound in \cref{thm:extremal-clique}(1) is always sharp, as follows.
\begin{conjecture}\label{conjecture: turan numbers are the same}
  Let $r,n,m \ge 1$ and $P$ be an $r$-partite $r$-pattern.
  Then,
  \[\ex_<(n, \clique{P}{m}) = \binom{n}{r}-\binom{(n-r(m-1))_+}{r}.\]
\end{conjecture}
In particular, given \cref{thm:extremal-clique}(3), an indirect route to prove \cref{conjecture: turan numbers are the same} would be to prove that $\ex_<(n,\clique{P}{m})$ is the same for every $r$-partite $r$-pattern $P$. In the case $r=2$, there is a bijective proof of (a vast generalisation of) this fact, due to Jonsson and Welker~\cite{JW07} (see also \cite[Corollary~2.5]{Mier07}). However, the considerations in \cref{rem:computer-check} suggest some difficulties in generalising such bijective proofs to higher uniformities.

Also, although there is no obvious connection to enumeration, it may still be of interest to determine $\ex_<(n, \clique{P}{m})$ in the case where $P$ is not $r$-partite. It is easy to show that $\ex_<(n, \clique{P}{m})\ge (n/r)^r$ when $n$ is a multiple of $r$, but this does not seem to be best-possible.

\subsection{Notation} 
We use standard asymptotic notation throughout, as follows. For functions $f=f(n)$ and $g=g(n)$, we write $f=O(g)$ to mean that there is a constant $C$ such that $|f(n)|\le C|g(n)|$ for sufficiently large $n$. Similarly, we write $f=\Omega(g)$ to mean that there is a constant $c>0$ such that $f(n)\ge c|g(n)|$ for sufficiently large $n$. We write $f=\Theta(g)$ to mean that $f=O(g)$ and $f=\Omega(g)$, and we write $f=o(g)$ or $g=\omega(f)$ to mean that $f(n)/g(n)\to0$ as $n\to\infty$. Subscripts on asymptotic notation indicate quantities that should be treated as constants.

For $n, r \ge 1$, we use $K_n^{(r)}$ to indicate the ordered $r$-uniform clique on the vertex set $\{1,\dots,n\}$.
For an ordered hypergraph $H$, we write $V(H)$ for its vertex set, $E(H)$ for its edge set and $e(H)=|E(H)|$ for its number of edges. For $e\in E(H)$, we write $e[1],\dots,e[r]$ for the vertices of $e$ (ordered according to the vertex ordering of $H$).
When we say that an ordered hypergraph $H'$ is a \emph{subgraph} of another ordered hypergraph $H$, we mean that the order of the vertices is maintained. (We use the notation $H' \subseteq H$).

We also remind the reader that $(x)_+$ is defined to be $\max(x,0)$.

\section{Preliminaries}
We first recall the following corollary of Mirsky's theorem~\cite{Mirsky71}, for our proof of \cref{thm:ES}.
\begin{theorem} \label{thm:mirsky}
  Suppose $P$ is a partially ordered set (poset) with $n$ elements, that contains no chain of length $x$ for some $x \in \mb{R}_+$.
  Then, $P$ contains an antichain of size at least $n/x$.
\end{theorem}

At several points in the paper (related to the directions introduced in \cref{subsec:intro-extremal}) we will need the following enumerative identity.
\begin{lemma}\label{lemma: count the number of r-tuples with given gap}
  Let $n \ge r \ge 1$ and $\delta_1,\dots,\delta_{r-1} \in \mb{Z}_{\ge 0}$. Then, the number of $r$-tuples $(a_1,\dots,a_r)\in \{1,\dots,n\}^r$ satisfying $a_{i+1}-a_i>\delta_i$ for all $1 \le i \le r-1$ is exactly
  \[\binom{(n-\sum_{k=1}^{r-1}\delta_k)_+}{r}.\]
\end{lemma}
\begin{proof}
We may assume $n\ge \sum_{k=1}^{r-1}\delta_k$, or else the quantity under consideration is zero.
For an $r$-tuple $(a_1,\dots,a_r)$ as in the lemma statement, let $f(a_1,\dots,a_r)=(b_1,\dots,b_r)$, where $b_1=a_1$ and $b_i=a_i-\delta_{i-1}$ for $2\leq  i \leq r$. 
The desired result follows by noting that $f$ is a bijection between the set of tuples under consideration and the set of unordered $r$-tuples in $\{1,2,\dots,n-\sum_{k=1}^{r-1}\delta_k\}$.
% number of nonegative integer solutions to the equation $\sum_{i=1}^{r-1} x_i= n-r-\sum_k\delta_k$.  The desired result follows.
\end{proof}

For \cref{thm:limit-exists} we need a number of probabilistic tools. First, we need a bounded-differences concentration inequality for permutations.
\begin{theorem}\label{thm:mcd-permutation}
Let $f:\mathcal S_n\to \mb R$ be a function which has ``bounded differences'' in the sense that for any $\sigma\in \mc S_n$ and any transposition $\tau\in \mc S_n$, we have $|f(\tau\circ\sigma)-f(\sigma)|\le c$. Then, if $\sigma\in \mathcal S_n$ be a uniformly random permutation of length $n$, we have
\[\Pr[|f(\sigma)-\mb Ef(\sigma)|\ge t]\le \exp\left(-\frac{2t^2}{c^2 n}\right).\]
\end{theorem}
This inequality seems to have first been proved by McDiarmid in the case where $c=1$ (see \cite[p.~18]{McD98}\footnote{There is a typo in this book; in the inequality we are citing, the exponent should really be $2t^2/n$, not $2t^2/n^2$.}).

\begin{remark}\label{rem:permutation-matching}Note that one can obtain a random ordered $r$-matching on the vertex set ${1,\dots,rn}$ via a random permutation $\sigma\in \mc S_{rn}$: simply take the matching whose edges are \[\{\sigma(1),\dots,\sigma(r)\},\; \{\sigma(r+1),\dots,\sigma(2r)\},\;\dots,\;\{\sigma(rn-r+1),\dots,\sigma(rn)\}.\]
\end{remark}

We also need a version of Talagrand's concentration inequality (see
for example \cite[Theorem~2.29 and~Equation~(2.43)]{JLR00}
\begin{theorem}\label{thm:talagrand}
Consider a function $X=f(Z_{1},\dots,Z_{n})$ of independent random
objects $Z_{1}\in\Lambda_{1},\dots,Z_{n}\in\Lambda_{n}$. Suppose that the following conditions are satisfied.
% Suppose that there is a ``certifying function'' $\psi$ such that the following conditions are satisfied.
\begin{compactitem}
\item If $z,z'\in\prod_{i=1}^{n}\Lambda_{i}$ differ only in the $i$th
coordinate, then $|f(z)-f(z')|\le1$.
\item \smallskip For any $z\in\prod_{i=1}^{n}\Lambda_{i}$, there is a subset of indices
$J\subseteq\{1,\dots,n\}$ with $|J|\le f(z)$, such that for
any $y\in\prod_{i=1}^{n}\Lambda_{i}$ which agrees with $z$ on the
coordinates indexed by $J$, we have $f(y)\ge f(z)$.
\end{compactitem}
Then, for some universal constant $\gamma>0$, we have
\[
\Pr[|X-\mb E X|\ge t]\le4\exp\left(-\frac{\gamma t^{2}}{\mb E X+t}\right)
\]
for any $t\ge0$.
\end{theorem}

We will also need a version of Kingman's subadditive ergodic theorem
(see for example \cite[Theorem~A.3]{Rom15}). We have ``flipped''
the conditions to make it apply under a \emph{superadditivity} condition
rather than a subadditivity condition.
\begin{theorem}\label{thm:kingman}
Let $(X_{m,n})_{0\le m<n}$ be a family of nonnegative random variables,
defined on some common probability space, such that the following
conditions are satisfied.
\begin{compactitem}
\item $X_{0,n}\ge X_{0,m}+X_{m,n}$ for all $m<n$.
\item \smallskip For any $k\ge1$, $(X_{nk,(n+1)k})_{n=0}^{\infty}$ is a sequence
of i.i.d.\ random variables.
\item For any $m\ge1$, we have $(X_{0,k})_{k=0}^{\infty}\overset{d}{=}(X_{m,m+k})_{k=0}^{\infty}$.
\item \smallskip There exists a constant $M>0$ such that $\mb E X_{0,n}\le Mn$ for all
$n$.
\end{compactitem}
Then, $\mb E X_{0,n}/n$ converges to a limit $\gamma$ as $n\to \infty$. Also, $X_{0,n}/n$ converges almost surely (therefore in probability) to $\gamma$.
\end{theorem}

%% file: erdos-szekeres.tex
\let\line\relax
\newcommand*{\line}{{\alpha}}
\newcommand*{\wave}{{\kappa}}
\newcommand*{\stack}{{\nu}}
\let\AA\relax
\newcommand{\AA}{{\blockA\blockA}}
\newcommand{\BB}{{\blockB\blockB}}
\newcommand{\AAA}{{\blockA\blockA\blockA}}
\newcommand{\BBB}{{\blockB\blockB\blockB}}
\newcommand{\clip}[3]{#1_{{#2}:{#3}}}

\section{Lower bounds on Ramsey parameters}\label{sec:ramsey-lower}
%In this section, we prove \cref{thm:ES,thm:ES-3-4}.
%\subsection{Lower bounds}\label{sec:ES}
In this section we will outline our general framework for proving lower bounds on $\ramsey r n$, and show how to use it to prove the lower bounds in \cref{thm:ES} and in the $r=3$ case of \cref{thm:ES-3-4}. The $r=4$ case of \cref{thm:ES-3-4} follows a similar strategy but has some very complicated casework, so we defer it to \cref{appendix: ES}.

The starting point is that patterns sometimes give rise to posets.
For an $r$-matching $M$ and an $r$-pattern $P$, we define the
relation $\preceq_{P}$ on the edges of $M$ by taking $e\preceq f$
if $e[1]<f[1]$ and $e,f$ form pattern $P$ (or if $e=f$). It is
sometimes (but not always!) the case that $\preceq_{P}$ is a partial order.
The interesting property here is \emph{transitivity}: we need to know
whether $e\preceq_{P}f$ and $f\preceq_{P}g$ implies $f\preceq_{P}g$.
For example, in the case $r=2$, the only $r$-patterns are alignments,
crossings and nestings. It is easy to see that the alignment and nesting
pattern always give rise to a poset, but the crossing pattern may
not (it is possible to have edges $e,f,g$ with $e[1]<f[1]<g[1]$
such that $e$ and $f$ form a crossing, and $f$ and $g$ form a
crossing, but $e$ and $g$ form an alignment). However, if we restrict
ourselves to matchings that are alignment-free, then the crossing
pattern \emph{does} give rise to a poset.

Now, if $\preceq_{P}$ is a poset, then Mirsky's theorem (\cref{thm:mirsky}) implies that there is a long chain (which corresponds to a large
$P$-clique) or a large antichain (which corresponds
to a $P$-free sub-matching). 
One can iterate this, in several different ways, to give a simple proof of the optimal bound $\ramsey 2 n\ge n^{1/3}$ (previously
proved with a different method in \cite{DGJR23-r-matching}). 
For example, in a size-$n$
matching $M$, first, we look for an alignment-clique of size $n^{1/3}$
or an alignment-free sub-matching $M'$ of size $n^{2/3}$, and in
the latter case, inside $M'$ we look for a nesting-clique of size
$n^{1/3}$ or a nesting-free sub-matching (therefore a crossing-clique)
of size $n^{1/3}$.

For general $r$, if one is careful about the order of operations,
one can show $\ramsey r n\ge \Omega_r(n^{1/3^{r-1}})$ (as in \cite[Theorem~1.3]{DGR23-later})
by generalising the above proof, applying Mirsky's theorem once for
each of the $3^{r-1}$ collectable $r$-patterns (one also needs some separate
arguments to handle the non-collectable patterns).

The authors of \cite{DGR23-later} did not manage to improve on the above bound, and they may have
been tempted to conjecture that the ``$1/3^{r-1}$'' exponent is best-possible. However, they
found some clues that the situation is more intricate than it may
first appear. For example, they observed that some patterns ``cannot
interact with each other'': if $P$ and $P'$ are the collectable
$r$-patterns represented by $|\AB|\blockA\blockA\blockB\blockB|$ and $|\blockA\blockA\blockB\blockB|\AB|$, and we
have a matching $M$ in which every pair of edges forms pattern $P$
or $P'$, then in fact $M$ must be a $P$-clique or a $P'$-clique.

To leverage this type of observation, instead of using Mirsky's theorem to
process patterns one-by-one, we partition the set of all $r$-patterns
into subsets, and process patterns one subset at a time (defining a poset in terms of the entire subset, instead of a single pattern). We need to choose the subsets in our partition judiciously, such that
the patterns in each subset ``cannot interact much with each other'',
and such that each subset gives rise to a poset (after eliminating
the patterns from previous subsets).

Both these properties are encapsulated in the following lemma, which is the key ingredient for the lower bound in \cref{thm:ES}.

\begin{definition}\label{def:calP-clique}
    For a set of $r$-patterns $\mc P$, say that a matching $M$ is $\mc P$-free if it is $P$-free for all $P\in \mc P$. Say that $M$ is a $\mc P$-clique if every pair of edges forms a pattern in $\mc P$.

    Also, for a matching $M$ and a set of $r$-patterns $\mc P$, define the
relation $\preceq_{\mc P}$ on the edges of $M$ by taking $e\preceq f$
if $e[1]<f[1]$ and $e,f$ form a pattern in $\mc P$ (or if $e=f$).
\end{definition}

\begin{lemma}\label{lem:key-ES}
    Let $b=(r+1)2^{r-2}$. There is a partition of the $r$-patterns into subsets $\mathcal P_1,\dots,\mathcal P_b$ such that the following properties hold for any $i\in \{1,\dots,b\}$.
    \begin{enumerate}
        \item [(A)] For any $(\mc P_1\cup \dots\cup\mc P_{i-1})$-free matching $M$, the relation $\preceq_{\mc P_{i}}$ is a partial order on $M$.
        \item [(B)] For any $\mc P_i$-clique $M$ of size $n$, we have $L(M)\ge n/(r-1)$.% for some $P\in \mc P_i$.
    \end{enumerate}
\end{lemma}

The proofs of the sharper lower bounds in \cref{thm:ES-3-4} are a bit more delicate; they use the same idea, with some extra twists. We will discuss them at the end of this section, after the proof of \cref{lem:key-ES}.

For \cref{lem:key-ES}, our partition into subsets will be defined in terms of \emph{weak patterns}, which we define and investigate next.

\subsection{Weak patterns}\label{subsec:weak} Weak patterns measure the relationship between two edges $e,f$ in a slightly coarser way than ordinary patterns. Specifically, for weak patterns, we are only concerned with the behaviour of pairs of consecutive vertices in $e$ and in $f$.
\begin{definition}
  For an $r$-pattern $P$ and some $1\le i<j\le r$ let $\clip P i j$ be the $(j-i+1)$-pattern formed by $\{e[i],e[i+1],\dots,e[j]\}$ and $\{f[i],f[i+1],\dots,f[j]\}$.
  % Fix an $r$-pattern $P$ with edges $e,f$, and for each $i\in\{1,\dots,r-1\}$ write $e_i$ and $f_i$ for the pairs of vertices $\{e[i],e[i+1]\}$ and $\{f[i],f[i+1]\}$. These can be viewed as 2-uniform (graph) edges on the vertex set of $P$

  Then, define the \emph{weak $r$-pattern} $\phi(P)=(\phi_1(P),\dots,\phi_{r-1}(P)) \in \{\line,\wave,\stack\}^{r-1}$ corresponding to $P$, by taking
  \begin{compactitem}
    \item $\phi_i(P)=\line$ if $\clip{P}{i}{i+1}$ is the alignment 2-pattern,
    \item $\phi_i(P)=\wave$ if $\clip{P}{i}{i+1}$ is the crossing 2-pattern,
    \item $\phi_i(P)=\stack$ if $\clip{P}{i}{i+1}$ is the nesting 2-pattern.
  \end{compactitem}
  In general, a sequence $W=(W_1,\dots,W_{r-1}) \in \{\line,\wave,\stack\}^{r-1}$ is called a \emph{weak $r$-pattern}.
\end{definition}
%For example, the pattern represented by $|\AA\BB|\AB|\BA|$ has corresponding weak pattern $\line\wave\stack$. The pattern represented by $\AA\BB\BB\AB\AA$ (which is not collectable) has corresponding weak pattern $\line\stack\line\line$.
In \cref{table: 3-patterns}, we list all the 3-patterns and their corresponding weak 3-patterns.

\begin{table}[h!]
  \centering
  \setlength{\tabcolsep}{7mm}
  % \resizebox{\textwidth}{7mm}
  % \begin{tabular}{|c|c||c|c| }
  %   \hline Pattern & Weak pattern & Pattern & Weak pattern \\
  %   \hline  
  %     $|\AAA|\BBB|$ & $ \line\line $ & $\AA\BA\BB$ & $\line\line$ \\
  %   \hline 
  %     $|\AA\BB|\BA|$ & $ \line\stack $ & $%P_3=
  %     |\AA\BB|\AB|$ & $ \line\wave $ \\
  %   \hline 
  %     $%P_4=
  %     |\AB|\BB\AA|$ & $ \stack\line $ & $%P_5=
  %     |\AB|\BA|\AB|$ & $ \stack\stack $\\
  %   \hline 
  %     $%P_6=
  %     |\AB|\BA|\BA|$ & $ \stack\wave $ & $%P_7=
  %     |\AB|\AA\BB|$ & $ \wave\line $\\
  %   \hline 
  %     $%P_8=
  %     |\AB|\AB|\BA|$ & $ \wave\stack $ & $%P_9=
  %     |\AB|\AB|\AB|$ & $ \wave\wave $\\
  %   \hline
  % \end{tabular}
  \begin{tabular}{|c|c|}
\hline 
Pattern & Weak pattern\tabularnewline
\hline 
$|\AAA|\BBB|$ & $\line\line$\tabularnewline
\hline 
$\AA\BA\BB$ & $\line\line$\tabularnewline
\hline 
$|\AA\BB|\BA|$ & $\line\stack$\tabularnewline
\hline 
$%P_{3}=
|\AA\BB|\AB|$ & $\line\wave$\tabularnewline
\hline 
$%P_{4}=
|\AB|\BB\AA|$ & $\stack\line$\tabularnewline
\hline 
\end{tabular}\enskip{}%
\begin{tabular}{|c|c|}
\hline 
Pattern & Weak pattern\tabularnewline
\hline 
$%P_{5}=
|\AB|\BA|\AB|$ & $\stack\stack$\tabularnewline
\hline 
$%P_{6}=
|\AB|\BA|\BA|$ & $\stack\wave$\tabularnewline
\hline 
$%P_{7}=
|\AB|\AA\BB|$ & $\wave\line$\tabularnewline
\hline 
$%P_{8}=
|\AB|\AB|\BA|$ & $\wave\stack$\tabularnewline
\hline 
$%P_{9}=
|\AB|\AB|\AB|$ & $\wave\wave$\tabularnewline
\hline 
\end{tabular}
  \caption{All $(3\cdot 2)!/(2\cdot (3!)^2)=10$ different 3-patterns, together with their corresponding weak pattern. Note that there are two different patterns corresponding to the weak pattern $\line\line$ (one collectable, and one not).}
  %represented as where the ``In words'' row reveals the block decomposition if the corresponding 3-pattern is collectable. A pattern with superscript $*$ is non-collectable.}
  \label{table: 3-patterns}
\end{table}

%For two edges $e,f$ in some ordered $r$-matching, we say $e,f$ form an $r$-pattern $P$ if $e$ and $f$ is order-isomorphic to $P$ while $e,f$ form a weak $r$-pattern $W$ if $e,f$ form some $r$-pattern $P$ with $\phi(P)=W$.
\begin{definition}
For a weak pattern $W$, say a pair of edges ``form $W$'' if they form a pattern $P$ with $\phi(P)=W$. An ordered $r$-matching $M$ is said to be a \emph{$W$-clique} if every pair of edges form $W$.
\end{definition}
In the proof of \cref{lem:key-ES}, we will define each of our subsets $\mc P_i$ to be a set of all patterns $P$ such that $\phi(P)$ has a prescribed number of ``$\line$''s, and has its ``$\stack$''s in prescribed positions. To understand why this works, we need quite a thorough study of weak patterns.

First, we collect a few basic properties of weak patterns. A ``generalised alignment'' is a pattern with block representation $\AA\cdots\blockA\BB\cdots\blockB$ (i.e., a collectable pattern with a single block).
\begin{lemma}\label{lemma: properties of collectable patterns}
  Fix an $r$-pattern $P$ with edges $e,f$ (and assume $e[1]<f[1]$).
  \begin{enumerate}
    \item[(i)] For each $i \in [r]$, we have $e[i]<f[i]$ if and only if there are an even number of ``$\stack$''s among $\phi_1(P),\dots,\phi_{i-1}(P)$.
    \item[(ii)] $P$ is collectable if and only if:
    \begin{enumerate}
        
    \item [($*$)]for all $1 \le i< j \le r$ where
    \[\phi(\clip{P}{i}{j})=\phi_i(P)\,\phi_{i+1}(P)\dots\phi_{j-1}(P)=\line\line\dots\line,\]
    $\clip P i j$ is a generalised alignment.
    
    \end{enumerate}
    %for all $1 \le i < j \le r$ with $\phi_i(P)=\phi_{i+1}(P)=\dots=\phi_{j-1}(P)=\line$, we have $e[i]<e[i+1]<\dots<e[j]<f[i]<f[i+1]<\dots<f[j]$. %Equivalently, the sets $\{e[i],e[i+1],\dots,e[j]\}$ and $\{e[i],e[i+1],\dots,e[j]\}$ (interpreted as $(j-i+1)$-uniform edges) form the collectable pattern $\AA\cdots\blockA\BB\cdots\blockB$ or $\BB\cdots\blockB\AA\cdots\blockA$.
    \item[(iii)] For each weak $r$-pattern $W$, there is \emph{exactly one} collectable $r$-pattern $P$ satisfying $\phi(P)=W$. 
    % \item If $P$ is collectable, then the number of blocks in its block partition is equal to the number of $\wave$'s and $\stack$'s plus 1.
  \end{enumerate}
\end{lemma}
%We provide a sketch proof of \cref{lemma: properties of collectable patterns} below, but it is perhaps easiest to convince oneself of the truth of all parts of \cref{lemma: properties of collectable patterns} by drawing out some cases on scratch paper.
\begin{proof}[Proof sketch]
  For (i), it is a simple observation (immediate from the definitions of alignments, crossings and nestings) that the relative order between $e[i]$ and $f[i]$ is different from the relative order between $e[i-1]$ and $f[i-1]$ if and only if $\phi(P)_{i-1}=\stack$. 
  The statement in (i) then follows by a straightforward induction on $i$.
  
  (ii) is a bit more involved. Recall from \cref{definition: splittable} that $P$ is collectable if and only if it is splittable (i.e., it can be partitioned into blocks that have an $\blockA$-run followed by a $\blockB$-run of the same length, or vice versa).
  
  For the ``only if'' direction, it is easy to see that if $P$ is collectable (splittable) then it satisfies $(*)$. Indeed, note that if $P$ is splittable then $\clip P i j$ is splittable for all $1 \le i < j \le r$ (this follows immediately
  from the definition of splittability). In the block representation of $\clip P i j$, each division between consecutive blocks gives rise to a ``$\stack$'' or ``$\wave$'', so if $\phi(\clip{P}{i}{j})=\alpha\dots\alpha$ then $\clip P i j$ must have a single block and is therefore a generalised alignment, as desired.
  
  % Note that the entries of $\phi(P_{i,j})$ are precisely So, if $\phi_i(P)=\dots=\phi_{j-1}(P)=\line$, then $\phi(P_{i,j})=\line\dots\line$, so by the above fact $P_{i,j}$ is a generalised alignment, which is equivalent to $e[i],e[i+1],\dots,e[j],f[i],f[i+1],\dots,f[j]$ form $\AA\cdots\blockA\BB\cdots\blockB$ or $\BB\cdots\blockB\AA\cdots\blockA$.
  For the ``if'' direction, suppose $P$ is not collectable (splittable). We need to prove that $(*)$ is violated. Let $j$ be minimal such that $\clip P 1 j$ is not splittable (note that every 2-pattern is splittable, so we have $j\ge 3$). Let $|\block_1|\dots|\block_k|$ be the block representation of $\clip{P}{1}{j-1}$. 
  If we had $\phi_{j-1}(P)\in\{\wave,\stack\}$, then the block representation of $\clip{P}{1}{j}$ would be $|\block_1|\dots|\block_k|\AB|$ or $|\block_1|\dots|\block_k|\BA|$ (depending on the parity of the number of ``$\stack$''s among $\phi_1(P),\dots,\phi_{j-1}(P)$).
  In either case, $\clip{P}{1}{j}$ would then be splittable (collectable), a contradiction.
  % of $\clip{P}{1}{j-1}$ is of form $|Q|\AB|$ or $|Q|\BA|$ where $Q$ is some collectable $(j-2)$-pattern (possibly empty if $j=2$). A case analysis shows that the block decomposition of $\clip P 1 j$ must be $|Q|\AA\BB|$ or $|Q|\BB\AA|$. In either case, $\clip P 1 j$ is collectable, a contradiction.
  So, we must have $\phi_{j-1}(P)=\line$.
  
  % If we had $\phi_{j-1}(P)=\wave$, we would be able to extend the block representation of the splittable pattern $P_{1,j-1}$ to a block representation of $P_{1,j}$ by simply adding the block $\AB\BA$ or $\BA\AB$ \mk{is this sufficiently obvious?}. This would contradict the non-splittability of $P_{1,j}$. Similarly, if we had $\phi_{j}(P)=\stack$, then we would obtain a contradiction by adding the block $\AB\BA$ or $\BA\AB$ to the block representation of $P_{1,j-1}$. So, we must have $\phi_{j-1}(P)=\line$.

  Now, let $i$ be minimal such that $\phi(\clip{P}{i}{j})=\line\dots\line$. We have $i \le j -1$, since $\phi_{j-1}(P)=\line$.
  In the case $i=1$, we know $\clip{P}{1}{j}=\clip{P}{i}{j}$ cannot be a generalised alignment because (by the choice of $j$) it is not splittable. This means that $P$ violates the condition in $(*)$.  So, it remains to consider the case $i \ge 2$.  By the definition of $i$, we know $\phi_{i-1}(P)\neq \line$, which implies that both $e[i-1]$ and $f[i-1]$ precede $e[i]$ and $f[i]$.
  If $\clip P i j$ were to form a generalised alignment, the block representation of $\clip P 1 j$ would be $|\block_1'|\dots|\block_k'|\AA\cdots\blockA\BB\cdots\blockB|$ or $|\block_1'|\dots|\block_k'|\BB\cdots\blockB\AA\cdots\blockA|$ (depending on the parity of the number of ``$\stack$''s among $\phi_1(P),\dots,\phi_{i-1}(P)$), where $|\block_1'|\dots|\block_k'|$ is the block representation of $\clip P 1 i$.
  In either case, $\clip P 1 j$ would be splittable, a contradiction. So, $\clip P i j$ cannot be a generalised alignment, meaning that $(*)$ is violated.
  This completes the proof for (ii).

  For (iii), fix a weak $r$-pattern $W=(W_1,\dots,W_{r-1})\in\{\line,\wave,\stack\}^{r-1}$.
  % and let $P$ be a collectable $r$-pattern with $\phi(W)=P$. 
  We can use properties (i) and (ii) to construct and force the structure of the collectable $r$-pattern $P$ with $\phi(P)=W$.
  Specifically, let $1 \le i_1<\dots<i_m < r$ be the indices $i$ with $W_i\in\{\stack,\wave\}$.
  Set $i_0=0$ and $i_{m+1}=r$ for convenience.
  For $1\le \ell\le m+1$, let $\block_\ell$ be a sequence of $(i_{\ell}-i_{\ell-1})$ ``$\blockA$''s followed by a sequence of $(i_{\ell}-i_{\ell-1})$ ``$\blockB$''s if the number of ``$\stack$''s among $W_{i_1},W_{i_2},\dots,W_{i_{\ell-1}}$ is even; let $\block_\ell$ be a sequence of $(i_{\ell}-i_{\ell-1})$ ``$\blockB$''s followed by a sequence of $(i_{\ell}-i_{\ell-1})$ ``$\blockA$''s otherwise.
  % For $\ell\in[m+1]$, let $s_\ell$ be the number of ``$\stack$''s among $W_{i_1},W_{i_2},\dots,W_{i_{\ell-1}}$. 
  % If $s_\ell$ is even, let $\block_\ell$ be a sequence of $i_{\ell}-i_{\ell-1}$ ``$\blockA$''s followed by a sequence of $i_{\ell}-i_{\ell-1}$ ``$\blockB$''s. 
  % Otherwise, let $\block_\ell$ be a sequence of $i_{\ell}-i_{\ell-1}$ ``$\blockB$''s followed by a sequence of $i_{\ell}-i_{\ell-1}$ ``$\blockA$''s. 
  Then, it is not hard to deduce from (i) and (ii) that $Q=\block_1\block_2\cdots \block_{m+1}$ is a collectable $r$-pattern with $\phi(Q)=W$, and the block representation of $P$ must be precisely the concatenation $|\block_1|\block_2|\cdots |\block_{m+1}|$. 
  % \mk{is this sufficiently obvious?}
  % For uniqueness, suppose $e,f$ form some collectable $P$ with $\phi(P)=W$.
  % (ii) implies for any $\ell\in[m+1]$, $e[i_{\ell-1}+1],e[i_{\ell-1}+2],\dots,e[i_{\ell}],f[i_{\ell-1}+1],f[i_{\ell-1}+2],\dots,f[i_{\ell}]$ form $P_\ell$.
  % Also, for all $\ell \in [m]$, since $W_{i_\ell} \in \{\stack,\wave\}$, if $s_{\ell}$ is even, then $e[i_\ell],f[i_\ell]<e[i_\ell+1],f[i_\ell+1]$, while if $s_\ell$ is odd, then $e[i_\ell+1],f[i_\ell+1]<e[i_\ell],f[i_\ell]$.
  % Thus, $e,f$ in fact forms $\tilde{P}$, i.e. $P=\tilde{P}$.
  % In other words, $\tilde{P}$ is the unique collectable $r$-pattern with $\chi(\tilde{P})=W$.
\end{proof}
Given \cref{lemma: properties of collectable patterns}(iii), it makes sense to introduce some notation for the unique collectable pattern corresponding to a particular weak pattern.

\begin{definition}
  For each weak $r$-pattern $W$, let $\psi(W)$ be the unique collectable $r$-pattern satisfying $\phi(\psi(W))=W$.
\end{definition}
Now, one advantage of considering weak patterns is that we can get a handle on how they can interact with each other by considering how their constituent alignments, crossings and nestings can interact with each other.
% To describe this, we introduce some more notation.
% \begin{definition}
%   For $r$-patterns $P,Q,R$, we write $PQ\to R$ if there exists an ordered $r$-matching with three edges $e,f,g$ such that $e[1]<f[1]<g[1]$, and $e,f$ form $P$, $f,g$ form $Q$ and $e,g$ form $R$.
  
%   Similarly, for weak $r$-patterns $W,X,Y$, we write $WX\to Y$ if there exists an ordered $r$-matching with three edges $e,f,g$ such that $e[1]<f[1]<g[1]$, and $e,f$ form $W$, $f,g$ form $X$ and $e,g$ form $Y$.
% \end{definition}
%For $r$-patterns $P,Q,R$, if $PQ\to R$, then $\phi(P)\phi(Q)\to\phi(R)$.

\begin{lemma}\label{lemma: composition of two weak patterns}
  Suppose $e,f,g$  are three edges in some ordered $r$-matching such that $e[1]<f[1]<g[1]$. 
  Let $W^{e,f},W^{f,g},W^{e,g}$ be the weak $r$-patterns formed by pairs $\{e,f\},\{f,g\},\{e,g\}$, respectively.
  % For some weak $r$-patterns $W,X,Y \in \{\line,\wave,\stack\}^{r-1}$, suppose that the pair $e,f$ forms $W$, the pair $f,g$ forms $X$, and the pair $e,g$ forms $Y$. 
  % weak $r$-patterns $W,X,Y \in \{\line,\wave,\stack\}^{r-1}$.
  %such that $WX\to Y$.% and such that the ``$\stack$''s occur in the same positions in $W$ as in $X$.
  If the ``$\stack$''s occur in the same positions in $W^{e,f}$ and in $W^{f,g}$, then, for $1 \le i \le r-1$, 
  \begin{enumerate}
    \item[(i)] if $W^{e,f}_i=W^{f,g}_i=\stack$, then $W^{e,g}_i=\stack$;
    \item[(ii)] if $W^{e,f}_i=\line$ or $W^{f,g}_i=\line$, then $W^{e,g}_i=\line$;
    \item[(iii)] if $W^{e,f}_i=W^{f,g}_i=\wave$, then $W^{e,g}_i\in\{\line,\wave\}$.
  \end{enumerate}
  Also, combining (ii) and (iii), we get:
  \begin{enumerate}
    \item[(iv)] if $W^{e,f}_i\ne\stack$ and $W^{f,g}_i\ne \stack$, then $W^{e,g}_i\ne \stack$.
  \end{enumerate}
\end{lemma}
%\cref{lemma: composition of two weak patterns} is easily proved by case analysis (for each of (i--iii)there are only three edges involved, and the only thing that matters is the relative order of their vertices); we omit the details. \mk{is this ok?}
\begin{proof}
  By assumption, the number of ``$\stack$''s among $W^{e,f}_1,W^{e,f}_2,\dots,W^{e,f}_{i-1}$ and among $W^{f,g}_1,W^{f,g}_2,\dots,W^{f,g}_{i-1}$ are the same. We assume that this number is even (the odd case is similar). By \cref{lemma: properties of collectable patterns}(i), we have $e[i]<f[i]<g[i]$.
  
  If $W^{e,f}_i=W^{f,g}_i=\stack$, then $e[i]<f[i]<f[i+1]<e[i+1]$ and $f[i]<g[i]<g[i+1]<f[i+1]$, implying that $e[i]<g[i]<g[i+1]<e[i+1]$.
  This means $\{e[i],e[i+1]\}$ and $\{g[i],g[i+1]\}$ form a nesting, i.e., $W^{e,g}_i=\stack$. This proves (i).
  
  If $W^{e,f}_i, W^{f,g}_i\neq \stack$, we know that $W^{e,f}_i,W^{f,g}_i \in \{\line,\wave\}$.
  This implies
  \begin{align*}
      e[i]<\min(e[i+1],f[i])&<\max(e[i+1],f[i])<f[i+1]\text{ and}\\
     f[i]<\min(f[i+1],g[i])&<\max(f[i+1],g[i])<g[i+1],
  \end{align*}
  which in turn imply
  \[e[i]<\min(e[i+1],g[i])<\max(e[i+1],g[i])<g[i+1].\]
  That is to say, $W^{e,g}_i \in \{\wave, \line\}$, proving (iii).  If we furthermore have $W^{e,f}_i=\line$ or $W^{f,g}_i=\line$, then we know $e[i]<e[i+1]<g[i]<g[i+1]$, i.e., $W^{e,g}_i=\line$, proving (ii).
\end{proof}

Now, the following lemma will be used to handle non-collectable patterns: if we manage to find a large $W$-clique (for some weak pattern $W$), we can very efficiently drop to a $\psi(W)$-clique. A similar fact (with a slightly worse constant) can actually be deduced from the main result of \cite{DGJR23-r-matching}, but here we provide a self-contained proof.
%We write $|M|$ for the number of size (number of edges) of a matching $M$.
\begin{lemma} \label{lemma: large weak clique implies large clique}
  Let $W$ be a weak $r$-pattern, and let $M$ be a $W$-clique of size $n$. Then, $M$ contains a $\psi(W)$-clique of size at least $n/\max(1,\delta(W))$, where $\delta(W)$ is the length of the longest run of  ``$\line$''s in $W$.
\end{lemma}
\begin{proof}
  By \cref{lemma: properties of collectable patterns}(ii), if $\delta(W)=0$ then there is no non-collectable $r$-pattern $P$ with $\phi(P)=W$.
  Thus, $M$ itself is a $\psi(W)$-clique.  So, we may assume $\delta:=\delta(W)\ge 1$.
  
  Let $e_1,e_2,\dots,e_n$ be the edges in $M$, ordered such that $e_1[1]<e_2[1]<\dots<e_n[1]$.
  We claim that for any $s,t$ with $t-s\ge \delta$, the edges $e_s,e_t$ form $\psi(W)$. This suffices to prove the lemma: we can simply take every $\delta$-th edge as our $\psi(M)$-clique.
  
For $1\leq a<b\leq n$ let $P^{a,b}$ be the pattern formed by $e_a,e_b$. Fix $s,t$ with $t-s\ge \delta$. We wish to show that $P^{s,t}$ is collectable (which will imply $P^{s,t}=\psi(W)$ by \cref{lemma: properties of collectable patterns}(iii)). By \cref{lemma: properties of collectable patterns}(ii), it suffices to show that for any $1\le i < j \le r$ with $W_i=W_{i+1}=\dots=W_{j-1}=\line$ (i.e., $\phi(\clip{P^{s,t}}{i}{j})=\line\dots\line$), the $(j-i+1)$-pattern $\clip{P^{s,t}}{i}{j}$ is a generalised alignment.
  
  % is $A\dots AB\dots B$ or $B\dots BA\dots A$.
  Suppose there are an even number of ``$\stack$''s among $W_1,W_2,\dots,W_{i-1}$ (the odd case is similar). 
  For $q\in\{s,s+1,\dots,t-1\}$ and $k\in\{i,i+1,\dots,j-1\}$, we have $e_q[k]<e_q[k+1]<e_{q+1}[k]<e_{q+1}[k+1]$ (since $W_k=\line$).
  This means we have 
  \[
    e_s[j]<e_{s+1}[j-1]<\dots<e_{s+j-i}[i]\le e_t[i],
  \]
  where the last inequality holds because $s+j-i \le s+\delta \le t$ and there are an even number of ``$\stack$''s among $W_1,\dots,W_{i-1}$.
  In other words, $\clip{P^{s,t}}{i}{j}$ is a generalised alignment, as desired.
  %In other words, the vertices $e_i[s],e_i[s+1],\dots,e_i[t],e_j[s],e_j[s+1],\dots,e_j[t]$ form the $(t-s)$-pattern $\AA\cdots\blockA\BB\cdots\blockB$.
  %The lemma follows by taking $M'$ formed by $\{e_{1+m\delta}: 0 \le m \le (n-1)/\delta \}$.
\end{proof}
\begin{remark}
  It is not hard to see that the constant factor $1/\max(1,\delta(M))$ cannot be improved.
  % The choice of $\delta(F)$ is tight for all $F$. by considering patterns like $AABABABABB$, say.
\end{remark}

\subsection{Proof of the key lemma}\label{subsec:key-ES-proof}
Now, we have all the preparations in place to prove \cref{lem:key-ES}. The subsets of patterns $\mc P_i$ will be defined in terms of \emph{signatures}, as follows.
\begin{definition}
  For each weak $r$-pattern $W\in\{\line,\wave,\stack\}^{r-1}$, the \emph{signature} of $W$ is defined to be
  \[
    \sigma(W)
    := \big(\abs{\{i:W_i=\line\}},\{i:W_i=\stack\}\big).
  \]
  That is to say, the signature specifies the number of ``$\line$''s and the positions of the ``$\stack$''s. We define the \emph{weight} of $\sigma(W)$ to be the number of ``$\line$''s in $W$.
  %We also say $\sigma(W)$ is the signature of any $r$-pattern $P$ with $\phi(P)=W$.
\end{definition}
The total number of signatures is %Let $\Sigma$
%=\{\sigma(W): W \text{ is a weak } r\text{-pattern}\}$
%be the set of all possible signatures; we compute
\[
%\abs{\Sigma}=
\sum_{S\subseteq\{1,\dots,r-1\}}(r-\abs{S})=\sum_{i=0}^{r-1}\binom{r-1}{i}(r-i)=(r+1)2^{r-2}=:b.\]
Let $\sigma_1,\dots,\sigma_b$ be an ordering of the signatures in descending weight (breaking ties arbitrarily). Let $\mc P_i$ be the set of all patterns $P$ such that $\sigma(\phi(P))=\sigma_i$.

%We are now able to state the main theorem for the lower bound of $\ramsey r n$, where we use the signatures to group weak $r$-patterns.
% \begin{theorem}\label{theorem: asymmetric general erdos-szekeres}
%   Suppose $(x_{\sigma})_{\sigma\in\Sigma}$ is a collection of positive real numbers.
%   Then, any ordered $r$-matching $M$ of size at least $\prod_{\sigma\in\Sigma}x_{\sigma}$ contains an $W$-clique of size $x_{\sigma(W)}$ for some weak $r$-pattern $W$.
% \end{theorem}
\begin{proof}[Proof of \cref{lem:key-ES}(A)]
With $\mc P_1,\dots,\mc P_b$ as defined above (in terms of signatures $\sigma_1,\dots,\sigma_b$), we wish to show that for each $i$, the relation $\preceq_{\mc P_i}$ is a partial order on any $(\mc P_1\cup \dots\cup \mc P_{i-1})$-free matching $M$.

It suffices to show the transitivity of $\preceq_{\mc{P}_i}$.
Fix a $(\mc P_1\cup \dots\cup \mc P_{i-1})$-free matching $M$ and consider any edges $e,f,g$ with $e[1]<f[1]<g[1]$. Let $P^{e,f},P^{f,g},P^{e,g}$ be the patterns formed by the pairs $\{e,f\}$, $\{f,g\}$ and $\{e,g\}$ respectively, and suppose that $P^{e,f},P^{f,g}\in \mc P_i$. Our objective is to show that $\mc P^{e,g}\in \mc P_i$.
We know that $\phi(P^{e,f})$ and $\phi(P^{f,g})$ have their ``$\stack$''s in the same positions (as recorded by the signature $\sigma_i$), so by \cref{lemma: composition of two weak patterns}(i,iv), the ``$\stack$''s in $\phi(\mc P^{e,g})$ must be in exactly these same positions.

Then, let $w$ be the weight of $\sigma_i$ (i.e., the number of ``$\line$''s in the weak patterns associated with $\sigma_i$). 
% Then, let $w$ be the number of ``$\line$''s in the weak patterns associated with $\sigma_i$. 
We know that $\phi(P^{e,f})$ and $\phi(P^{f,g})$ have exactly $w$ ``$\line$''s. 
Also, since $M$ is $(\mc P_1\cup \dots\cup \mc P_{i-1})$-free (and higher-weight signatures come earlier in our ordering), we know that $\phi(P^{e,f})$ has \emph{at most} $w$ ``$\line$''s. 
However, by \cref{lemma: composition of two weak patterns}(ii), in every position where $\phi(P^{e,f})$ or $\phi(P^{f,g})$ have a ``$\line$'', there is also a ``$\line$'' in $\phi(\mc P^{e,g})$. 
The only way this can happen is if $\phi(P^{e,f})$, $\phi(P^{f,g})$ and $\phi(P^{e,g})$ each have exactly $w$ ``$\line$''s, in exactly the same positions. 
We have proved that $\mc P^{e,g}\in \mc P_i$, as desired.
\end{proof}
\begin{proof}[Proof of \cref{lem:key-ES}(B)]
Suppose $M$ is a $\mc P_i$-clique of size $n$.
We wish to show that $L(M) \ge n/(r-1)$.% for some $P \in \mc{P}_i$.
% Now, we wish to show that for each $i$, every $\mc P_i$-clique $M$ of size $n$ has $L_P(M)\ge n/r$.

Consider any three edges $e,f,g$ with $e[1]<f[1]<g[1]$. Let $P^{e,f},P^{f,g},P^{e,g}$ be the patterns formed by the pairs $\{e,f\}$, $\{f,g\}$ and $\{e,g\}$ respectively, so $P^{e,f},P^{f,g},P^{e,g}\in \mc P_i$. By the considerations in the above proof of \cref{lem:key-ES}(A), each of $\phi(P^{e,f}),\phi(P^{f,g}),\phi(P^{e,g})$ must have their ``$\stack$''s and ``$\line$''s (and therefore their ``$\wave$''s) in exactly the same positions, which means that $e,f,g$ actually form a $W$-clique for some weak pattern $W$. 
But if every three edges form a $W$-clique, then the whole of $M$ must be a $W$-clique for some weak $r$-pattern $W$ with signature $\sigma_i$.
By \cref{lemma: large weak clique implies large clique}, we have $L_{\psi(W)}(M)\ge n/(r-1)$.
\end{proof}

\subsection{Proof of the lower bound in \cref{thm:ES}}
Now, it is easy to derive the lower bound in \cref{thm:ES} by iteratively applying \cref{lem:key-ES} and Mirsky's theorem.
\begin{proof}[Proof of the lower bound in \cref{thm:ES}]
Let $M$ be a size-$n$ matching, and recall the subsets of patterns $\mc P_1,\dots,\mc P_b$ in the proof of \cref{lem:key-ES}. First, combining Mirsky's theorem (\cref{thm:mirsky}) and \cref{lem:key-ES}(A), it is easy to prove by induction that for each $i< b$:
\begin{enumerate}
    \item [$(*)$]Either for some $j\le i$ we can find a $\mc P_j$-clique of size at least $n^{1/b}$, or we can find a $(\mc P_1\cup \dots\cup P_{i})$-free sub-matching of size at least $n^{1-i/b}$.
\end{enumerate}

A $(\mc P_1\cup \dots\cup P_{b-1})$-free matching is nothing more than a $\mc P_b$-clique, so $(*)$ with $i=b-1$ actually implies that we can always find a $\mc P_j$-clique $M'$ of size at least $n^{1/b}$, for some $j\le b$. Then, by \cref{lem:key-ES}(B), we have $L(M)\ge L(M')\ge n^{1/b}/(r-1)$, as desired.
%inside this $P_j$-clique we can find a $P$-clique of size at least $\frac{1}{r-1}n^{1/b}$. This proves the desired bound.
  % Let $M$ be an ordered $r$-matching of size $n$.
  % Applying \cref{theorem: asymmetric general erdos-szekeres} on $M$ with $x_{\sigma}=n^{\frac{1}{|\Sigma|}}$ for all $\sigma\in\Sigma$, we know that for some weak $r$-pattern $W$, $M$ contains an $W$-clique $M'$ of size at least $n^{\frac{1}{|\Sigma|}}$.
  % Denote $P=\psi(W)$ to be the unique collectable $r$-pattern corresponding to $W$.
  % \cref{lemma: large weak clique implies large clique} implies that $M'$ contains a $P$-clique of size at least $\frac{1}{\max(1,\delta(W))}n^{\frac{1}{|\Sigma|}}\ge \frac{1}{r}n^{\frac{1}{(r+1)2^{r-2}}}$.
  % Then, $L(M) \ge \frac{1}{r}\abs{M}^{\frac{1}{(r+1)2^{r-2}}}$ for all $M$, and thus $\ramsey{r}{N}\ge \frac{1}{r}n^{\frac{1}{(r+1)2^{r-2}}}$.
\end{proof}

\subsection{Further improvements}\label{subsec:further-ES}
Weak $r$-patterns are very convenient to work with, due to the fact that we can separately study how their constituent alignments, crossings and matchings interact (with \cref{lemma: composition of two weak patterns}). However, by directly considering how patterns can interact with each other, one can prove stronger bounds (specifically, one can prove an analogue of \cref{lem:key-ES} with a smaller value of $b$). In this subsection, we show how to do this to prove the essentially optimal %(up to constant multiplicative factor)
lower bound $L_3(n)\ge n^{1/7}/2$ featuring in \cref{thm:ES-3-4}. Note that \cref{thm:ES} only gives the non-optimal bound $\ramsey{3}{n}\ge n^{1/8}/2$.

The considerations in this $r=3$ proof can be generalised to larger $r$. %, and for general $r$, one can prove an analogue of \cref{lem:key-ES} for some $b$ of the form $(r-1)2^{r-1}+o_{r\to\infty}(2^r)$, which allows one to prove a general bound of the form $L_r(n)\ge \Omega(n^b)$. Also, 
%Actually, by bifurcating into cases in a more extreme way (that falls slightly outside the general framework of \cref{lem:key-ES})
For example, with some more sophisticated case analysis we can prove the optimal lower bound $L_4(n)\ge \Omega(n^{1/15})$. As the proof is quite technical, we refer the interested reader to
\ifarxiv
\cref{appendix: ES}.
\fi
\ifjournal
the appendix in the arXiv version of this paper.
\fi

\begin{proof}[Proof of the lower bound on $\ramsey{3}{n}$ in \cref{thm:ES-3-4}]
We use the notions of weak patterns and signatures introduced in \cref{subsec:weak,subsec:key-ES-proof}. Referring to \cref{table: 3-patterns}, we see that there are ten different 3-patterns, including one non-collectable pattern (which we call $P^*$; note that $\phi(P^*)=\line\line$).
%Note that there are exactly $(3\cdot 2)!/(2\cdot (3!)^2)=10$ different 3-patterns 
  %In this proof we refer to the ten different 3-patterns in \cref{table: 3-patterns} 
  %(one for each of the nine weak 3-patterns, and a single non-collectable pattern $P^*$, which has block representation $\AA\BA\BB$ and $\phi(P^*)=\line\line$).
  Define
  \[\mc P_1=\{\psi(\line \line),P^*\},\quad \mc P_2=\{\psi(\line \wave),\psi(\wave \line)\},\quad\mc P_3=\{\psi(\line \stack),\psi(\stack \line)\}.%,\;  \mc P_4=\{\psi(\stack\stack),\psi(\stack\wave),\psi(\wave\stack),\psi(\wave\wave)\};
  \]
We observe that these subsets can be used to define posets (we write $\preceq_i$ instead of $\preceq_{\mathcal P_i}$):
  % some simple case-checking (which can be reduced somewhat by appealing to the lemmas in \cref{subsec:weak}), it is easy to verify the following facts. \mk{is this ok?}
  \begin{compactitem}
      \item By the proof of \cref{lem:key-ES}(A), for any matching $M$, the relation $\preceq_1$ is a partial order. (Note that $\mc P_1$ corresponds to the signature $(2,\emptyset)$).
      \item Similarly, by the proof of \cref{lem:key-ES}(A), for any $\mc P_1$-free matching $M$, the relation $\preceq_2$ is a partial order. (Note that $\mc P_2$ corresponds to the signature $(1,\emptyset)$).
      \item One can also see that the relation $\preceq_3$ is always a partial order (despite $\mc P_3$ containing the patterns for two different signatures $(1,\{1\})$ and $(1,\{2\})$). 
      Observe that the two patterns $\psi(\line\stack),\psi(\stack\line)\in \mc P_3$ have block representations $|\AA\BB|\BA|$ and $|\AB|\BB\AA|$; informally speaking, these patterns are ``generalised nestings'', where all vertices of one edge are fully contained between two consecutive vertices of the other edge, and it is not hard to see that $\preceq_3$ is therefore a partial order. In detail: note that it suffices to show transitivity. 
        Suppose $e,f,g$ are three edges in a matching $M$ with $e[1]<f[1]<g[1]$ such that $e,f$ form pattern $\psi(\line\stack)$ and $f,g$ form pattern $\psi(\line\stack)$ or $\psi(\stack\line)$.
        Then, $e[2]<f[1]<f[3]<e[3]$ (as $e,f$ form pattern $\psi(\line\stack)$) and $f[1]<g[1]<g[3]<f[3]$ (as $f,g$ form pattern $\psi(\line\stack)$ or $\psi(\stack\line)$).
        So $e[2]<g[1]<g[3]<e[3]$, i.e., $e,g$ form pattern $\psi(\line\stack)$.
        Similarly, if $e,f$ form pattern $\psi(\stack\line)$ and $f,g$ form pattern $\psi(\line\stack)$ or $\psi(\stack\line)$, then $e,g$ form pattern $\psi(\stack\line)$.
      \item By the proof of \cref{lem:key-ES}(A), for any $(\mc P_1\cup \mc P_2\cup \mc P_3)$-free matching $M$, the relation $\preceq_P$ is a poset when $P$ is any of the four 3-partite patterns $\psi(\stack\stack),\psi(\stack\wave),\psi(\wave\stack),\psi(\wave\wave)$. (These correspond to the signatures $(0,S)$, for $S\subseteq \{1,2\}$).
  \end{compactitem}
  Proceeding in the same way as the proof of the lower bound in \cref{thm:ES}, we can therefore find a $\mc P_i$-clique of size at least $n^{1/7}$, for some $i\in \{1,2,3\}$, or we can find a $P$-clique of size at least $n^{1/7}$ for $P$ being one of the four 3-partite patterns $\psi(\stack\stack),\psi(\stack\wave),\psi(\wave\stack),\psi(\wave\wave)$.

  By the proof of \cref{lem:key-ES}(B), if we found a $\mc P_1$-clique or a $\mc P_2$-clique $M'$ of size at least $n^{1/7}$, then $L(M')\ge n^{1/7}/2$.
  
  Finally, it suffices to consider the case where we found a $\mc P_3$-clique $M'$ with at least $n^{1/7}$ edges. 
  Let $e_1,\dots,e_m$ be the edges in $M'$ (with $m \ge n^{1/7}$, and $e_1[1]<\dots<e_m[1]$). 
  As discussed in the third bullet point above, we see that for any $1 \le i < j < k \le m$, the edges $e_i,e_k$ always form the same 3-pattern as the edges $e_i,e_j$. 
  % By some direct case-checking, we see that for any $1 \le i < j < k \le m$, the edges $e_i,e_k$ always form the same 3-pattern as the edges $e_i,e_j$ do. 
  So, every index $i$ is of one of two types: we say it is of ``type $\psi(\line\stack)$'' if $e_i,e_j$ form $\psi(\line\stack)$ for each $i<j$, and we say it is of ``type $\psi(\stack\line)$'' if $e_i,e_j$ form $\psi(\stack\line)$ for each $i<j$. By the pigeonhole principle, at least half of the indices have the same type, and the corresponding edges give us a $\psi(\line\stack)$-clique or a $\psi(\stack\line)$-clique of size $n^{1/7}/2$.
\end{proof}

\section{Upper bounds on Ramsey parameters}\label{sec:ramsey-upper}
In this section we prove the upper bound in \cref{thm:ES} (which also implies the upper bounds in \cref{thm:ES-3-4}).
We consider the following notion of ``blow-up'' introduced by by Dudek, Grytczuk and Ruci\'nski~\cite{DGR23-later}.
\begin{definition}
  Suppose $M,M'$ are ordered $r$-matchings, and let $t$ be the size of $M'$.
  The \emph{$M'$-blow-up} of $M$, denoted by $M[M']$, is an ordered $r$-matching obtained from $M$ as follows.
  We replace every vertex $i\in V(M)$ by an ordered set $U_i$ of $t$ contiguous vertices. %For each $1\le i<j\le V(M)$, the vertices in $U_i$ occurs before the vertices in $U_j$;
  Then, for each edge $e\in E(M)$ we place on the vertex set $U_{e[1]}\cup\dots\cup U_{e[r]}$ a copy $M'_e$ of the matching $M'$.
  %on the vertex set $\bigcup_{j=1}^r U_{e[j]}$.%, in such a way that $M'_e$ is vertex-disjoint from $M'_{e'}$ for every distinct $e,e'\in E(M)$.
  %we replace each edge $e\in E(M)$ by a copy $M'_e$ of $M'$ on the vertex set $\bigcup_{j=1}^r U_{e[j]}$.%, in such a way that $M'_e$ is vertex-disjoint from $M'_{e'}$ for every distinct $e,e'\in E(M)$.
  Note that there are two kinds of pairs of edges $(f_1,f_2)$ in $M[M']$:
  \begin{compactitem}
      \item If $f_1,f_2$ both lie in the same $M_e'$ (for some $e\in E(M)$), then we say $f_1$ and $f_2$ comprise an~\emph{$M'$-pair}.
      \item If $f_1,f_2$ lie in different $M_e'$, then we say $f_1$ and $f_2$ comprise  an \emph{$M$-pair}.
      \end{compactitem}
\end{definition}

The above definition is useful typically when $M'$ is $r$-partite.
In this case, one can check that if two edges $f_1,f_2$ comprise an $M$-pair (say $f_1$ lie in $M_{e_1}'$ and $f_2$ lie in $M_{e_2}'$), then $f_1,f_2$ form the same $r$-pattern as $e_1,e_2$ do (in $M$).
Therefore, if $M$ contains no $r$-partite $r$-pattern and $M'$ is $r$-partite, then $L_P(M[M'])=L_P(M')$ if $P$ is $r$-partite and $L_P(M[M'])=L_P(M')$ otherwise, and thus $L(M[M'])=\max(L(M),L(M'))$.
% It is not hard to see that for any $r$-matchings $M,M'$ and any $r$-pattern $P$, the size of the largest $P$-clique in $M[M']$ is the product of the sizes of the largest $P$-cliques in $M$ and $M'$, i.e., $L_P(M[M'])=L_P(M)\cdot L_P(M')$.
% In particular, if $M$ contains no $r$-partite $r$-patterns and $M'$ contains only $r$-partite $r$-patterns (in other words, $M'$ is $r$-partite), then $L(M[M'])=\max(L(M),L(M'))$. 
This fact will be used in the proof of the upper bound in \cref{thm:ES} momentarily.
% As an example, suppose $M$ is a $P$-clique (i.e., all pairs of edges form pattern $P$) and $M'$ is a $P'$-clique (i.e., all pairs of edges form pattern $P'$). Then the edges in $M[M']$ form only patterns $P$ or $P'$. Specifically, the $M$-pairs form pattern $P$ and the $M'$-pairs form pattern $P'$.\mk{I rephrased this as an example, is that what the point of this was?}

For $r$-partite $r$-patterns, Dudek, Grytczuk and Ruci\'nski~\cite{DGR23-later} used a blow-up construction to prove the following tight upper bound (which is actually equivalent to the result of Burkill and Mirsky~\cite{BM73} and Kalmanson~\cite{Kal73} mentioned in the introduction, on $(r-1)$-tuples of permutations).
\begin{theorem}[Dudek, Grytczuk and Ruci\'nski~\cite{DGR23-later}]\label{theorem: partite matching erdos-szekeres construction}
  For $r \ge 2$, and $n \ge 1$, there exists an $r$-partite ordered $r$-matching $M$ of size $n^{2^{r-1}}$ such that for all ($r$-partite) $r$-patterns $P$, the largest $P$-clique is of size $n$.
\end{theorem}

We are now ready to prove that $\ramsey{r}{n} \le \lceil n^{1/(2^r-1)}\rceil$.
\begin{proof}[Proof of the upper bound in \cref{thm:ES}]
Fix $r \ge 2$ and $n \ge 1$. We are going to construct an ordered $r$-matching $M$ of size $n^{2^{r}-1}$ such that any pair of edges form a collectable $r$-pattern, and such that the largest clique has size $n$.
  We will do this by induction on $r$.
  
  The base case $r=2$ already appears in \cite{DGR23-later,DGJR23-r-matching}. Indeed, consider $M:=M_1[M_2]$, where $M_1$ is an alignment-clique of size $n$ and $M_2$ is the 2-partite ordered 2-matching from \cref{theorem: partite matching erdos-szekeres construction}. Then, $M$ has size $n^3$. A pair of distinct edges $f_1,f_2\in E(M)$ form an alignment if $(f_1,f_2)$ is an $M_1$-pair, and they form a nesting or a crossing if $(f_1,f_2)$ is an $M_2$-pair. It is easy to check that 
  %As the copies of $M_2$ in $M$ are vertex-disjoint,
  there is no clique of size larger than $n$.
  
  For the inductive step, suppose we have constructed an ordered $(r-1)$-matching $M_1$ of size $n^{2^{r-1}-1}$ in which every pair of edges forms a collectable $(r-1)$-pattern, such that the largest clique has size $n$.
  For each edge $e \in E(M_1)$, we add a new vertex $v_e$ just to the right of $e[r-1]$, and extend $e$ to an edge $e'$ of uniformity $r$ by adding $v_e$ to $e$. Let the resulting ordered $r$-matching be $M_2$.

  Now, we claim that $M_2$ is free of $r$-partite $r$-patterns, and its largest clique has size $n$. 
  To see this, note that if $e,f\in E(M_1)$ form a pattern with block representation $|\block_1|\dots|\block_k|$, then $e',f'\in E(M_2)$ form the pattern whose block representation is $|\block_1|\dots|\block_{k-1}|\block_k'|$, where $\block_k'$ is obtained by extending the ``$\blockA$-run'' and the ``$\blockB$-run'' in $\block_k$ by one. (For example, if $e,f$ form a pattern with block representation $|\AA\BB|\AB|\BA|$, then $e',f'$ form a pattern with block representation $|\AA\BB|\AB|\BB\AA|$). This latter pattern is never $r$-partite, because it has at most $r-1$ blocks.

  %Now, let $P$ be any collectable $(r-1)$-pattern, and write $|\block_1|\dots|\block_k|$ for its block representation. So, $\block_k$ consists of a run of ``$\blockA$''s and a run of ``$\blockB$''s of the same length (in some order). Let $\block_k'$ be obtained from $\block_k$ by extending each of these runs by one symbol (e.g. if $\block_k=\BB\AA$ then $\block_k'=\BB\blockB\AA\blockA$). Let $f(P)$ be the pattern with block representation $|\block_1|\dots|\block_{k-1}|\block_k'|$.
  % Define $f(P)=J_1\cup\dots\cup J_{k}'$ where 
  % \begin{equation} \nonumber
  %   J_k'= \left\{
  %   \begin{array}{ccc}
  %      \underbrace{\AA\dots\blockA}_{(t+1)\text{ times } \blockA}\underbrace{\BB\dots\blockB}_{(t+1)\text{ times } \blockB}  & & \text{ if } J_k=\underbrace{\AA\dots\blockA}_{t\text{ times } \blockA}\underbrace{\BB\dots\blockB}_{t\text{ times } \blockB},\\
  %       \underbrace{\BB\dots\blockB}_{(t+1)\text{ times } \blockB}\underbrace{\AA\dots\blockA}_{(t+1)\text{ times } \blockA} & & \text{ if } J_k= \underbrace{\BB\dots\blockB}_{t\text{ times } \blockB}\underbrace{\AA\dots\blockA}_{t\text{ times } \blockA}.
  %   \end{array}
  %   \right.
  % \end{equation}
  %For example, $f(|\AA\BB|\AB|\AB|)=|\AA\BB|\AB|\AA\BB|$.
  %Clearly, $f(\cdot)$ is an injection from collectable $(r-1)$-patterns to collectable $r$-patterns that are not $r$-partite.
  
  %Then, whenever distinct edges $e_1,e_2 \in E(M_1)$ form some $(r-1)$-pattern $P$ (which is collectable by induction), the edges $e_1',e_2'$ form the $r$-pattern $f(P)$.
  %Thus, $M_2$ is free of $r$-partite $r$-patterns and the largest clique in $M_2$ has size $n$.
  
  Now, let $M_3$ be the $r$-partite ordered $r$-matching given by \cref{theorem: partite matching erdos-szekeres construction} (with size $n^{2^{r-1}}$, and whose largest clique has size $n$). Every pair of edges in $M_3$ form an $r$-partite $r$-pattern.
  Then, we take $M$ to be the blow-up $M_2[M_3]$, which has size $n^{2^{r-1}-1}\cdot n^{2^{r-1}}=n^{2^r-1}$. For any distinct $f_1,f_2\in E(M)$, the edges $f_1,f_2$ form an $r$-partite $r$-pattern if and only if $(f_1,f_2)$ is an $M_3$-pair.
  This means a clique in $M$ comes either from $M_2$ or from some copy of $M_3$. So, the largest clique in $M$ has size $n$, proving the induction step.
\end{proof}

%% file: random.tex
\section{Variants on the longest increasing subsequence problem}\label{sec:BW-variants}
The famous Ulam--Hammersley problem (see \cite{Rom15} for a book-length treatment)
asks for the expected length of the longest monotone subsequence in
a random set of $n$ points in a box $[0,1]^{2}$. This was generalised
to higher dimensions by Steele~\cite{Ste77} (and studied further by Bollob\'as
and Winkler~\cite{BW88}). In order to prove \cref{thm:limit-exists}, we will need a further generalisation.
\begin{definition}\label{def:BW-variant}
Fix a partition $\mathcal{A}$ of $\{1,\dots,r\}$ into disjoint parts
$I_{1},\dots,I_{\ell}$. Then, for any vectors $(x_{1},\dots,x_{r}),(y_{1},\dots,y_{r})\in\mb R^{r}$,
write $(x_{1},\dots,x_{r})\preceq_{\mc A}(y_{1},\dots,y_{r})$ if $\max\left\{ x_{i}:i\in I_{j}\right\} \le\min\left\{ y_{i}:i\in I_{j}\right\} $
for all $j\in\{1,\dots,\ell\}$. Note that $\preceq_{\mc A}$ is a partial order on $\mb R^{r}$.

For a set of points $\mathcal{S}\subseteq\mb R^{d}$, write $L_{\mathcal{A}}(\mathcal{S})$
for the longest chain in $\mathcal{S}$ with respect to $\preceq_{\mc A}$.
\end{definition}

We will be interested in $L_{\mathcal{A}}(\mathcal{S}_{n})$, for
a set $\mathcal{S}_{n}$ of $n$ independent uniformly random points
in $[0,1]^{r}$.

%Note that it is ``easiest to be a chain'' in the case where $\mathcal{A}$
%is the partition of $\{1,\dots,r\}$ into $r$ singleton sets. This
%is precisely the case studied by Bollob\'as and Winkler.

\begin{theorem}\label{lem:poissonised}
Fix a partition $\mathcal{A}$ of $\{1,\dots,r\}$. For $m\in\mb N$,
let $\mathcal{T}_{m}\subseteq[0,m]^{r}$ be a set of points obtained
by a Poisson process of rate 1 in $[0,m]^{r}$, and let $L=L_{\mc A}(\mathcal{T}_{m})$. Then, as $m\to\infty$,
\begin{enumerate}
\item ${\displaystyle {\displaystyle \frac{\mb E L}{m}\to a_{\mathcal{A}}}}$,
and
\item \vspace{2pt}${\displaystyle \frac{L}{m} \overset{p}{\to}a_{\mathcal{A}}}$,
\end{enumerate}
for some $a_{\mathcal{A}}>0$ only depending on $\mathcal{A}$. (By
symmetry, in fact $a_{\mathcal{A}}$ only depends on the multiset
of sizes of the parts in $\mathcal{A}$).
\end{theorem}

\begin{proof}
Let $\mathcal{T}\subseteq\mb R^{r}$ be a Poisson process with rate
1 in $\mb R^{r}$. Let $L_{m,n}=L(\mathcal{T}\cap[m,n]^{r})$ be the
longest chain in $\mathcal{T}\cap[m,n]^{r}$ with respect to $\mathcal{P}(\mc A)$.
Noting that $\mathcal{T}\cap[0,m]^{r}\overset{d}{=}\mathcal{T}_{m}$, it suffices to prove that the conditions of Kingman's subadditive ergodic theorem (\cref{thm:kingman}) are satisfied, as follows.
\begin{compactitem}
\item For $0\le m\le n$ we have $L_{0,n}\ge L_{0,m}+L_{m,n}$, because
if we have a chain in $\mathcal{T}\cap[0,m]^{r}$ and a chain in $\mathcal{T}\cap[m,n]^{r}$,
their union is always a chain in $\mathcal{T}\cap[0,n]^{r}$.
\item For any $k\ge1$, $(L_{nk,(n+1)k})_{n=0}^{\infty}$ is a sequence
of i.i.d.\ random variables.
\item For any $m\ge1$, we have $(L_{0,k})_{k=0}^{\infty}\overset{d}{=}(L_{m,m+k})_{k=0}^{\infty}$.
\item There exists a constant $M > 0$ such that $\mathbb{E}L_{0,m} \le Mm$ for all $m$.
% To see this, we first note that conditioned on $\abs{\mc{T}_m}=N$, the distribution of $\mc{T}_m$ is equivalent to $\mc{S}_N$, the distribution by picking $p_1,\dots,p_N$ points in $[0,m]^r$ independently and uniformly random.
% Then, 
To see this, we first note that if $\mathcal{S}_N$ is a set of $N$ independent
uniformly random points in $[0,m]^{r}$, then 
we have $\mb E L_{\mathcal{A}}(\mathcal{S}_N)\le\mb E L_{\mathcal{BW}}(\mathcal{S}_N)$,
where $\mathcal{BW}$ is the partition of $\{1,\dots,r\}$ into $r$
singleton sets. %\zj{Maybe the wrong direction}
Bollob\'as and Winkler~\cite{BW88} proved that $\mb E L_{\mathcal{BW}}(\mathcal{S}_N)\le(1+o(1))eN^{1/r}$,
so
\begin{align*}
\mb E L_{0,m}&=\mb E L_{\mathcal{A}}(\mathcal{T}_{m})=\mb E\left[\mb E\left[\vphantom{\sum}L_{\mathcal{A}}(\mathcal{T}_{m})\,\middle|\,|\mathcal{T}_{m}|\right]\right]=\mb E\left[\mb E\left[\vphantom{\sum}L_{\mathcal{A}}(\mathcal{S}_{|\mathcal{T}_{m}|})\,\middle|\,|\mathcal{T}_{m}|\right]\right]\\
&\le\mb E\left[\vphantom{\sum}(1+o(1))e\,|\mathcal{T}_{m}|^{1/r}\right]=(1+o(1))em,
\end{align*}
where the last equation holds as $\abs{\mc{T}_m}$ has distribution $\operatorname{Poisson}(m^r)$.%\mk{wait so this was all ok in the end? (I think BW consider a poisson process in their paper, so i think we can cite the poisson result)}
\end{compactitem}
\end{proof}
\begin{theorem}\label{thm:BW-variant}
Fix a partition $\mathcal{A}$ of $\{1,\dots,r\}$. Consider a set
$\mathcal{S}_{n}$ of $n$ independent uniformly random points in
$[0,1]^{r}$, and let $L=L_{\mathcal{A}}(\mathcal{S}_{n})$. Let $a_{\mathcal{A}}$ be as in \cref{lem:poissonised}. Then, the following hold.
\begin{enumerate}
\item ${\displaystyle \Pr[|L-\mb E L|\ge t]\le4\exp\left(-\frac{\gamma t^{2}}{\mb E L+t}\right)}$
for some universal constant $\gamma>0$.
\item ${\displaystyle \frac{L}{n^{1/r}}\overset{p}{\to}a_{\mathcal{A}}}$.
\item ${\displaystyle \frac{\mb E L}{n^{1/r}}\to a_{\mathcal{A}}}$.
\end{enumerate}
\end{theorem}

\begin{proof}
First, (1) follows directly from Talagrand's inequality (\cref{thm:talagrand}):
\begin{compactitem}
\item If we change a single point of $\mathcal{S}_n$, we change $L$ by at
most 1.
\item Whenever $L\ge r$ then there is a set of $r$ points which
certifies that $L\ge r$.
\end{compactitem}
Then, for (2), recall the random set $\mathcal{T}_{m}$ from \cref{lem:poissonised}.
Note that $|\mathcal{T}_{m}|$ has a $\operatorname{Poisson}(m^{r})$
distribution, so by Chebyshev's inequality we can choose $m_{1},m_{2}\in\mb N$,
both of the form $(1+o(1))n^{1/r}$, such that $|\mathcal{T}_{m_{1}}|\le n$
whp and $n\le|\mathcal{T}_{m_{2}}|$ whp. 

Note that if we condition on $|\mathcal{T}_{m}|$, then $\mathcal{T}_{m}$
is conditionally a set of that many independent uniformly random points
in $[0,m]^{r}$ (which is equivalent to taking random points in $[0,1]^{r}$
and rescaling them by a factor of $m$). So, there is a coupling of
$\mathcal{S}_{n},\mathcal{T}_{m_{1}},\mathcal{T}_{m_{2}}$ for which $L(\mathcal{T}_{m_{1}})\le L\le L(\mathcal{T}_{m_{2}})$ whp.
Then, (2) follows from \cref{lem:poissonised}(2).

Finally, (3) follows from (1) and (2).
\end{proof}
\begin{proposition}\label{prop:BW-limit-information}
Fix a partition $\mathcal{A}$ of $\{1,\dots,r\}$, and define the constant $a_{\mc A}$ as in \cref{lem:poissonised}. %Then, the following hold.
\begin{enumerate}
\item If $\mc A$ has a single part of size $r$, then we have ${\displaystyle a_{\mc A}=\frac{r}{\Gamma(1/r)}}$
\item If $\mc A$ has $r$ parts of size $1$, then we have
${\displaystyle a_{\mc A}>\frac{r^{2}}{(r!)^{1/r}\Gamma(1/r)}}$.
\end{enumerate}
\end{proposition}

\begin{proof}
For (1), we use the algorithmic approach of Justicz, Scheinerman
and Winkler~\cite{JSW90}: we can build a longest
chain $\vec{x}^{(1)},\vec{x}^{(2)},\dots,\vec{x}^{(L)}$ by first
taking the point $\vec{x}^{(1)}=(x_{1}^{(1)},\dots,x_{r}^{(1)})\in\mathcal{T}_{m}$
with the smallest value of $\max_{i}x_{i}^{(1)}$, then throwing out
all other points $\vec{y}\in\mathcal{T}_{m}$ which do not satisfy
$\vec{x}\preceq\vec{y}$, and taking the point $\vec{x}^{(2)}=(x_{1}^{(2)},\dots,x_{r}^{(2)})$
with the next smallest value of $\max_{i}x_{i}^{(2)}$, and so on.

If we let $\mathcal{T}$ be the set of points corresponding to a Poisson
process of rate 1 in the semi-infinite box $[0,\infty)^{r}$ (instead
of a finite box $[0,m)^{r}$) then we obtain an infinite sequence
of points $(\vec{x}^{(k)})_{k=1}^{\infty}$ as above. Then, $(\vec{x}^{(k)})_{k=1}^{\infty}$ has the following two properties. First, for every $j\geq 1$ the sequence of points $\vec{x}^{(1)},\vec{x}^{(2)},\dots,\vec{x}^{(j)}$ is a longest
chain of points in $\mathcal{T} \cap [0,\max_i x_i^{(j)}]^r$, and second, the increments $\max_{i}x_{i}^{(k+1)}-\max_{i}x_{i}^{(k)}$
are independent and identically distributed (let $Z$ be a random variable with this common distribution).
We have $\Pr[Z> z]=\Pr[\mathcal{T}\cap[0,z]^{r}=\emptyset]=\exp(-z^{r})$
so the density of $Z$ is
\[
\frac{d}{dz}(1-\exp(-z^{r}))=rz^{r-1}\exp(-z^{r})
\]
and
\[
\mb E Z=\int_{0}^{\infty}z\cdot rz^{r-1}\exp(-z^{r})\,dz=\frac{\Gamma(1/r)}{r}.
\]
So, for any $\varepsilon>0$, taking $L_{1}\le(r/\Gamma(1/r)-\varepsilon)m$
and $L_{2}\ge(r/\Gamma(1/r)+\varepsilon)m$, we have $\max_{i}x_{i}^{(L_{1})}\le m\le\max_{i}x_{i}^{(L_{2})}$
whp, by the law of large numbers. Since $\varepsilon$ was arbitrary,
it follows that whp $L=(r/\Gamma(1/r)-o(1))m$; that is, $a=r/\Gamma(1/r)$.

For (2), we recall the algorithmic lower bound of Bollob\'as and
Winkler~\cite{BW88}. They build a chain $\vec{x}^{(1)},\vec{x}^{(2)},\dots,\vec{x}^{(Q)}$
by first taking the point $\vec{x}^{(1)}=(x_{1}^{(1)},\dots,x_{r}^{(1)})\in\mathcal{T}_{m}$
such that $\sum_{i}x_{i}^{(1)}$ is minimal, then throwing out all
other points $\vec{y}\in\mathcal{T}_{m}$ which do not satisfy $\vec{x}\preceq\vec{y}$,
and taking the point $\vec{x}^{(2)}=(x_{1}^{(2)},\dots,x_{r}^{(2)})$
with the next smallest value of $\sum_{i}x_{i}^{(2)}$, and so on.
It is proved in \cite{BW88}, by a very similar method as above, that this
algorithm produces a chain of length $r^{2}/(r!)^{1/r}\Gamma(1/r)$.

We then just need to observe that this algorithm is not optimal: for
example, instead of finding our chain point-by-point, we can build
our chain two points at a time, at step $k$ choosing two points $\vec{x}^{(2k-1)},\vec{x}^{(2k)}$
such that $\vec{x}^{(2k-1)}\preceq\vec{x}^{(2k)}$ and such that $\max\left(\sum_{i}x_{i}^{(2k-1)},\sum_{i}x_{i}^{(2k)}\right)$
is minimized. It is not hard to see that the expected ``distance
travelled'' in one such step is strictly less than two times the
distance travelled in a step of the Bollob\'as--Winkler algorithm.
\end{proof}

\section{Cliques in random ordered matchings}\label{sec:random}
In this section we prove \cref{thm:limit-exists,prop:limit-information}, using the results in \cref{sec:BW-variants}. To give a brief overview, the key insight is that every $P$-clique in an ordered matching $M$ can be interpreted as a chain in a certain point set (defined in terms of a partition of the vertices of $M$), with respect to a certain partial order $\preceq_{\mc A}$ (as in \cref{def:BW-variant}). Taking advantage of the strong concentration in \cref{thm:BW-variant}(1), we can take the union bound over \emph{all} appropriate vertex partitions of a random matching $M$, giving us control over the longest chains in the corresponding point sets.
\begin{proof}
[Proof of \cref{thm:limit-exists}] Suppose $P$ has block partition given by $J_{1}\cup\dots\cup J_{\ell}=\{1,\dots,2r\}$.
Let $\mathcal{A}$ be the partition of $\{1,\dots,r\}$ with parts
$I_{i}=\{j:2j\in J_{i}\}$ for $1\le i\le \ell$.
%(i.e., we ``compress'' our partition by a factor of two). 
Note that one can uniquely recover $I_{1},\dots,I_{\ell}$
from $J_{1},\dots,J_{\ell}$; we have $J_i=\bigcup_{j\in I_i}\{2j-1,2j\}$ for all $i\le \ell$.
% the blocks $J_{i}$ always have
% even size.

Observe that for every $P$-clique $C$ in $M$, there is some partition
$Q_{1}\cup\dots\cup Q_{\ell}$ of $V(M)=\{1,\dots,rn\}$ into contiguous
intervals, such that for every edge $e\in C$ and every $i\in\{1,\dots,\ell\}$,
we have $|e\cap Q_{i}|=|I_{i}|$. Say that $C$ is \emph{consistent}
with $Q_{1},\dots,Q_{\ell}$ if this is the case.

Fix a partition $\mathcal{B}$ of $\{1,\dots,rn\}$ into contiguous
intervals $Q_{1},\dots,Q_{\ell}$; we will only consider cliques consistent
with $\mathcal{B}$ (at the end we will take a union bound over all
partitions).

Let $M_{\mathcal{B}}$ be the sub-matching of $M$ containing the
edges of $M$ which are consistent with $\mathcal{B}$. Condition
on some outcome $N$ for the number of edges in $M_{\mathcal{B}}$.

We will realise the conditional distribution of $M_{\mathcal{B}}$
via an alternative construction, as follows.
\begin{compactenum}
\item Let $V=[0,\ell]\subseteq \mb R$ and for each $i\in\{1,\dots,\ell\}$, let $V_{i}=[i-1,i]$,
so $V_{1}\cup\dots\cup V_{\ell}=[0,\ell]$.
\item For each $j\in\{1,\dots,r\}$, let $i(j)$ be such that $j\in I_{i(j)}$.
\item Let $\mathcal{R}$ be a set of $N$ independent uniformly random points
in $\prod_{j=1}^{r}V_{i(j)}$.
\item For each $\vec{v}\in\mathcal{R}$, let $e(\vec{v})$ be the set of
coordinates of $\vec{v}$ (so $|e(\vec{v})\cap V_{i}|=|I_{i}|$ for
each $i$).
\item Let $M^{*}$ be the hypergraph on the vertex set $\bigcup_{\vec{v}\in\mathcal{R}}e(\vec{v})$
whose edges are the sets $e(\vec{v})$ for $\vec{v}\in\mathcal{R}$.
\end{compactenum}
Note that with probability 1, $M^{*}$ is a matching, and the distribution
of (the order-isomorphism class of) $M^{*}$ is the same as the conditional
distribution of (the order-isomorphism class of) $M_{\mathcal{B}}$,
by symmetry.

Our next goal is to realise the distribution of our random point set $\mathcal{R}\subseteq \prod_{j=1}^{r}V_{i(j)}$ in terms of a set $\mathcal{S}_{N}$ of $N$ independent uniformly random points in $[0,1]^{r}$, in such a way that the size of the largest $P$-clique $L_{P}(M^{*})$ in $M^{*}$ corresponds precisely to the length of the longest chain $L_{\mc A}(\mathcal{S}_{N})$ in $\mathcal{S}_{N}$ with respect to the partial order $\preceq{P}_{\mc A}$ defined in \cref{def:BW-variant}. 

Recall that in the block representation (\cref{definition: splittable}) of $P$, the incident vertices of one edge are assigned with label ``$\blockA$'' while those of the other edge are assigned with label ``$\blockB$''.
Now, consider an isometry $\phi:[0,1]^{r}\to\prod_{j=1}^{r}V_{i(j)}$
defined as follows.
\begin{compactenum}
% \item Assign the label ``$\blockA$'' to one of the edges in $P$, and assign the label ``$\blockB$'' to the other one.
\item For each $i\in\{1,\dots,r\}$, define the function $\sigma_i:\mb R\to\mb R$ by taking $\sigma_{i}(x)=x$ if the block $I_{i}$ consists of an $\blockA$-run followed by a $\blockB$-run,
and taking $\sigma_{i}(x)=1-x$ if the block $I_{i}$ consists of a $\blockB$-run
followed by an $\blockA$-run.
\item Let $\phi_0:[0,1]^{r}\to[0,1]^{r}$ be the isometry $(x_{1},\dots,x_{r})\mapsto(\sigma_{1}(x_{1}),\dots,\sigma_{r}(x_{r}))$.
%\mk{benny made a definition for $\sigma_i(x)$, but this was just supposed to be a product. I put a clarification.}.%, where $\sigma_i(x)=x$ if $\sigma_i=1$ and 
%$\sigma_i(x)=1-x$ if $\sigma_i=-1$. 
\item Let $\vec{t}$ be such that $[0,1]^r+\vec{t}=\prod_{j=1}^{r}V_{i(j)}$,
and let $\phi_{1}:[0,1]^r\to\prod_{j=1}^{r}V_{i(j)}$
be the translation $\vec{x}\mapsto\vec{x}+\vec{t}$.
\item Let $\phi=\phi_{1}\circ\phi_{0}$.
\end{compactenum}
It is easy to check that $L_{P}(M^{*})=L_{I_{1},\dots,I_{\ell}}(\mathcal{S}_{N})$.
% We can realise the distribution of $\mathcal{R}$ as $\phi(\mathcal{S}_{N})$, where $\mathcal{S}_{N}$ is a set of $N$ independent uniformly random points in $[0,1]^{n}$. 
% Moreover, the size of the largest $P$-clique $L_{P}(M^{*})$ in $M^{*}$ is precisely the length of the longest chain $L_{I_{1},\dots,I_{\ell}}(\mathcal{S}_{N})$ in $\mathcal{S}_{N}$ with respect to the poset $\mathcal{P}_{\{I_{1},\dots,I_{\ell}\}}$ defined in \cref{def:BW-variant}. 
So, still conditioning on an outcome $|M_{\mc B}|=N$, by \cref{thm:BW-variant}, with conditional probability $1-N^{-\omega(1)}$ we have that the random variable $L_{P}(M_{\mc B})$ is of the form $(a_{\mathcal{A}}+o(1))N^{1/r}$.

Also, by the concentration inequality \cref{thm:mcd-permutation} (see \cref{rem:permutation-matching}), with probability
at least $1-n^{-\omega(1)}$ we have
\[
|M_{\mc B}|=\left(f_{P}\left(\frac{|Q_{1}|}{rn},\dots,\frac{|Q_{\ell}|}{rn}\right)+o(1)\right)n,
\]

where $f_{P}(q_{1},\dots,q_{\ell})$ is the probability that for a
set of $r$ independent uniformly random points in $[0,1]$, exactly
$I_{i}$ of them fall between $q_{i-1}$ and $q_{i-1}+q_{i}$ (here we take $q_{0}=0$ for convenience). Note that $f_{P}(q_{1},\dots,q_{\ell})$ is a
homogeneous polynomial of degree $r$ in the variables $q_{1},\dots,q_{\ell}$.

Let $\alpha_{P}$ be the maximum value of $f_{P}(q_{1},\dots,q_{\ell})$
over all $q_{1},\dots,q_{\ell}\ge0$ summing to 1. Then, taking a
union bound over all possible partitions $\mathcal{B}$ (of which
there are at most $n^{\ell-1}$), we see that 
\[
\frac{L(Q)}{n^{1/r}}\overset{p}{\to}a_{\mathcal{A}}\alpha_{P}^{1/r}.
\]
So, the desired result follows with $b_{P}=a_{\mathcal{A}}\alpha_{P}^{1/r}$.
\end{proof}
\begin{proof}[Proof of \cref{prop:limit-information}] 
Recall the definitions of $f_{P}$ and $\alpha_{P}$ from the proof of \cref{thm:limit-exists}, and recall the estimates in that proof.
\begin{compactenum}
    \item Suppose $P$ has type $r$. As in the proof of \cref{thm:limit-exists}, let $\mathcal A$ be the trivial partition of $\{1,\dots,r\}$ into one part. Note that $\alpha_{P}=f_P(1)=1$, and by \cref{prop:BW-limit-information}, $a_{\mc A}=r/\Gamma(1/r)=1/\Gamma((r+1)/r)$. So, $b_P=a_{\mc A}\alpha_P^{1/r}=1/\Gamma((r+1)/r)$.
    \item Suppose $P$ has type $1+\dots+1$. As in the proof of \cref{thm:limit-exists}, let $\mathcal A$ be the partition of $\{1,\dots,r\}$ into $r$ singleton parts. Note that $\alpha_{P}=f_{P}(1/r,\dots,1/r)=r!/r^{r}$, and note that $a_{\mc A}$ is precisely the Bollob\'as--Winkler constant $c_r$, which by \cref{prop:BW-limit-information} is strictly larger than  $\frac{r^{2}}{(r!)^{1/r}\Gamma(1/r)}$. Therefore we have $b_P=c_r(r!/r^{r})^{1/r}=c_r(r!)^{1/r}/r>r/\Gamma(1/r)=1/\Gamma((r+1)/r)$.\qedhere
\end{compactenum}
% If $P$ has type $1+\dots+1$, then $\alpha_{P}=f_{P}(1/r,\dots,1/r)=r!/r^{r}$. 
% Let $c_r$ be the constant in \cref{prop:BW-limit-information} when $\mc A$ has $r$ parts of size 1.
% We have $b_P=c_r(r!/r^{r})^{1/r}=c_r(r!)^{1/r}/r>r/\Gamma(1/r)=1/\Gamma((r+1)/r)$.
% If $P$ has type $r$ then $\alpha_{P}=f_P(1)=1$, 
% and if $P$ has type $1+\dots+1$ then $\alpha_{P}=f_{P}(1/r,\dots,1/r)=r!/r^{r}$. The desired results follow, recalling the estimates in \cref{prop:BW-limit-information}.
% \zj{Not consistent with Proposition 1.11.}
\end{proof}

%% file: counting.tex
\section{Enumeration}\label{sec:counting}
In this section we estimate $\countClique{n}{P}{m}$ up to an exponential factor, proving  \cref{thm:count} and therefore \cref{thm:count-basic}.
The upper bound and the lower bound in \cref{thm:count} are proved separately. Despite \cref{thm:count} being about ordered matchings, both the upper and lower bound require the consideration of general ordered $r$-graphs; in particular, the upper bound requires the extremal results introduced in \cref{subsec:intro-extremal} (proved in \cref{sec:extremal}).

For the upper bound, we prove the following general estimate on the number of ordered matchings $M$ avoiding a particular sub-matching $Q$.

\begin{theorem}\label{thm:extremal-to-counting}
  Let $r\ge 2$ and $Q$ be an ordered $r$-matching with at least two edges.
  Suppose there exists some $n_Q>0$ such that $\ex_<(n,Q)\le \alpha n^{r-1}$ for $n\ge n_Q$.
  Then, if $n$ is sufficiently large in terms of $r$ and $n_Q$, the number of $Q$-free ordered $r$-matchings of size $n$ is at most
    \[e^{40n}(\alpha r^r n^{r-2})^{(r/(r-1))n}.\]
\end{theorem}
\begin{comment}
    next sentence was added 
\end{comment}
In the above statement, we allow $\alpha$ to depend on $n$. 
To obtain the upper bound in \cref{thm:count}, we can simply apply \cref{thm:extremal-to-counting} with $Q=\clique P m$ and the estimate in \cref{thm:extremal-clique}(2).

In order to prove the general statement in \cref{thm:extremal-to-counting}, we need a few basic facts about extremal numbers of general ordered $r$-matchings. 
First, we need the fact that extremal numbers always have order of magnitude at least $n^{r-1}$. (This is not hard to prove; in particular it follows from the $m=2$ case of \cref{theorem: lower bound for all patterns}, which we prove in \cref{sec:extremal}).
\begin{proposition} \label{claim: turan number is not small}
  Let $r \ge 2$, and let $Q$ be any $r$-matching with at least two edges. Then
  \[\ex_<(n,Q) \ge \binom{n}{r}-\binom{(n-r)_+}{r}.\]
\end{proposition}
% \begin{proof}
% Since removing edges from $Q$ cannot decrease $\ex_<(n,Q)$, we can assume that $Q$ has exactly 2 edges (i.e., it is an $r$-pattern). If $Q$ is an $r$-partite $r$-pattern, then the desired result follows from \cref{theorem: turan number for m=2}.
% Otherwise, if $Q$ is not $r$-partite, it is easy to see that $\ex_<(n,M)\ge \floor{n/r}^r$ by considering the ``balanced complete $r$-partite $r$-graph on $n$ vertices'': to be precise, consider an ordered $r$-graph $G$ with its vertices divided into $r$ contiguous parts of size at least $\floor{n/r}$, including every possible edge that intersects all the parts. This graph is $Q$-free, and one can check that  $\floor{n/r}^r\ge \binom{n}{r}-\binom{(n-r)_+}{r}$ for sufficiently large $n$ (in terms of $r$).
% \end{proof}

We also need a basic monotonicity property of extremal numbers of ordered matchings.

\begin{proposition} \label{claim: turan number doubles}
  Let $r,n\ge 2$ and let $Q$ be an ordered $r$-matching with more than one edge.
  Then, \[\ex_<(n,Q) \ge 2\ex_<(\ceil{n/2},Q).\]
\end{proposition}
\begin{proof}
Write the edges of $Q$ as $e_1,\dots,e_m$, with $e_1[1]<\dots<e_m[1]$. Say that $Q$ \emph{splits into sub-matchings} if there is an index $\ell\in\{1,\dots,m-1\}$ such that $e_i[r]<e_j[1]$ for any $1 \le i \le \ell<j \le m$ (i.e., $Q$ consists of two matchings placed side-by-side). We proceed differently depending on whether $Q$ splits into sub-matchings.

If $Q$ does not split into sub-matchings, then given any $Q$-free ordered $r$-graph $G$, we can ``glue together two copies of $G$'' to make a larger $Q$-free ordered $r$-graph $G'$. Specifically, we can take two copies of $G$, and identify the last vertex of the first copy with the first vertex of the second copy (so except for the single identified vertex, all vertices of the first copy come before all vertices of the second copy). Clearly, if the resulting graph $G'$ were to contain some copy of $Q$, then this copy must completely lie in one of the two copies of $G$.
But $G$ is $Q$-free, meaning that $G'$ is also $Q$-free. 
In addition, if $G$ had exactly $\ceil{n/2}$ vertices and $\ex_<(\ceil{n/2},Q)$ edges, then $G'$ has $2\ceil{n/2}-1\le n$ vertices and $2\ex_<(\ceil{n/2},Q)$ edges, proving the desired inequality.

On the other hand, if $Q$ does split into sub-matchings, then we can consider the $r$-graph $G$ on the vertex set $\{1,\dots,n\}$ containing all possible edges $e\in E(K_n^{(r)})$ with $e[1]\le\ceil{n/{2}}<e[r]$. 
Note that $G$ is $Q$-free. We deduce
\[
    \ex_<(n,Q)\ge e(G)\ge \binom{n}{r}-\binom{\ceil{n/2}}{r}-\binom{\floor{n/2}}{r}\ge 2\binom{\ceil{n/2}}{r}\ge 2\ex_<(\ceil{n/2},Q),
\]
as desired.
\end{proof}

Now, our proof of \cref{thm:extremal-to-counting} % the upper bound in \cref{thm:count}
proceeds by a ``contraction'' argument, similar to an argument for $r$-dimensional orders by Cibulka and Kyncl~\cite{CK17}. 
We execute this argument in two steps: first we use an inductive contraction argument to estimate the number of $Q$-free graphs on $n$ vertices, then we use this bound and a different contraction argument to upper bound the number of $Q$-free ordered matchings of size $n$. 
The first of these steps is encapsulated in the following lemma.

\begin{lemma} \label{lemma: count ordered graphs without M}
Fix a constant $r\ge 2$, and suppose $n$ is sufficiently large in terms of $r$. For any ordered $r$-matching $Q$ with more than one edge, the number of ordered $Q$-free $r$-graphs on the vertex set $\{1,\dots,n\}$ is at most
\[\exp\left(\vphantom \sum 2^{r+2}\ex_<(\ceil{n/2},Q)\right).\]
\end{lemma}
Note that if $\ex_<(n,Q)\le \alpha n^{r-1}$ for $n\ge n_Q$ (as in the statement of \cref{thm:extremal-to-counting}), the bound in \cref{lemma: count ordered graphs without M} is at most $\exp(O(\alpha n^{r-1}))$ for $n\ge 2n_Q$.
\begin{proof}
For any $n\in \mb N$, let $\mathcal G_n$ be the set of all $r$-graphs on the vertex set $\{1,\dots,n\}$ which are $Q$-free.

First, we describe how to ``contract pairs of vertices'' to transform an $n$-vertex ordered graph $G\in \mathcal G_n$ into an $\ceil{n/2}$-vertex ordered graph $\phi(G)\in \mathcal G_{\ceil{n/2}}$. For $G\in \mathcal G_n$, partition its vertex set $\{1,\dots,n\}$ into $\ceil{n/2}$ contiguous intervals $I_1,\dots,I_{\ceil{n/2}}$, where each $I_i$ has size $2$ except possibly $I_{\ceil{n/2}}$ (which has size $1$ if $n$ is odd). Then, $\phi(G)$ is obtained by ``contracting'' each $I_i$ to a single vertex. Specifically, we include $e \in E(K_{\ceil{n/2}}^{(r)})$ as an edge of $\phi(G)$ if and only if $\{i_1,\dots,i_r\} \in E(G)$ for some $i_1\in I_{e[1]}, i_2\in I_{e[2]},\dots,i_r\in I_{e[r]}$. Note that this contraction operation cannot create copies of $Q$.

Now, fixing $G' \in \mc{G}_{\ceil{n/2}}$, we are going to estimate the number of $G \in \mc{G}_n$ such that $\phi(G)=G'$. For all possible edges $e\in E(K_n^{(r)})$ of $G$, say that $e$ is \emph{contractible} if $\ceil{e[1]/2}<\dots<\ceil{e[r]/2}$ (i.e., if the vertices of $e$ lie in different $I_i$, in the definition of $\phi(G)$).

First, every edge $e\in E(G')$ arose by contracting some edge $f_e\in E(G)$. There are $2^r$ possibilities for $f_e$, and at least one must have been present in $G$, so given that $\phi(G)=G'$ there are at most $(2^{2^r}-1)^{e(G')}<2^{2^r\ex_<(n',Q)}$ ways to choose the contractible edges in $G$.

On the other hand, if $e\in E(G)$ is not contractible, then $e[k]+1=e[k+1]$ for some $1 \le k < r$. By \cref{lemma: count the number of r-tuples with given gap,claim: turan number is not small}, the number of $e\in E(K_n^{(r)})$ which have this property is
\[\binom{n}{r}-\binom{(n-r)_+}{r}<2^r\left(\binom{\ceil{n/2}}{r}-\binom{(\ceil{n/2}-r)_+}{r}\right)\le 2^r \ex_<(\ceil{n/2},Q),\]
assuming that $n$ is sufficiently large in terms of $r$ (say, $n\ge n_r$). So, the total number of ways to choose the non-contractible edges in $G$ is at most $2^{2^r\ex_<(\ceil{n/2},Q)}$.
An $r$-graph $G$ is specified by its contractible and non-contractible edges, so we have
\[\abs{\mc{G}_n} \le \abs{\mc{G}_{\ceil{n/2}}}\cdot 2^{2^r\ex_<(\ceil{n/2},Q)}\cdot 2^{2^r\ex_<(\ceil{n/2},Q)}\le \abs{\mc{G}_{\ceil{n/2}}}\cdot \exp\left(\vphantom\sum 2^{r+1}\ex_<(\ceil{n/2},Q)\right).\]
Iterating this inequality, and using that $\ex_<(\lceil n/2^k\rceil,Q) \leq 2^{-(k-1)} \ex_<(\ceil{n/2},Q)$ (by Proposition \ref{claim: turan number is not small}), we have
% applying \cref{claim: turan number doubles}, yields
\[
    \abs{\mc{G}_n} 
    \le \abs{\mc{G}_{n_r}} \prod_{k=1}^{\ceil{\log_2 (n/n_r)}} \exp\left(\vphantom\sum 2^{r+1}\ex_<(\lceil n/2^k\rceil,Q)\right)
    \le \abs{\mc{G}_{n_r}} \cdot \exp\left(\vphantom\sum (2^{r+2}-1)\ex_<(\ceil{n/2},Q)\right).
  \]
  % for some $n'\le n_r$. The desired result follows (note that $|\mc G_{n'}|=O_r(1)$).
  When $n$ is sufficiently large in terms of $r$, \cref{claim: turan number is not small} implies $\abs{\mc{G}_{n_r}} \le 2^{\binom{n_r}{r}} \le \exp\left( \ex_<(\ceil{n/2},Q) \right)$.
  Then, $\abs{\mc{G}_n} < \exp\left(\vphantom\sum 2^{r+2}\ex_<(\ceil{n/2},Q)\right)$, as desired.
\end{proof}
We are now ready to complete the proof of \cref{thm:extremal-to-counting}.
\begin{proof}[Proof of \cref{thm:extremal-to-counting}]
Let $\mc{M}_n$ be the set of ordered $r$-matchings on the vertex set $\{1,\dots,rn\}$ that are $Q$-free, and for any $n'$ let $\mc G_{n'}$ be the set of $r$-graphs on the vertex set $\{1,\dots,n'\}$ that are $Q$-free.

As in the proof of \cref{lemma: count ordered graphs without M}, we will define a ``contraction'' operation (though this time we will contract larger intervals, and we will not need to iterate our operation). For some $B$, whose value we will specify shortly, let $n'=\lceil rn/B\rceil$ and partition $\{1,\dots,rn\}$ into $n'$ contiguous intervals $I_1,\dots,I_{n'}$, where each $I_i$ has size $B$ except possibly $I_{n'}$ (which is smaller if $n$ is not divisible by $B$). For $M \in \mc{M}_n$ let $\psi(M)\in \mc G_{n'}$ be the graph on the vertex set $\{1,\dots,n'\}$ obtained by contracting each interval $I_i$ to a single vertex (specifically, we include $e \in E(K_{n'}^{(r)})$ as an edge of $\psi(G)$ if and only if $(i_1,\dots,i_r)\in E(M)$ for some $i_1 \in I_{e[1]},\dots,i_r \in I_{e[r]}$).
Now, we specify $B$ to be the largest integer such that $B^{r-1}\le \alpha r^r n^{(r-2)}$. 
Note that $B^{r-1}\ge\alpha r^rn^{r-2}/2$ and $(n')^{r-1}\le 2(rn/B)^{r-1}= 2{r^{r-1}n^{r-1}}/{B^{r-1}}\le4n/(r\alpha)$; and $(n')^{r-1}\ge (rn/B)^{r-1}= {r^{r-1}n^{r-1}}/{B^{r-1}} \ge  n/(r\alpha)$.
\begin{comment}
%; The last series of inequalitites  was added
\end{comment}
By \cref{lemma: count ordered graphs without M} we have
\[\abs{\mc{G}_{n'}} \le \exp\left(\vphantom\sum 2^{r+2}\ex_<(\lceil n'/2\rceil,Q)\right) \le  \exp\left(\vphantom\sum 2^{r+3} \alpha(n'/2)^{r-1}\right) \le \exp\left(\vphantom\sum 16\alpha(n')^{r-1}\right)\le e^{64n/r}\le e^{32n},\]
assuming $n$ is sufficiently large with respect to $r$. (In this deduction we assumed that $\alpha\leq 4n^{3/4}/r$ and $n$ is sufficiently large with respect to $n_Q$ so that $  (n/(r\alpha))^{1/(r-1)}/2\ge n_Q$, thus $\lceil n'/2\rceil\ge (n/(r\alpha))^{1/(r-1)}/2\ge n_Q$.
Note that if $\alpha> 4n^{3/4}/r$ then $(n')^{r-1} \leq  4n/r\alpha < n^{1/4}$, hence $\abs{\mc{G}_{n'}}\le 2^{\binom{n'}{r}} \leq 2^{(n')^r}<e^{32n}$).
\begin{comment}
 (In this deduction we assumed $\lceil n'/2\rceil\ge n_Q$, but note that the above inequality $\abs{\mc{G}_{n'}}\le e^{32n}$ holds trivially if $\lceil n'/2\rceil < n_Q$, assuming $n$ is sufficiently large with respect to $n_Q$).
\end{comment}
Now, fixing $G\in\mc{G}_{n'}$, let us estimate the number of $M\in\mc{M}_n$ with $\psi(M)=G$. Viewing each edge $e\in E(G)$ as an ordered $r$-tuple $(e[1],\dots,e[r])$, let
\[\mc{T}=E(G)\cup\{(t_1,\dots,t_r)\in[n']^r: t_1\le t_2\le\dots\le t_r \text{ and there exists $k$ such that } t_k=t_{k+1} \}.\]
The idea is that $\mc T$ specifies the ``possible places where edges of $M$ can lie'': if $\psi(M)=G$ then for each edge $e\in M$ there must be some tuple $(t_1,\dots,t_r)\in \mathcal T$ such that $e[1]\in I_{t_1},\dots,e[r]\in I_{t_r}$ (say that $e$ is \emph{consistent} with $\vec t$). 
% By \cref{claim: turan number is not small}, we have
Then,
  \begin{equation} \nonumber
    \begin{aligned}
      \abs{\mc{T}} 
      =&\, |E(G)| + \binom{n'+r-1}{r}-\binom{n'}{r} 
      \le \ex_<(n',Q) + 2\binom{n'}{r}-2\binom{(n'-r)_+}{r} \\
      \le&\, 3\ex_<(n',Q) \le 3\alpha (n')^{r-1} 
      \le \frac{12n}{r}\le 6n.
    \end{aligned}
  \end{equation}
  % \[
  %   \abs{\mc{T}} 
  %   \le \ex_<(n',Q) + \binom{n'+r-1}{r}-\binom{n'}{r} 
  %   \le 2\ex_<(n',Q) \le 2\alpha (n')^{r-1} 
  %   \le \frac{8n}{r}\le 4n.
  % \]
(In the deduction above  we assumed that $\alpha\le 4 n^{3/4}/r$ and $n$ is sufficiently large with respect to $n_Q$ and $r$ so that the first inequality holds, and $n'\geq (n/(r\alpha))^{1/(r-1)} \ge n_Q$ so that the third inequality holds. In the second inequality we used \cref{claim: turan number is not small}. Also note that the inequality $\abs{\mc{T}}\le 6n$ holds trivially if $\alpha > 4n^{3/4}/r$ as this implies that $n'< n^{1/(4(r-1))}$ and thus $\abs{\mc T}\le (n')^r<n$.)
 \begin{comment}
(In the first and third inequalities, we assumed that $n'\ge n_Q$, but $\abs{\mc{T}}\le 6n$ holds trivially if $n'<n_Q$, assuming $n$ is sufficiently large with respect to $n_Q$; in the second inequality we used \cref{claim: turan number is not small}.)
 \end{comment}

Now, for every $M\in\mc{M}_n$ with $\psi(M)=G$, we consider the number $a_{\vec{t}}$ of edges $e\in M$ which are consistent with $\vec t$. Since $\sum_{\vec{t}\in\mc{T}} a_{\vec{t}}=n$, the number of possibilities for $(a_{\vec{t}}:\vec t\in \mc T)$ is at most
\[\binom{n+\abs{\mc{T}}-1}{n}<\binom{n+6n}{n}\le 2^{7n}\le e^{5n}.\]
Then, given any choice of $(a_{\vec{t}}:\vec t\in \mc T)$, the number of ways to choose the edges of $M$ is at most
\[\prod_{\vec{t}\in\mc{T}} ({B^r})^{a_{\vec{t}}}\le B^{rn}.\]
Thus,
\[\abs{\mc{M}_n}
    \le \abs{\mc{G}_{n'}}\cdot e^{5n}\cdot B^{rn} \le e^{37n}\left(\alpha r^r n^{r-2} \right)^{(r/(r-1))n},
    \]
    as desired.
\end{proof}

Now, we turn our attention to the lower bound in \cref{thm:count}. Actually, we are able to prove the desired bound even when restricting our attention to \emph{$r$-partite} $r$-matchings (i.e., $r$-dimensional orders). Let $N_{P,m}^{\mr{part}}(n)$ be the number of $r$-partite $r$-matchings $M$ on the vertex set $\{1,\dots,rn\}$ with $L_P(M)<m$.
\begin{theorem}\label{thm:count-r-partite}
  Let $r \ge 2$ and $P$ be an $r$-partite $r$-pattern. 
  Also consider any $m\ge 2$ such that $m-1$ is divisible by $r-1$, and consider any integer $b\ge r(m-1)/(r-1)$. 
  Then, defining
  \[
  n=\binom{b+r-(m-1)/(r-1)}{r}-\binom{b+r-r(m-1)/(r-1)}{r},
  \]
  we have
  \[N_{P,m}^{\mr{part}}(n)\ge e^{-rn}(r!)^{-(r/(r-1))n} (m-1)^{(r/(r-1))n} n^{(r-1-1/(r-1))n}.\]
\end{theorem}
\begin{remark}\label{rem:common-increasing}
    %Recall from the proof of \cref{lemma: turan number for r-partite graphs} that 
    If $P$ is an $r$-partite $r$-pattern with block representation $|\block_1|\block_2|\dots|\block_r|$, and $M$ is an $r$-partite $r$-matching (with edges viewed as $r$-tuples in $\{1,\dots,n\}$), then the condition $L_P(M)<m$ is equivalent to the condition that there is no sequence of edges $e_1,\dots,e_m$ which is increasing in all coordinates $i$ with $\block_{i}=\AB$, and
decreasing in all coordinates $i$ with $\block_{i}=\BA$. So, $N_{P,m}^{\mr{part}}(n)$ does not actually depend on $P$, and is precisely equal to the number of $(r-1)$-tuples of permutations of $\{1,\dots,n\}$ which contain no common increasing subsequence of length $m$.
\end{remark}
\begin{remark}
    The bound in \cref{thm:count-r-partite} is only for $n$ of a certain form, but it's easy to deduce almost as strong of a bound for all $n$ (if we view $r$ as a constant).

Indeed, first note that if (say) $b\ge 2m$, the resulting value of $n$ scales like $m b^{r-1}$. So, for any desired order of magnitude (greater than $m^r$) we can choose $b$ such that $n$ has our desired order of magnitude.
    
    Then, note that $N_{P,m}^{\mr{part}}(n+1)\ge N_{P,m}^{\mr{part}}(n)$ for all $n$ (i.e., $N_{P,m}^{\mr{part}}(n)$ is monotone increasing in $n$). This is because for any $(r-1)$-tuple of permutations of $\{1,\dots,n\}$ which contain no common increasing subsequence of length $m$, we can simply prepend each permutation with ``$n+1$'' to obtain an $(r-1)$-tuple of permutations of $\{1,\dots,n+1\}$ which still contain no common increasing subsequence of length $m$.
\end{remark}
\begin{proof}[Proof of \cref{thm:count-r-partite}]
As we have just discussed (in \cref{rem:common-increasing}), it suffices to consider the case where~$P$ has representation $|\AB|\AB|\dots|\AB|$, in which case the condition $L_P(M)< m$ is equivalent to the condition that there is no sequence of $m$ edges that is increasing in every coordinate.

We start by constructing an $r$-partite ordered $r$-graph $G$, which we will ``uncontract'' to obtain our desired matchings $M$ (by analogy with the contraction operation in the proof of \cref{thm:extremal-to-counting} earlier in this section). 
Let $a=(m-1)/(r-1)$, and let $U_1,\dots,U_r$ be disjoint sets of size $b-a+1$ (we write $U_i=\{u_a^i,\dots,u_b^i\}$). Define $G$ to be the graph on the vertex set $U_1\cup\dots\cup U_r$ (ordered such that $U_1<U_2<\dots<U_r$), where $\{u_{j_1}^1,u_{j_2}^2,\dots, u_{j_r}^r\}$ is included as an edge if and only if 
\[b+1\le j_1+\dots+j_r\le b+m-1.\]
By symmetry, for any $j \in \{a,\dots,b\}$, the vertices $u_j^1,\dots,u_j^r$ have the same degree, which we denote as~$d_j$. Note that each $d_j\ge 1$, since for any integer $k$ in the range $[a,b]=[(m-1)/(r-1),b]$, there is always some way to add $r-1$ more integers in this range to $k$, to obtain a number in the range $[b+1,b+m-1]$. 
Also, observe that $b \ge r(m-1)/(r-1)= ra$.
% as $b \ge r(m-1)/(r-1)$, it holds that $b-ra \ge 0$.
Writing $n=d_a+\dots+d_b=e(G)$, we have 
\begin{align}
      n
      =&\; \#\left( \vphantom\sum(j_1,\dots,j_r)\in\{a,\dots,b\}^r\;:\; b+1\le  j_1+\dots+j_r \le b+(m-1) \right) \notag\\
      =&\; \#\left( \vphantom\sum(x_1,\dots,x_r)\in \{0,\dots,b-a\}^r\;:\; b+1-ra\le x_1+\dots+x_r \le b+(m-1)-ra=b-a\right)\notag \\
      =&\; \#\left( \vphantom\sum(x_1,\dots,x_r)\in \mathbb{Z}^r_{\ge 0}\;:\; x_1+\dots+x_r \le b-a\right)\notag\\
      &\qquad\qquad-\#\left( \vphantom\sum(x_1,\dots,x_r)\in \mathbb{Z}^r_{\ge 0}\;:\; x_1+\dots+x_r \le b-ra\right)\notag\\
      =&\;\binom{b+r-(m-1)/(r-1)}{r}-\binom{b+r-r(m-1)/(r-1)}{r}.\label{eq:n-formula}
  \end{align}
% Now, let $n=d_a+\dots+d_b=e(G)$.
Starting from the vertex set $U_1\cup\dots\cup U_r$ of $G$, ``uncontract'' each vertex $u_j^i$ into an interval $I_j^i$ of length $d_j$, to obtain an ordered set $V$ of $rn$ vertices such that $I^1_a<\dots<I^1_b<I^2_a<\dots<I^{r-1}_b<I^r_a<\dots<I^r_b$. 
Then, we consider each matching on the vertex set $V$ which ``yields $G$ after contracting the intervals $I_j^i$ to single vertices''. Specifically, we consider each matching $M$ on $V$ with the property that for $e\in E(M)$ there is a unique edge $f\in E(G)$ with $e[1] \in I_{f[1]}^1,e[2]\in I_{f[2]}^2,\dots,e[r]\in I_{f[r]}^r$ (and every $f \in E(G)$ corresponds to precisely one $e \in E(M)$).

It is not hard to see there are exactly $\prod_{v\in V(G)}d_G(v)=\left(\prod_{j} d_j!\right)^r$ ways to choose ordered $r$-matchings $M$ as above.
% The number of ways to choose matchings $M$ as above is exactly $\prod_{i,j}d_j!$. 
Note that each such matching $M$ satisfies $L_P(M)<m$. 
To see this, recall that every edge of $G$ (and therefore every edge of $M$) can be identified by a unique $r$-tuple $(j_1,\dots,j_r)$ with $b+1\le j_1+\dots+j_r\le b+m-1$. 
For two edges $e,e'\in E(M)$ (corresponding to two different $r$-tuples $(j_1,\dots,j_r)$ and $(j_1',\dots,j_r')$), we can only have $e[i]>e'[i]$ when $j_i\ge j_i'$. In particular, if $e[i]>e'[i]$ for all $i$, then $j_1+\dots+j_r>j_1'+\dots+j_r'$ (note that $(j_1,\dots,j_r)\neq(j_1',\dots,j_r')$ when $e\ne e'$). For our $r$-tuples $(j_1,\dots,j_r)$ under consideration, $j_1+\dots+j_r$ is constrained to an interval of $m-1$ integers, so there cannot be any sequence of $m$ edges which is increasing in every coordinate (i.e., $L_P(M)< m$).

It remains to lower-bound $\left(\prod_{j} d_j!\right)^r$ in terms of $n,m,r$. 
Recall that the Gamma function $\Gamma$ is log-convex and satisfies $\Gamma(x)\ge (x/e)^{x-1}$.
Using that $d_a+\dots+d_b=n$, we have
\begin{align}
    \left(\prod_{j=a}^b d_j!\right)^r
    = \left( \prod_{j=a}^b\Gamma(d_j+1)\right)^r
    &\ge \left( \Gamma\left(\frac{n}{b-a+1} +1\right)\right)^{(b-a+1)r} \notag\\
    &> \left(\left(\frac{n}{e(b-a+1)}\right)^{n/(b-a+1)}\right)^{(b-a+1)r} \notag\\
    &= e^{-rn} n^{rn} \left({b-\frac{m-1}{r-1}+1}\right)^{-rn}.\label{eq:dj-factorial-intermediate}
  \end{align}
On the other hand, recalling \cref{eq:n-formula},
\begin{align*}
    n&=\binom{b+r-(m-1)/(r-1)}{r}-\binom{b+r-r(m-1)/(r-1)}{r} \\
    &= \frac{1}{r!}\left[{\prod_{i=1}^r\left(b-\frac{m-1}{r-1}+i \right) - \prod_{i=1}^r\left(b-\frac{r(m-1)}{r-1}+i\right)}\right] \\
    &> \frac{1}{r!}\cdot (m-1)\prod_{i=1}^{r-1} \left(b-\frac{m-1}{r-1}+i \right)\\
    &> \frac{m-1}{r!}{\left(b-\frac{m-1}{r-1}+1\right)^{r-1}}.
  \end{align*}
In other words, $b-(m-1)/(r-1)+1 < (r!n/(m-1))^{1/(r-1)}$.
% The last inequality gives an upper bound on $b-(m-1)/(r-1)+1$. 
Plugging this into \cref{eq:dj-factorial-intermediate}, we have
\begin{align*}
    N_{P,m}^{\mr{part}}(n)
    \ge \prod_{i=1}^r\prod_{j=a}^b d_j!
    &\ge e^{-rn}n^{rn} \left(\frac{r!n}{m-1}\right)^{-(rn)/(r-1)}\\
    &=e^{-rn}(r!)^{-(r/(r-1))n} (m-1)^{(r/(r-1))n} n^{(r-1-1/(r-1))n},
\end{align*}
as desired.
\end{proof}

%% file: hypergraph-partition.tex
\section{Partitioning ordered hypergraphs}\label{sec:partitioning}
A well-known result of Erd\H os and Kleitman~\cite{EK68} says that every $r$-graph with $m$ edges has an $r$-partite subgraph with at least $(r!/r^r)m$ edges (here ``$r$-partite'' means that we can divide the vertex set into $r$ parts such that every edge has exactly one vertex in each part). This theorem is enormously useful in extremal hypergraph theory, as it allows one to prove a result about $r$-partite $r$-graphs and ``transfer'' it to general $r$-graphs.

It would be very useful if the same would be possible in the setting of ordered hypergraphs (where our notion of $r$-partiteness should impose that the $r$ parts are each \emph{intervals} according to our ordering). Although it is not possible to prove a direct analogue of the Erd\H os--Kleitman theorem, F\"uredi, Jiang, Kostochka, Mubayi and  Verstra\"ete~\cite{FJKMV21-partition} managed to prove a result that is almost as good: in a certain sense every ordered $r$-graph $G$ has a subgraph $G'$ which is ``nearly'' $r$-partite, in such a way that the ``relative density'' of $G'$ is not too much less than the relative density of $G$.

In this section we prove a refinement of the F\"uredi--Jiang--Kostochka--Mubayi--Verstra\"ete theorem, with essentially optimal dependence on $r$, which might be of independent interest.
This will be used in our proof of \cref{thm:extremal-clique}(2). 
% \mk{benny said that this is analogous to some stuff in coding theory? let's briefly talk about that.} 
To state our theorem we need some definitions.

\begin{definition}\label{def:equipartite}
Let $1 \le k \le r$ and $\vec{a}$ be a positive integer-valued vector of length $k$ with sum of entries $\|\vec{a}\|_1=r$.
We say an ordered $r$-graph $G$ is \emph{$\vec{a}$-equipartite} if the following holds.
% Let $\mathcal{A}_{r}$ be the set of positive integer vectors $\vec{a}\in\bigcup_{\ell=1}^{r}\mb Z_{>0}^{\ell}$ with sum of entries $\|\vec{a}\|_{1}=r$. 
% For $\vec{a}\in\mathcal{A}_{r}$, say an ordered $r$-graph $G$ is \emph{$\vec{a}$-equipartite} \emph{with order $n$} if the following properties hold.
\begin{compactenum}
\item 
%Write $k=\|\vec{a}\|_{0}$ for the number of entries of $\vec{a}$.
$G$ must have exactly $kn$ vertices.
\item Partition the vertex set of $G$ into contiguous intervals $I_{1},\dots,I_{k}$
of size $n$. For each $i\in\{1,\dots,k\}$ and each edge $e$
of $G$, we must have $|e\cap I_{i}|=a_{i}$.
\end{compactenum}
If $G$ is $\vec{a}$-equipartite with order $n$, then we define
its \emph{$\ell$-density} to be its number of edges divided by $n^{\ell}/\prod_{i}a_{i}!$
(this is a twisted notion of density, which can take values
between $0$ and $n^{r-\ell}$).

Moreover, say that $G$ is $k$-equipartite if it is $\vec{a}$-equipartite
for some $\vec{a}$ with $\|\vec{a}\|_{0}=k$, and say that $G$ is
$(\ge t)$-equipartite if it is $k$-equipartite for some $k\ge t$.
\end{definition}

Note in particular that there is only one possible $\vec a$ with $\|\vec{a}\|_{0}=r$: namely, the all-ones vector of length $r$. There are $r-1$ essentially equivalent $\vec a$ with $\|\vec{a}\|_{0}=r-1$: namely, the the vectors of length $r-1$, whose all entries are ``1'' except one ``2''. So, being $(\ge r-1)$-equipartite corresponds to being $\vec{a}$-equipartite for one of the above  two types of $\vec{a}$.

Now, our refinement of the F\"uredi--Jiang--Kostochka--Mubayi--Verstra\"ete theorem is as follows.

\begin{theorem}\label{thm:mk-partitioning}
Fix $r\in\mb N$ and $\varepsilon>0$, and suppose $n$ is sufficiently
large in terms of $r$ and $\varepsilon$. 
Let $p\ge\varepsilon$ and let $G$ be an ordered $r$-graph with of $n$ vertices and at least $(p/n)\binom{n}{r}$
edges.
% (i.e., as an $(r)$-equipartite ordered hypergraph, $G$ has $(r-1)$-density at least $p$).
Then, $G$ contains a $(\ge r-1)$-equipartite
subgraph $G'$ of order $n^{\Omega_{r}(1)}$ with $(r-1)$-density at least $\Omega(p/r^2)$.% (i.e., the corresponding~$\vec{a}$ for~$G'$ in \cref{def:equipartite} is either the all-ones vector of length $r$ or a vector of length $(r-1)$ with all but one entry to be 1 and the one entry to be 2).
\end{theorem}

We also show that the multiplicative factor $O(r^{-2})$ in the ``loss of density''  in \cref{thm:mk-partitioning} is essentially best possible.

\begin{theorem}\label{thm:mk-partitioning-best-possible}
Fix a constant $r\in\mb N$. There exists an ordered $r$-graph $G$ with $(1+o(1))r(r-1)n^{r-1}/r!$ edges such that any order-$m$ $(\ge r-1)$-equipartite subgraph $G'\subseteq G$ has $(r-1)$-density at most $2+o_{m\to \infty}(1)$.
\end{theorem}

\begin{remark}
    Instead of making an $(r-1)$-density assumption, one could make an $\ell$-density assumption for any $\ell\le r$. In general, such an assumption would allow one to prove an analogue of \cref{thm:mk-partitioning} in which we guarantee a $(\ge\ell)$-equipartite subgraph with reasonably large $\ell$-density. (Such a result was previously proved in \cite{FJKMV21-partition}, and our proof approach can always attain a better dependence on $r$). However, since $(r-1)$-density is the relevant parameter in most applications, we prefer to keep our proof as simple as possible, and only treat the case $\ell=r-1$.
\end{remark}

To prove \cref{thm:mk-partitioning}, the following lemma turns out to be crucial.
% Now, the key lemma for the proof of \cref{thm:mk-partitioning} is as follows.
Roughly speaking, given any $\vec a$-equipartite $G$ (for some $\vec a$), we can either obtain a subgraph of $G$ with more parts or a subgraph of $G$ with the same number of parts but higher density.
\begin{lemma}\label{lem:zhihan}
Suppose $1 \le k \le r$ and $\vec{a}$ is a positive integer-valued vector of length $k$ with $\|\vec{a}\|_1=r$.
Let $G$ be an $\vec{a}$-equipartite ordered $r$-graph with $(r-1)$-density~$p$, and consider any $R\in\mb N$. 
% Consider $\vec{a}\in\mathcal{A}_{r}$ with $\|\vec{a}\|_{0}=k$, let $G$ be an $\vec{a}$-equipartite ordered $r$-graph with $(r-1)$-density~$p$, and consider any $R\in\mb N$. 
Suppose that $n$ is sufficiently large in terms of $r,R$. 
Then, for some $t\in\mb N$ satisfying $\|\vec{a}\|_{0}\le t\le r$, we can find a $t$-equipartite subgraph of order $\floor{n/R}$ in~$G$, with $(r-1)$-density at least 
\[
\frac{p}{4R}\left(\frac{R}{(r-k+1)^{2}}\right)^{r-t}.
\]
\end{lemma}
We remark that in practice we will apply \cref{lem:zhihan} with $R$ a large constant.

\begin{proof}
Write $I_1,\dots,I_k$ for the $k$ parts of $G$. First, it is convenient to reduce to a subgraph whose order is divisible by $R$. 
So, let $n'=R\floor{n/R}$. By averaging, there's a way to delete $n-n'<R$ vertices from each $I_i$ to obtain an $\vec{a}$-equipartite subgraph $G'$, such that at most $(rR/n)e(G)$ edges are deleted.
So, the $(r-1)$-density of $G'$ is at least $p(1-rR/n)\ge p/2$
for $n \ge 2rR$.

Now, we divide each $I_i$ into $R$ contiguous intervals $I_{i,1},\dots,I_{i,R}$ of the same length $n/R$. 
Write $\mbm 1\in \mb R^R$ be the all-ones vector of length $R$, and for every $t\in\{k,k+1,\dots,r\}$, define $\mathcal{B}_{\vec{a},t}^{R}$ to be the set of $k\times R$ matrices $b\in\mb Z_{\ge0}^{k\times R}$ with $\|b\|_{0}=t$ nonzero entries, satisfying $b\mbm1=\vec{a}$.
Then, for each $b\in \mathcal B_{\vec a,t}^R$, we define a $t$-equipartite subgraph $G_b$ of order $\floor{n/R}$, as follows.
The vertex set of $G_b$ is obtained by including part $I_{i,j}$ whenever $b_{ij}\ne 0$, and the edge set is obtained by including each $e\in G'$ if $|e\cap I_{i,j}|=b_{ij}$ for all $i,j$. 
Note that the subgraphs $G_b$ (as $b$ varies in $\mathcal B_{\vec a,t}^R$, while $t$ varies in $\{k,\dots,r\}$) partition the edges of $G'$.

% Recall the definition of $\mathcal{B}_{\vec{a},t}^{R}$ from \cref{lem:partitioning-inequality}.
% For each $b\in \mathcal B_{\vec a,t}^R$, we define a $t$-equipartite subgraph $G_b$, as follows. 
% The vertex set of $G_b$ is obtained by including part $I_{i,j}$ whenever $b_{ij}\ne 0$, and the edge set is obtained by including each $e\in G$ if and only if $|e\cap I_{i,j}|=b_{ij}$ for all $i,j$. 
% Note that the subgraphs $G_b$ (as $b$ varies in $b\in \mathcal B_{\vec a,t}^R$, and $t$ varies in $\{k,\dots,r\}$) partition the edges of $G$.

Now, if none of the subgraphs $G_b$ satisfied the conclusion of the lemma, the total number of edges in $G'$ would be 
\[
e(G')
< \sum_{t=k}^{r}\sum_{b\in\mathcal{B}_{\vec{a},t}^{R}} \frac{p}{4R}\left(\frac{R}{(r-k+1)^{2}}\right)^{r-t}\frac{(n/R)^{r-1}}{\prod_{i,j}b_{ij}!}
= \sum_{t=k}^{r} \frac{p\cdot R^{-t}}{4(r-k+1)^{2(r-t)}} \frac{n^{r-1}}{\prod_{i}a_i!} \sum_{b\in\mathcal{B}_{\vec{a},t}^{R}} \frac{\prod_{i}a_i!}{\prod_{i,j}b_{ij}!}.
\]
To show that this leads to contradiction, we need the following combinatorial inequality.
%But then, we can use the following claim, which will be proved later.
\begin{claim}\label{claim:partitioning-inequality}
    For every $k \le t \le r$, we have 
    $
    \sum_{b\in\mathcal{B}_{\vec{a},t}^{R}} \left({\prod_{i}a_{i}!}/{\prod_{i,j}b_{i,j}!}\right)
    \le R^{t}\left({(r-k+1)^{2}}/{2}\right)^{r-t}.
    $
%     \[
% \sum_{b\in\mathcal{B}_{\vec{a},t}^{R}}\frac{\prod_{i}a_{i}!}{\prod_{i,j}b_{i,j}!}\le R^{t}\left(\frac{(r-k+1)^{2}}{2}\right)^{r-t}.
% \]
\end{claim}
We will prove \cref{claim:partitioning-inequality} momentarily, but first we see how to use it to conclude the proof of \cref{lem:zhihan}: recalling our estimate for $e(G')$ above, we have 
\[
e(G')
< \sum_{t=k}^{r} \frac{p\cdot R^{-t}}{4(r-k+1)^{2(r-t)}} \frac{n^{r-1}}{\prod_{i}a_i!} \cdot R^{t}\left(\frac{(r-k+1)^{2}}{2}\right)^{r-t}
= \sum_{t=k}^{r}\frac{p}{4}\cdot2^{-(r-t)}\cdot\frac{n^{r-1}}{\prod_{i}a_{i}!}
< \frac{(p/2)n^{r-1}}{\prod_{i}a_{i}!},
\]
% \[
% e(G)<\sum_{t=k}^{r}\sum_{b\in\mathcal{B}_{\vec{a},t}^{R}} \frac{p}{4R}\left(\frac{R}{(r-k+1)^{2}}\right)^{r-t}\frac{(n/R)^{r-1}}{\prod_{i,j}b_{ij}!}
% \le\sum_{t=k}^{r}\frac{p}{4}\cdot2^{-(r-t)}\cdot\frac{n^{r-1}}{\prod_{i}a_{i}!}< \frac{(p/2)n^{r-1}}{\prod_{i}a_{i}!},
% \]
contradicting the fact that $G'$ has $(r-1)$-density at least $p/2$. 
% The second inequality follows by \cref{lem:partitioning-inequality}.
\end{proof}

% To prove \cref{thm:mk-partitioning} we will need the following inequality.
% \begin{lemma}\label{lem:partitioning-inequality}
% Consider $\vec{a}\in\mathcal{A}_{r}$ with $\|\vec{a}\|_{0}=k$, and $R,t\in\mb \mb{Z}_{>0}$ satisfying
% $\|\vec{a}\|_{0}\le t\le r$. Let $\mbm 1\in \mb R^R$ be the all-ones vector of length $R$, and let $\mathcal{B}_{\vec{a},t}^{R}$ be
% the set of $k\times R$ matrices $b\in\mb Z_{\ge0}^{k\times R}$ with
% $\|b\|_{0}=t$ nonzero entries, satisfying $b\mbm1=\vec{a}$.
% (We think of $\mathcal{B}_{\vec{a},t}^{R}$ as the set of possible
% $t$-partite ``refinements'' of $\vec{a}$, after splitting each of the $k$ parts
% into $R$ sub-parts). Then
% \[
% \sum_{b\in\mathcal{B}_{\vec{a},t}^{R}}\frac{\prod_{i}a_{i}!}{\prod_{i,j}b_{i,j}!}\le R^{t}\left(\frac{(r-k+1)^{2}}{2}\right)^{r-t}.
% \]
% \end{lemma}
Now we prove \cref{claim:partitioning-inequality}.
\begin{proof}[Proof of \cref{claim:partitioning-inequality}]
We prove this inequality combinatorially, relating both sides of the inequality to certain counts of tuples of set partitions.

In this proof we write $[n]=\{1,\dots,n\}$, and write $\|f\|_{0}$
for the number of elements in the image of a function $f$. For $1\le q\le n$,
let $\mathcal{P}_{n,q}$ be the collection of (unlabelled) partitions
of $[n]$ into $q$ nonempty subsets (so $|\mathcal{P}_{n,q}|=\stirling nq$
is a Stirling number of the second kind). Let $\mathcal{P}_{n}=\bigcup_{q=1}^{n}\mathcal{P}_{n,q}$
be the collection of partitions of $[n]$ into any number of subsets,
and let $\|P\|_{0}$ be the number of parts of a partition $P\in\mathcal{P}_{n}$.
Then, define 
\[
\mathcal{P}_{\vec{a},t}=\left\{ (P_{1},\dots,P_{k})\in\mathcal{P}_{a_{1}}\times\dots\times\mathcal{P}_{a_{k}}\;\middle|\;\sum_{i=1}^{k}\|P_{i}\|_{0}=t\right\} .
\]
Note that
\begin{equation}
\begin{aligned}
  \sum_{b\in\mathcal{B}_{\vec{a},t}^{R}}\frac{\prod_{i}a_{i}!}{\prod_{i,j}b_{i,j}!}
  = \sum_{b\in\mathcal{B}_{\vec{a},t}^{R}}\prod_{i}\binom{a_i}{b_{i,1}\dots b_{i,R}} \le R^{t}|\mathcal{P}_{\vec{a},t}|.
  % = \#\left(f_{1}:[a_{1}]\to[R],\;\dots,\;f_{k}:[a_{k}]\to [R]\;\middle|\;\sum_{i=1}^{k}\|f_{i}\|_{0}=t\right)\le R^{t}|\mathcal{P}_{\vec{a},t}|.
\end{aligned}\label{eq:functions-partitions}
\end{equation}
(To see this, observe that the right-hand side of \cref{eq:functions-partitions} can be interpreted as the number of ways to choose a tuple of partitions $(P_1, \dots, P_k)\in \mathcal{P}_{\vec{a},t}$ together with a labelling of the parts in these partitions, in such a way that for each $1\le i\le k$, each of the parts in $P_i$ is assigned a distinct label from $[R]$).

Now, note that every $(P_{1},\dots,P_{k})\in\mathcal{P}_{\vec{a},t}$
gives rise to a partition $P\in\mathcal{P}_{r,t}$: first, $P_{1}$
describes a partition of $\{1,\dots,a_{1}\}$, then $P_{2}$ describes
a partition of $\{a_{1}+1,\dots,a_{1}+a_{2}\}$, and $P_{3}$ describes
a partition of $\{a_{1}+a_2+1,\dots,a_{1}+a_{2}+a_3\}$, and so on. 
We can then ``collapse'' $P$ into a partition $P'\in\mathcal{P}_{r-k+1,t-k+1}$, by sequentially merging the parts containing $a_{1}$ and $a_{1}+1$, then merging the parts containing $a_{1}+a_{2}$ and $a_{1}+a_{2}+1$, and so on.
% (Note that all these $k-1$ pairs of parts are distinct, by the construction of $P$). 
Since this mapping from $(P_{1},\dots,P_{k})$ to $P'$
is injective, we deduce that 
\[
|\mathcal{P}_{\vec{a},t}|\le|\mathcal{P}_{r-k+1,t-k+1}|=\stirling{r-k+1}{t-k+1}.
\]
Then, note that for all $1\le q\le n$ we have
\[
\stirling nq\le\prod_{i=0}^{n-q-1}\binom{n-i}{2}\le\binom{n}{2}^{n-q},
\]
because we can generate a partition $P\in\mathcal{P}_{n,q}$ by starting
with the trivial partition $P_{0}\in\mathcal{P}_{n,n}$ into $n$
parts, and iteratively merging pairs of parts ($n-q$ times). We deduce
that
\[
|\mathcal{P}_{\vec{a},t}|\le\binom{r-k+1}{2}^{r-t}<\left(\frac{(r-k+1)^{2}}{2}\right)^{r-t}.
\]
Combining this with \cref{eq:functions-partitions}, the desired result follows.
\end{proof}

\begin{proof}
[Proof of \cref{thm:mk-partitioning}]

Let $C$ be a large constant ($C=200$ will do).
Let $R(k)=C\cdot (r-k+1)^2$.

The plan is to repeatedly apply \cref{lem:zhihan} until we find the desired subgraph.
First, let $G_{0}=G$, $p_{0}=p/2$, $n_{0}=n$ and $k_{0}=1$. 
Observe that $G_0$ (like any ordered $r$-graph) is $\vec a$-equipartite for $\vec a=(r)$ (i.e., 1-equipartite). It has order $n$ and $(r-1)$-density at least $(p/n)\binom{n}{r}/(n^{r-1}/r!)>p/2$ (assuming $n$ is sufficiently large in terms of $r$).
Then, for each $i\ge1$ (``at step $i$''), we take the $k_{i-1}$-equipartite subgraph $G_{i-1}$ with $(r-1)$-density $p_{i-1}$, and apply \cref{lem:zhihan} with $R=R(k_{i-1})$ to obtain a
$k_{i}$-equipartite subgraph $G_{i}$ with order $n_{i}=\floor{n_{i-1}/R}$ %\ge(2R(k_{i-1}))^{-i}n$
and with $(r-1)$-density at least
\begin{equation} \label{eq:pi-recurrence}
p_{i}
=\frac{p_{i-1}}{4R} \left(\frac{R}{(r-k_{i-1}+1)^2}\right)^{r-{k_i}}
=\frac{p_{i-1}}{4(r-k_{i-1}+1)^{2}}C^{r-k_{i}-1},
\end{equation}
for some $k_{i}\ge k_{i-1}$. We stop when $k_{i}\ge r-1$ (i.e., when we have found a $(\ge r-1)$-equipartite subgraph), and we also abort the process if at any point $G_{i-1}$ has too few vertices to apply \cref{lem:zhihan} with
the desired $R$. Let $\tau$ be the total number of steps taken.

We first prove that $p_{j}\ge p/(8Cr^{2})$ for $j\leq \tau$. In fact, we prove, by reverse induction, the stronger statement that for each
step $i<j$ we have
\[p_{j}\ge \frac{p_{i}}{4C(r-k_{i}+1)^{2}}.\]
The base case (where $i=j-1$) follows from \cref{eq:pi-recurrence}, so consider some $i<j-1$ and suppose as our induction hypothesis that the statement is true for $i+1$. Since $i+1 \leq j-1\leq \tau-1$, we have $k_{i+1}\leq r-2$. Hence,  
\[
p_{j}\ge\frac{p_{i+1}}{4C(r-k_{i+1}+1)^{2}}=\frac{p_{i}}{(4C)^2(r-k_{i+1}+1)^{2}(r-k_{i}+1)^{2}}C^{r-k_{i+1}}\ge\frac{p_{i}}{(4C)(r-k_{i}+1)^{2}},
\]
using the induction hypothesis and \cref{eq:pi-recurrence} (for $i+1$), and the fact that $C^{r-k}/(4C(r-k+1)^2)\ge 1$ for $k\le r-2$ and $C \ge 36$. 

It now suffices to show that this process cannot continue for too many steps (in particular, we want to make sure that the  graphs $G_{i}$ never get too small). 
For each $1\le k\le r-2$, let $i(k)$ be the first step $i$ such that $k_{i}\ge k$, and let $N(k)$ be the number of steps $i>i(k)$ such that $k_{i}=k$. 
Note that for each step $i(k)<i\le i(k)+N(k)$ we have $p_{i}\ge(C/40)p_{i-1}$ by \cref{eq:pi-recurrence} (because $C^{r-k}(r-k+1)^{-2}\ge C^2/10$ for $k\le r-2$) and $n_{i}\le n_{i-1}/R(k)$.
This gives $p_{i(k)+N(k)} \ge (C/40)^{N(k)} p_{i(k)}$ and $n_{i(k)+N(k)}\le n_{i(k)}/R(k)^{N(k)}$.
But then, due to \cref{def:equipartite}, the $(r-1)$-density of $G_{i(k)+N(k)}$ is at most $n_{i(k)+N(k)}$, so $p_{i(k)+N(k)} \le n_{i(k)+N(k)}$.
This, plus the fact that $p_j\ge p/(8Cr^2)\ge \varepsilon/(8Cr^2)$ for $j \le \tau$, implies
$$
    (C/40)^{N(k)}\cdot\varepsilon/(8Cr^2) 
    \le (C/40)^{N(k)} p_{i(k)} \le p_{i(k)+N(k)} \le n_{i(k)+N(k)} \le n_{i(k)}/R(k)^{N(k)}.
$$
So we must have 
$N(k) \le \log\left(8Cr^2 n_{i(k)}/\varepsilon\right) / \log\left( R(k)\,C/40 \right)
    \le \log(n_{i(k)})/\log(4R(k))$,
where in the latter inequality we assume $n_{i(k)}$ is sufficiently large in terms of $r$ and $\varepsilon$ and $C$ is a large constant.
Consequently, $n_{i(k)+N(k)}\ge (2R(k))^{-N(k)}\cdot n_{i(k)}\ge n_{i(k)}^{\alpha(k)}$, where 
$\alpha(k)=1-\log(2R(k))/\log(4R(k))>0$, and $n_{i(k+1)}=n_{i(k)+N(k)+1}\ge n_{i(k)+N(k)}/(2R(k)) \ge n_{i(k)}^{\alpha(k)}/(2R(k))$.

% It now suffices to show that this process cannot continue for too many steps (in particular, we want to make sure that the $G_{i}$ never get too small). 
% For each $k$, let $i(k)$ be the first step $i$ such that $k_{i}\ge k$, and let $N(k)$ be the number of steps $i>i(k)$ such that $k_{i}=k$. 
% Note that for each step $i(k)<i\le i(k)+N(k)$ we have $p_{i}\ge(C/40)p_{i-1}$ by \cref{eq:pi-recurrence}, since $C^{r-k}(r-k+1)^{-2}\ge C^2/10$ for $k\le r-2$. 
% So, $p_{i(k)+\Delta}\ge (C/40)^{\Delta} p_{i(k)}$ for $\Delta\le N(k)$. On the other hand, note
% that for each such $\Delta$, the $(r-1)$-density of $G_{i(k)+\Delta}$ is
% in the range between $p/(4Cr^2) \ge \varepsilon/(4Cr^2)$ and $n_{i(k)+\Delta}\le n_{i(k)}/R(k)^{\Delta}$,
% so we must have $N(k)\le\log((4Cr^2)n_{i(k)}/\varepsilon)/\log(R(k)C/40)\le\log(n_{i(k)})/\log(4R(k))$ (assuming $n_{i(k)}$ is sufficiently large in terms of $r$ and $\varepsilon$). This means that $n_{i(k)+N(k)}\ge (2R(k))^{-N(k)}n_{i(k)}\ge n_{i(k)}^{\alpha(k)}$,
% where 
% \[
% \alpha(k)=1-\frac{\log(2R(k))}{\log(4R(k))}>0.
% \]

So, $n_{i}$ is at least $n^{\alpha(1)\dots\alpha(r-2)}/(2^{r-2}R(1)\dots R(r-2))$ for each $i$. Note that $\alpha(1)\dots\alpha(r-2)>0$ and $R(1)\dots R(r-2)$ only depend on $r$, %\mk{someone had previously written $\alpha(k)<1$, I think you meant to write $>0$?}
so if $n$ is sufficiently large in terms of $r$ then we can guarantee that each $n_i$ is always large enough (in terms of $r$ and $R(k_{i})$) to apply \cref{lem:zhihan} (i.e., the process never aborts).
%(Recall that  since $G_{i-1}$ always has enough (i.e., some fixed power of $n$) vertices to apply
%\cref{lem:zhihan}.
The process therefore terminates with the $(\ge r-1)$-equipartite subgraph $G_{\tau}$ with $(r-1)$-density $p_{\tau}\geq p/(8Cr^2)$. This concludes the proof.
\end{proof}

Now we prove \cref{thm:mk-partitioning-best-possible}.
\begin{proof}[Proof of \cref{thm:mk-partitioning-best-possible}]
    Let $G$ be the ordered $r$-graph on the vertex set $\{1,\dots,n\}$, with every possible edge that contains two consecutive vertices.
    By \cref{lemma: count the number of r-tuples with given gap}, the number of edges in $G$ is $$\binom{n}{r}-\binom{(n-r+1)_+}{r}=(1+o(1))r(r-1)\frac{n^{r-1}}{r!}.$$
    \begin{compactenum}
        \item If $G'$ is an $r$-equipartite subgraph with parts $I_1<\dots<I_r$ of size $m$, then for every edge $e$ there must be some pair of parts $I_i,I_{i+1}$ such that $e$ contains the last element of $I_i$ and the first element of $I_{i+1}$ (and these two elements must be consecutive). So, the number of edges in $G'$ is at most $(r-1)m^{r-2}$, meaning that $G'$ has $(r-1)$ density at most $(r-1)/m=o_{m\to\infty}(1)$.
        \item \smallskip If $\vec a$ satisfies $\|\vec a\|_0=r-1$ and $G'$ is an $\vec a$-equipartite subgraph with parts $I_1<\dots<I_{r-1}$ of size $m$, then we can divide the edges of $G'$ into two types.
            \begin{compactitem}
            \item First, we could have some edges $e$ containing the last element of some $I_i$ and the first element of $I_{i+1}$, as in (1). The number of such edges is at most $(r-2)m^{r-2}$.
            \item Second, writing $j$ for the single index with $a_j=2$, we could have some edges $e$ containing two consecutive vertices in $I_j$. The number of such edges is at most $(m-1)m^{r-2}$.
            \end{compactitem}
        All in all, the number of edges in $G'$ is at most $(1+o_{m\to \infty}(1))m^{r-1}$, meaning that $G'$ has $(r-1)$-density at most $2+o_{m\to \infty}(1)$.\qedhere
    \end{compactenum}
\end{proof}

%% file: extremal.tex
\section{Extremal numbers for \texorpdfstring{$P$}{}-cliques}\label{sec:extremal}
In this section, we prove \cref{theorem: turan number for m=2,thm:extremal-clique}. First, the following theorem is a generalisation of \cref{thm:extremal-clique}(1). It will also be used for the lower bound in \cref{theorem: turan number for m=2}.
%Indeed, we prove the following lower bound for all $r$-pattern $P$ ($P$ is not necessarily $r$-partite).
% \zj{For non-$r$-partite $r$-patterns, this bound is loose, but maybe useful when $n$ is small. At least, this will make the \cref{claim: turan number is not small} nicer and the length of the proof is not longer.}

\begin{theorem}\label{theorem: lower bound for all patterns}
  Let $r,n\ge 1$, let $P$ be an $r$-pattern, and let $m\ge 1$.
  % Suppose that $P$ is collectable, or that $m\in\{1,2\}$.
%  Suppose that $m \ge 1$, with the restriction that $m \in\{1,2\}$ if $P$ is non-collectable.
  If $P$ is collectable or $m \in \{1,2\}$
  %($P$ can be non-collectable in this case),
  then $$\ex_<(n,\clique{P}{m}) \ge \binom{n}{r}-\binom{(n-r(m-1))_+}{r}.$$
\end{theorem}
(Note that \cref{theorem: lower bound for all patterns} allows $P$ to be non-collectable if $m\in \{1,2\}$. We do not need to worry about non-collectable patterns in our proofs of \cref{theorem: turan number for m=2,thm:extremal-clique} the extra generality is so that \cref{claim: turan number is not small} (which we stated without proof in \cref{sec:counting}) is a direct corollary.
\begin{proof}
  Note that $\ex_<(n,\clique{P}{1})=0$ and $\ex_<(n,\clique{P}{m})=\binom{n}{r}$ if $rm > n$.
  We may assume $2 \le m \le n / r$ for the rest of the proof.
  
  Write $f_1,f_2$ for the two edges of $P$, with $f_1[1]<f_2[1]$.
  For each $1\le i < r$, let $s_i=\#(j\in[r]:f_1[i]<f_2[j]<f_1[i+1])$ be the number of vertices of $f_2$ lying between the $i$-th and $(i+1)$-th vertices of $f_1$. 
  Let $s_r=\#(j\in[r]:f_2[j]>f_1[r])$ be the number of vertices of $f_2$ after all the vertices of $f_1$.
  Clearly, $s_1+\dots+s_r=r$.

  Let $G$ be the ordered $r$-graph with vertex set $\{1,2,\dots,n\}$, where we include an edge $e$ whenever $e[i+1]-e[i] \le s_i\cdot(m-1)$ for some $1\le i<r$ or $e[r] \ge n-s_r\cdot(m-1)+1$.
  We first claim that $G$ is $\clique P m$-free. 
  Indeed, suppose for the purpose of contradiction that the edges $e_1,\dots,e_m$ in $G$ form $\clique P m$, with $e_1[1]<\dots<e_m[1]$.
  Fix $2 \le j \le m$. 
  As $e_1,e_j$ form the pattern $P$, we know that for all $1 \le i < r$, there are $s_i$ vertices of $e_j$ lying between $e_1[i],e_1[i+1]$, and that there are $s_r$ incident vertices of $e_j$ lying after $e_1[r]$.
  By considering all $j$, it holds that $e_1[i+1]-e_1[i]>s_i\cdot(m-1)$ for $1 \le i < r$ and $e_1[r] \le n-s_r\cdot(m-1)$. But then, $e_1$ should not have been put into $G$, a contradiction.

  Now, it suffices to count the number of edges in $G$. 
  It is convenient to instead count \emph{non-edges}: an $r$-subset $e$ of $\{1,\dots,n\}$ is a non-edge of $G$ if and only if $e[i+1]-e[i]>s_i\cdot(m-1)$ for all $1 \le i < r$ and $e[r]\le n-s_r\cdot(m-1)$.
  By \cref{lemma: count the number of r-tuples with given gap}, the number of such $e$ is exactly $$\binom{n-s_r(m-1)-\sum_{i=1}^{r-1}s_i(m-1)}{r}=\binom{n-r(m-1)}{r}.$$
  Thus, the number of edges in $G$ is
  \[\binom{n}{r}-\binom{n-r(m-1)}{r},\]
  so this is a lower bound on $\ex_<(n,\clique{P}{m})$, as desired.
\end{proof}
\begin{remark}
  In fact, the above lower bound works in a stronger setting: given any $r$-pattern $P$ and any $r$-matching $Q$ with edges $e_1,\dots,e_m$ (ordered such that $e_1[1]<e_2[1]<\dots<e_m[1]$), such that $e_1,e_i$ form the pattern $P$ for all $2\le i\le m$, we have $$\ex_<(n,Q) \ge \binom{n}{r}-\binom{(n-r(m-1))_+}{r}.$$
\end{remark}

We next prove \cref{thm:extremal-clique}(3).    
\begin{proof}[Proof of \cref{thm:extremal-clique}(3)]
  Without loss of generality, we may assume that $m \le n/r$, as otherwise $\ex_<(n,\clique{P}{m})=\binom{n}{r}$, as desired.
  Let $P$ be the ``alternating pattern'' represented by $|\AB|\BA|\AB|\BA|\cdots|$, and let $\tilde G$ be the ordered $r$-graph in the proof of \cref{thm:extremal-clique}(1) above, i.e., $e$ is an edge of $G$ if $e[2i]-e[2i-1] \le 2(m-1)$ for some $1 \le i \le r/2$ or if $e[r]\ge n-(m-1)+1$ and $r$ is odd.
  The number of edges in $\tilde{G}$ is $$\binom{n}{r}-\binom{n-r(m-1)}{r}.$$
  Consider any $n$-vertex $\clique P m$-free ordered $r$ graph $G$.
  We are going to show that $e(G) \le e(\tilde{G})$.
  Our strategy is to partition the collection of all possible edges into subsets, and to separately show that in each subset the number of edges of $G$ is at most the number of edges of $\tilde{G}$.

  Let $E(K_n^{(r)})$ be the collection of all subsets of $\{1,\dots,n\}$ (i.e., all edges of the complete $r$-graph). 
  We define an equivalence relation $\sim$ on $E(K_n^{(r)})$: for $e,f\in E(K_n^{(r)})$, we write $e\sim f$ if there is $\ell\in \mb Z$ such that $f[i]=e[i]+\ell$ whenever $i$ is odd, and $f[i]=e[i]-\ell$ whenever $i$ is even. 
  Fix an equivalence class $A$ of $\sim$.
  It is easy to see that any two edges in $A$ form the pattern $P$, and thus all the edges (in $K_n^{(r)}$) in $A$ form a $P$-clique\footnote{This is the crucial fact we are using about the alternating pattern (for all other patterns $P$, one can define similar equivalence classes, but it is not true that all equivalence classes are $P$-cliques).}. 
  
  Now, it suffices to show that $\abs{E(G)\cap A} \le |{E(\tilde{G})\cap A}|$ for every equivalence class $A$ of $\sim$.
  Suppose there exists some $e \in A\cap (E(G)\setminus E(\tilde G))$ (otherwise $\abs{E(G)\cap A} \le |{E(\tilde{G})\cap A}|$ holds trivially). Since $e\notin E(\tilde G)$, we have $e[2i]-e[2i-1]> 2(m-1)$ for all $1 \le i\le r/2$, and $e[r]\le n-(m-1)$ if $r$ is odd.
  For $\ell\in \mb N$, let
    \[f_\ell(e):=\{e[1]+\ell,e[2]-\ell,e[3]+\ell,e[4]-\ell,\dots\}.
    \]
  As we increase $\ell$, we are ``pushing each $e[2i-1]$ and $e[2i]$ towards each other'', and ``pushing $e[r]$ towards $n$'' (if $r$ is odd). 
  At some point, there will be a ``collision'': let $\ell^\star$ be the ``time just before this collision'', i.e., the largest $\ell$ for which $e[r] +(-1)^{r+1}\ell \le n$ and $e[2i-1]+\ell< e[2i]-\ell$ for all $1 \le i \le r/2$.
  Note that $\ell^\star\ge m-1$ (since $e\notin E(\tilde G)$) and  that $f_{\ell}(e)\in E(\tilde G)$ for all $\ell\in \{\ell^\star-(m-2),\dots,\ell^\star\}$ (the condition for $f$ to be an edge of $\tilde G$ is the condition that $f$ is ``within $m-1$ steps of a collision''). Note that $f_{\ell}(e)\sim e$ for each $\ell\le \ell^\star$, so $|{E(\tilde{G})\cap A}|\ge m-1$. 
  On the other hand, $E(G)\cap A$ is a $P$-clique as every two edges in $A$ form the pattern $P$. Since $G$ is $\clique{P}{m}$-free, we have $\abs{E(G)\cap A} \le m-1$. This proves that $\abs{E(G)\cap A} \le |{E(\tilde{G})\cap A}|$, as desired.
\end{proof}

Now, we prove \cref{theorem: turan number for m=2} (giving the exact extremal numbers for all $r$-partite $r$-patterns $P$).

\begin{proof}[Proof of \cref{theorem: turan number for m=2}]
  We only need to prove the upper bound, as the lower bound is given by the $m=2$ case of \cref{thm:extremal-clique}(1).
  
  We proceed by induction on $r$ (the base case $r=1$ is trivial). Fix an $r$-partite $r$-pattern $P$, and write $|\block_1|\cdots|\block_r|$ for its block representation, where $\block_1 = \AB$ and $\block_i\in\{\AB,\BA\}$ for $2 \le i \le r$. Let $G$ be an $n$-vertex $P$-free ordered $r$-graph.  We may assume that $n \ge r$; our goal is to prove that
  \[e(G)\le \binom{n}{r}-\binom{n-r}{r}.\]

  Let $\hat P$ be the pattern with block representation $|\block_1|\block_2|\cdots|\block_{r-1}|$, and for an edge $e=\{e[1],\dots,e[r]\}\in E(K_n^{(r)})$, let $\hat e=\{e[1],\dots,e[r-1]\}\in E(K_{n-1}^{(r-1)})$.
  Let $G'$ be obtained from $G$ as follows. For each $f\in E(K_{n-1}^{(r-1)})$, consider all edges $e\in E(G)$ with $\hat e=f$. If there are any such edges, then delete the one with the largest value of $e[r]$. Note that $e(G') \ge e(G)-\binom{n-1}{r-1}$.
  
  Then, for $v \in \{1,\dots,n\}$, let $\hat G_v$ be the graph with vertex set $\{1,\dots,v-1\}$ and edge set \[\{\hat e: e \in E(G'), e[r]=v\}.\] We claim that $\hat G_v$ is $\hat P$-free. To see this, suppose for the purpose of contradiction that $\hat e_1,\hat e_2 \in E(\hat G_v)$ form pattern $P'$ with $e_1[1]<e_2[1]$.
  By the definition of $G'$, there are $x, y\in\{v+1,\dots,n\}$ such that $\{e_1[1],\dots,e_1[r-1],x\}, \{e_2[1],\dots,e_2[r-1],y\} \in E(G)$. But note that
  \begin{compactitem}
      \item if $\block_r = \AB$ then $\{e_1[1],\dots,e_1[r-1],v\},\{e_2[1],\dots,e_2[r-1],y\}$ form pattern $P$;
      \item if $\block_r = \BA$ then $\{e_1[1],\dots,e_1[r-1],x\},\{e_2[1],\dots,e_2[r-1],v\}$ form pattern $P$.
  \end{compactitem}
  In either case, we have found $P$ in $G$, which is a contradiction.
  So, each $\hat G_v$ is $\hat P$-free and by induction we have
  \[
    e(\hat G_v)
    \le \binom{v-1}{r-1}-\binom{(v-1-(r-1))_+}{r-1}
    = \binom{v-1}{r-1}-\binom{(v-r)_+}{r-1}.
  \]
  Also, note that $e(\hat G_n)=0$, because any $e \in E(G)$ with $e[r]=n$ must have been deleted in the construction of $G'$.
  Consequently, 
  \begin{equation} \nonumber
    \begin{aligned}
      e(G) &\le \binom{n-1}{r-1} + e(G')=\binom{n-1}{r-1} + \sum_{v=1}^{n-1}e(\hat G_v)\\
      &\le \binom{n-1}{r-1} + \sum_{v=1}^{n-1} \left[\binom{v-1}{r-1}-\binom{(v-r)_+}{r-1} \right]\\
      &\le \binom{n-1}{r-1} + \sum_{k=0}^{n-2} \binom{k}{r-1} - \sum_{k=0}^{n-r-1} \binom{k}{r-1}\\
      &= \binom{n-1}{r-1} + \binom{n-1}{r} - \binom{n-r}{r}
      =\binom{n}{r}-\binom{n-r}{r},
    \end{aligned}
  \end{equation}
  as desired.
\end{proof}

The last result we prove in this section is \cref{thm:extremal-clique}(2). Starting from a $\clique{P}{m}$-free $r$-graph $G$, we will apply \cref{thm:mk-partitioning} to pass to a $(\ge r-1)$-equipartite subgraph (whose $(r-1)$-density is comparably large, and which is still $\clique{P}{m}$-free).
To this end, we need a generalisation of $\ex_<(n,F)$.
\begin{definition}
Let $r \ge 1$ and $\vec{a}$ be a positive integer-valued vector with sum of entries $\|\vec{a}\|_{1}=r$.
For an ordered $r$-graph $H$, we write $\ex_<^{\vec{a}}(n, H)$ for the maximum number of edges in an $H$-free $\vec{a}$-equipartite $r$-graph of order $n$.
% Consider a positive integer vector $\vec{a}\in\bigcup_{\ell=1}^{r}\mb Z_{>0}^{\ell}$ with sum of entries $\|\vec{a}\|_{1}=r$, and consider an ordered $r$-graph $H$. We write $\ex_<^{\vec{a}}(n, H)$ for the maximum number of edges in an $H$-free $\vec{a}$-equipartite $r$-graph with order $n$.
\end{definition}
We separately consider the $r$-equipartite ($\|\vec a\|_0=r$) and $(r-1)$-equipartite ($\|\vec a\|_0=r-1$) cases. First, we consider the $r$-equipartite case.
\begin{lemma}\label{lemma: turan number for r-partite graphs}
  Let $n,m,r \ge 1$.
  Let $\vec{a}=(1,1,\dots,1)$ be the all-one vector of length $r$, and $P$ be an $r$-partite $r$-pattern.
  Then,
  \[\ex_<^{\vec{a}}(n,\clique{P}{m})=n^r-\big((n-m+1)_+\big)^r\le r(m-1)n^{r-1}.\]
\end{lemma}
\begin{proof}
\newcommand{\turan}{{\ex_<^{\vec{a}}}}
Let $V_{1},\dots,V_{r}$ be disjoint copies of $\{1,\dots,n\}$, and
let $\mathcal{V}=V_{1}\times\dots\times V_{r}$. Note that an $\vec{a}$-equipartite
$r$-graph $G$ can be viewed as a set of tuples $e\in\mathcal{V}$.
From this perspective, a copy of $\clique P m$ (where $P$ has block representation
$|\block_1|\block_2|\cdots|\block_r|$) corresponds to a sequence of edges $e_{1},\dots,e_{m}\in E(G)$
which is increasing in all coordinates $i$ with $\block_{i}=\AB$, and
decreasing in all coordinates $i$ with $\block_{i}=\BA$. 
By symmetry, $\turan(n,\clique P m)$ depends only on $n$ and $m$ (i.e., it does not depend on $P$).
Therefore, it suffices to consider the case where $P$ has representation $|\AB|\AB|\cdots|\AB|$ (i.e., $P$ is
a ``generalised crossing'').

  For the lower bound, we simply consider the $\vec{a}$-equipartite $r$-graph $\tilde{G}$ consisting of all possible edges $e\in\mathcal{V}$ such that some $e[i]<m$. By the above discussion, any copy of $P^{(m)}$ must contain an edge with all coordinates at least $m$.
  Thus, $\tilde G$ is $\clique P m$-free, and it has the desired number of edges.

For the upper bound, consider a $\clique P m$-free $n$-vertex graph $G$ with
parts $V_{1},\dots,V_{r}$. Define an equivalence relation $\sim$ on $\mathcal{V}$ by taking
$e\sim f$ when there is some $\ell\in\mathbb{Z}$ such that each $e[i]=f[i]+\ell$. We proceed very similarly to the proof of \cref{thm:extremal-clique}(3), showing that $|A\cap E(G)|\le|A\cap E(\tilde{G})|$
for each equivalence class $A$ of $\sim$.

Consider some equivalence class $A$, and suppose there is $\ensuremath{e\in A\cap(E(G)\setminus E(\tilde{G}))}$
(otherwise trivially $|A\cap E(G)|\le|A\cap E(\tilde{G})|$). Since
$e\notin E(\tilde{G})$, each $e[i]\ge m$. Define $f_{\ell}(e)=(e[1]-\ell,\dots,e[r]-\ell)$,
and let $\ell^\star=\min_{i}e[i]-1\ge m-1$. Note that $f_{\ell}(e)\in E(\tilde{G})$
for all $\ell\in\{\ell^\star-(m-2),\dots,\ell^\star\}$, and note that
$f_{\ell}(e)\sim e$ for each $\ell\le\ell^\star$. So, $|A\cap E(\tilde{G})|\ge m-1$.
On the other hand, note that $A\cap E(G)$ is a $P$-clique so $|A\cap E(G)|\le m-1$.
This proves that $|A\cap E(G)|\le|A\cap E(\tilde{G})|$, as desired.
\end{proof}
For the $(r-1)$-partite case, we will make a reduction to the setting where $r=2$.
As mentioned in the introduction, in this setting the extremal number is known exactly, as follows.
\begin{theorem}\label{theorem: turan number for r=2}
Let $n, m \ge 1$ and let $P$ be the crossing 2-pattern or the nesting 2-pattern.
Then, \[\ex_<(n, \clique{P}{m}) = \binom{n}{2}-\binom{(n-2(m-1))_+}{2}\le 2(m-1)n.\]
\end{theorem}
The crossing case ($|\AB|\AB|$) of \cref{theorem: turan number for r=2} is due to Capoyleas and Pach~\cite{CP92} while the nesting case ($|\AB|\BA|$) is a special case of \cref{thm:extremal-clique}(3). (The nesting case also follows from results on queue-numbers of graphs due to Pemmaraju~\cite{Pemmaraju92} and to Dujmovi\'c and Wood~\cite{DW04}; see \cite[Lemma~8]{DW04}).
Now, our result in the $(r-1)$-partite setting is as follows.
\begin{lemma}\label{lemma: turan number for r-1 partite}
  Let $n,m,r \ge 1$. Let $\vec{a}=\{1,2\}^{r-1}$ be a vector with exactly one ``2'' (so all other entries are ``1'').
  Let $P$ be an $r$-partite $r$-pattern.
  Then, 
  \[\ex_<^{\vec{a}}(n,\clique{P}{m})\le 4r(m-1)n^{r-1}.\]
\end{lemma}
\begin{proof}
  Fix an $r$-partite $r$-pattern $P$, and write $|\block_1|\cdots|\block_r|$ for its block representation.
  Let $V_{1},\dots,V_{r-1}$ be disjoint copies of $\{1,\dots,n\}$, and let $i^{\star}$ be the unique index such that $a_{i^{\star}}=2$.
  Note that an $\vec a$-equipartite $r$-graph $G$ can be viewed as a set of tuples
  \[
    e\in V_{1}\times\dots\times V_{i^{\star}-1}\times E(K_{n})\times V_{i^{\star}+1}\times\dots\times V_{r-1},
  \]
where $K_{n}$ is the complete 2-graph on the vertex set $V_{i^{\star}}$.
From this perspective, a copy of $\clique P m$ corresponds to a sequence of
edges $e_{1},\dots,e_{m}\in E(G)$ which is increasing in all coordinates
$i\ne i^{\star}$ with $\block_{i}=\AB$, and decreasing in all coordinates
$i\ne i^{\star}$ with $\block_{i}=\BA$, and in coordinate $i^{\star}$,
the corresponding 2-edges have a consistent 2-partite pattern (the possibilities can be represented as $\AB\AB$, $\BA\BA$, $\AB\BA$ and $\BA\AB$, i.e., a crossing-clique going left, a crossing-clique going right, a nesting-clique going inwards, or a nesting-clique going outwards).%\mk{i rephrased this, because the way we defined the word ``pattern'', ABAB means the same thing as BABA}

%(they comprise a crossing-clique going left, a crossing-clique going right, a nesting-clique going inwards, or a nesting-clique going outwards).
 By symmetry, we may assume that $i^{\star}=r-1$. In addition, by exchanging the ``$\blockA$''s and the ``$\blockB$''s in each $\block_i$, if necessary, we may assume that $|\block_{r-1}|\block_r|=|\AB|\AB|$ or $|\AB|\BA|$, and finally (as in \cref{lemma: turan number for r-partite graphs}) we may assume that $\block_{i}= \AB$ for $i\leq r-2$.
% By symmetry,
%(as in the proof of \cref{lemma: turan number for r-partite graphs}), we may assume that $i^\star=r-1$, $|\block_{r-1}|\block_r|\in\{|\AB|\AB|,|\AB|\BA|\}$, and $\block_i=\AB$ for all $i\in[r-2]$.
In other words, $P$ has block representation $|\AB|\AB|\dots|\AB|\AB|$ or $|\AB|\AB|\dots|\AB|\BA|$.

Now, we want to reduce the upper bound of $\ex_<^{\vec{a}}(n,\clique{P}{m})$ to the graph ($r=2$) case and apply \cref{theorem: turan number for r=2}.
Fix an $\vec{a}$-equipartite $\clique P m$-free $r$-graph $G$ with parts $V_{1},\dots,V_{r-1}$. 
Write $Q$ for the 2-partite 2-pattern with block representation $|\block_{r-1}|\block_{r}|$.
Observe that if $G$ has two edges $e,f$ such that $e[i]-e[r-1]=f[i]-f[r-1]$ for all $i\in\{1,\dots,r-2\}$, and such that the last two coordinates of $e$ and the last two coordinates of $f$ (both are viewed as edges in $K_n$) form pattern $Q$, then $e,f$ form pattern~$P$.

Now, for each $\vec{x}=(x_{1},\dots,x_{r-2})\in\{1-n,2-n,\dots,n-1\}^{r-2}$, we define $G_{\vec{x}}$ to be the subgraph of $G$ consisting of all edges $e\in E(G)$ such that $e[i]-e[r-1]=x_i$ for all $i\in\{1,\dots,r-2\}$ (i.e., the edges of the form $(x_{1}+u,\,x_2+u,\,\dots,\,x_{r-2}+u,\,uv)$ with $u<v$).
% $\vec{x}=(x_{1},\dots,x_{r-2})\in(\mb Z\cap (-n,n))^{r-2}$ (i.e., each integer vector with coordinates between $1-n$ and $n-1$) we define $G_{\vec{x}}$ to be the subgraph of $G$ consisting of all edges $e\in E(G)$ such that $e[i]-e[r-1]=x_i$ for all $i\in\{1,\dots,r-2\}$ (i.e., the edges of the form $(x_{1}+u,\,x_2+u,\,\dots,\,x_{r-2}+u,\,uv)$ with $u<v$).
Write $G_{\vec{x}}^{(2)}$ for the $2$-graph consisting of all $uv$ such that $uv$ appears as the last coordinate of some edge in $G_{\vec{x}}$.
Note that there is a natural bijection between edges in $G_{\vec{x}}$ and edges in $G_{\vec{x}}^{(2)}$, and by the above discussion, any $Q$-clique in $G_{\vec{x}}^{(2)}$ corresponds to a $P$-clique in $G_{\vec{x}}$ of the same size.
Thus, $G_{\vec{x}}^{(2)}$ is $\clique{Q}{m}$-free.
Also, as  $G=\bigcup_{\vec{x}}G_{\vec{x}}$, we know $e(G)\le\sum_{\vec{x}}e(G_{\vec{x}})=\sum_{\vec{x}}e(G_{\vec{x}^{(2)}})$.
% that for two edges $e,f \in E(G_{\vec{x}})$ specified by $uv, xw\in E(G_{\vec{x}}^{(2)})$, respectively, if $uv$ and $xw$ form pattern $Q$
% Note that $G=\bigcup_{\vec{x}}G_{\vec{x}}$. 
% as $G_{\vec{x}}$ is $\clique P m$-free we have that $G_{\vec{x}}^{(2)}$ is $\clique{Q}{m}$-free.
% , for the 2-partite 2-pattern $Q$
% with block representation $|\block_{r-1}|\block_{r}|$.

% For each $\vec{x}=(x_{1},\dots,x_{r-2})\in(\mb Z\cap (-n,n))^{r-2}$ (i.e., each integer vector with coordinates between $1-n$ and $n-1$), let $G_{\vec{x}}$ be the $r$-graph consisting of all edges
% $(x_{1}+u,\,x_2+u,\,\dots,\,x_{r-2}+u,\,uv)\in E(G)$ with $u<v$, and $G_{\vec{x}}^{(2)}$ be the $2$-graph consisting of the projections of all edges of $G_{\vec{x}}$ to their last $2$ coordinates. 
% Note that $G=\bigcup_{\vec{x}}G_{\vec{x}}$. 
% Also, as $G_{\vec{x}}$ is $\clique P m$-free we have that $G_{\vec{x}}^{(2)}$ is $\clique{Q}m$-free, for the 2-partite 2-pattern $Q$
% with block representation $|\block_{r-1}|\block_{r}|$.

Then, \cref{theorem: turan number for r=2} implies $e(G_{\vec{x}})=e(G_{\vec{x}}^{(2)})\le2(m-1)n$ for all possible $\vec{x}$.
Na\"ively, there are $(2n-1)^{r-2}$ possible choices for $\vec{x}$,
but we can do better. Note that for each edge $uv\in G_{\vec{x}}^{(2)}$,
we must have $\max_{i}x_{i}+u\le n$ and $\min_{i}x_{i}+u\ge 1$. So,
for $G_{\vec{x}}$ to be nonempty, we must have $\max_{i}x_{i}-\min_{i}x_{i}<n$.
The number of choices of $\vec{x}$ which satisfy this inequality
is at most $(2n-1)rn^{r-3}<2rn^{r-2}$ (enumerating over all choices
of $\min_{i}x_{i}$, and all choices of an $i$ for which $x_{i}$
is minimal). So,
\[
e(G)\le 2rn^{r-2}\cdot  2(m-1)n =4r(m-1)n^{r-1},
\]
and the desired result follows.
%
%Now, fix an $\vec{a}$-equipartite $P$-free $r$-graph $G$ with parts $V_{1},\dots,V_{r-1}$. For each $\vec{x}=(x_{1},\dots,x_{r-2})\in(\mb Z\cap (-n,n))^{r-2}$ (i.e., each integer vector with coordinates between $1-n$ and $n-1$), let $G_{\vec{x}}$ be the $r$-graph consisting of all edges
%\begin{itemize}
%\item $(x_{1}+u,\,x_2+u,\,\dots,\,x_{r-2}+u,\,uv)\in E(G)$, if $\block_{r}=\block_{r-1}$, or;
%\item $(x_{1}-u,\,x_2-u,\,\dots,\,x_{r-2}-u,\,uv)\in E(G)$, if $\block_{r}\ne \block_{r-1}$,
%\end{itemize}
%and $G_{\vec{x}}^{(2)}$ be the $2$-graph consisting of the projections of all edges of $G_{\vec{x}}$ to their last $2$-coordinates. Note that $G=\bigcup_{\vec{x}}G_{\vec{x}}$. Also, as $G_{\vec{x}}$ is $P$-free we have that $G_{\vec{x}}^{(2)}$ is $\clique{P'}m$-free, for the 2-partite 2-pattern $P'$
%with block representation $|\block_{r-1}|\block_{r}|$.
%
%Now, \cref{theorem: turan number for r=2} tells us that each $e(G_{\vec{x}}^{(2)})\le2(m-1)n.$
%Na\"ively, there are $(2n-1)^{r-2}$ possible choices for $\vec{x}$,
%but we can do better. Note that for each edge $uv\in G_{\vec{x}}^{(2)}$,
%we must have $\max_{i}x_{i}+u\le n$ and $\min_{i}x_{i}-u\ge0$. So,
%for $G_{\vec{x}}$ to be nonempty, we must have $\max_{i}x_{i}-\min_{i}x_{i}<n$.
%The number of choices of $\vec{x}$ which satisfy this inequality
%is at most $(2n-1)rn^{r-3}=2n^{r-2}$ (enumerating over all choices
%of $\min_{i}x_{i}$, and all choices of an $i$ for which $x_{i}$
%is minimal). So,
%\[
%e(G)\le4r(m-1)n,
%\]
%and the desired result follows.
\end{proof}
\begin{remark}
  Let $i^\star$ be the only $i\in[r-1]$ with $a_i=2$ and let $|\block_1|\block_2|\cdots|\block_r|$ be the block representation of $P$.
  Similarly to the proof of \cref{thm:extremal-clique}(3), one can show that if $\block_{i^\star}\neq \block_{i^\star+1}$, then 
  \[\ex_<^{\vec{a}}(n,\clique{P}{m})=n^{r-2}\binom{n}{2}-(n-(m-1))^{r-2}\binom{(n-2(m-1))_+}{2}.\]
  We wonder if a more general statement is true.
  Let $n,m,r \ge 1$, let $P$ be an $r$-partite $r$-pattern, and let $\vec{a}$ be a vector of positive integers with sum of entries $\|\vec{a}\|_{1}=r$.
  Then, is it always the case that
  \[\ex_<^{\vec{a}}(n,\clique{P}{m})=\prod_{i=1}^k\binom{n}{a_i}-\prod_{i=1}^k\binom{(n-(m-1)a_i)_+}{a_i}?\]
  (Cf.\ \cref{conjecture: turan numbers are the same}).
\end{remark}

We are finally ready to prove \cref{thm:extremal-clique}(2) (giving a general upper bound on $\ex_<(n,\clique{P}{m})$ using \cref{thm:mk-partitioning} to reduce to the $(\ge r-1)$-equipartite case).
\begin{proof}[Proof of \cref{thm:extremal-clique}(2)]
  Let $G$ be an $n$-vertex $\clique{P}{m}$-free ordered $r$-graph, with $(p/n) \binom n r$ edges. Our goal is to prove that $p=O(r^3(m-1))$.

  By \cref{thm:mk-partitioning}, we can find a $(\ge r-1)$-equipartite subgraph $G'\subseteq G$, of order $n'=n^{\Omega_r(1)}$, which has $(r-1)$-density at least $\Omega(p/r^2)$. Note that $G'$ is still $\clique P m$-free.
\begin{compactitem}
    \item If $G'$ is $r$-partite, then by \cref{lemma: turan number for r-partite graphs}, we have $e(G')\le r(m-1)n^{r-1}$. That is to say, $G'$ has $(r-1)$-density at most $r(m-1)$.
    \item If $G'$ is $(r-1)$-partite, then by \cref{lemma: turan number for r-1 partite}, we have $e(G')\le 4r(m-1)n^{r-1}$. That is to say, $G'$ has $(r-1)$-density at most $8r(m-1)$.
\end{compactitem}
In both cases, $G'$ has $(r-1)$-density $O(r(m-1))$; the desired result follows.
\end{proof}

%% file: appendix.tex
\appendix
\section{Further lower bounds on Ramsey parameters
%\texorpdfstring{$\ramsey r n$}{}
} \label{appendix: ES}

In this appendix we prove that $\ramsey 4 n\ge n^{1/15}/4$, completing the
proof of \cref{thm:ES-3-4}. It is convenient to introduce some notation.
\begin{definition}
For a pair of $r$-patterns $P,Q$, we write $P\circ Q$ for the set of all $r$-patterns that can be obtained as a ``composition'' of $P$ and $Q$. 
Specifically, we put $R\in P\circ Q$ if it is possible to have three edges $e,f,g$ (with $e[1]<f[1]<g[1]$) in an ordered $r$-matching such that $e,f$ form~$P$, $f,g$ form~$Q$ and $e,g$ form~$R$.
% there exists
% an ordered $r$-matching consisting of three edges $e,f,g$ (with
% $e[1]<f[1]<g[1]$) such that $e,f$ form $P$ and $f,g$ form $Q$ and
% $e,g$ form $R$.
\end{definition}

For example, in the case $r=2$, writing $\line$ for the alignment
pattern $\AA\BB$, we clearly have $\line\circ\line=\{\line\}$. As a few more examples, writing $\stack$ for the nesting pattern
$\AB\BA$ and $\wave$ for the crossing pattern $\AB\AB$, it is not hard to see that
$\line\circ\stack=\{\line\}$ and $\stack\circ\line=\{\line,\stack,\wave\}$.
\begin{definition}
Say that a collection of $r$-patterns $\mathcal{P}$ is \emph{left-dominated}
if $(P\circ Q)\cap\mathcal{P}\subseteq\{P\}$ for all $P,Q\in\mathcal{P}$.
Say that $\mathcal{P}$ is \emph{right-dominated} if $(P\circ Q)\cap\mathcal{P}\subseteq\{Q\}$
for all $P,Q\in\mathcal{P}$.

Roughly speaking, the idea is that, if $\mc P$ is left-dominated, then in a $\mc P$-clique, if we compose a pattern~$P$ with any pattern we always get $P$. (Similarly, if $\mc P$ is right-dominated, then when we compose any pattern with $P$ we always get $P$).
For example, as implicitly observed in the proof of the lower bound on $\ramsey 3 n$ in \cref{subsec:further-ES}, the set $\{\psi(\line\stack),\psi(\stack\line)\}$
is left-dominated. The key property of a left- or right-dominated
set $\mathcal{P}$ (which also implicitly appeared in the proof of the lower bound on $\ramsey 3 n$) is that every $\mathcal{P}$-clique always contains a very large
$P$-clique for some $P\in\mathcal{P}$, as follows.
\end{definition}

\begin{lemma}\label{lem:dominated-pigeonhole}
If $\mathcal{P}$ is left-dominated or right-dominated, then for every
$\mathcal{P}$-clique $M$ of size $n$ we have $L(M)\ge n/|\mathcal{P}|$.
\end{lemma}

\begin{proof}
Suppose $\mathcal{P}$ is left-dominated (the right-dominated case
is similar), and let $M$ be a $\mathcal{P}$-clique of size $n$.
Write the edges of $M$ as $e_{1},\dots,e_{n}$ with $e_{1}[1]<\dots<e_{n}[1]$.
Now, for every index $i$, there is a single pattern $P\in\mathcal{P}$
such that $e_{i}$ and $e_{j}$ form $P$ for each $j>i$; say that
$i$ is of ``type $P$''. By the pigeonhole principle, at least
$n/|\mathcal{P}|$ of the indices have the same type (say, $P$);
the corresponding edges form a $P$-clique.
\end{proof}
We also need to introduce a new notion of ``one-sided'' $P$-freeness.
\begin{definition}
Given an $r$-matching $M$ and an $r$-pattern $P$, say that an
edge $e\in E(M)$ is \emph{left-$P$-free} if there is no $f\in E(M)$
with $f[1]<e[1]$ such that $f$ and $e$ form pattern $P$. Similarly,
say that $e$ is \emph{right-$P$-free }if there is no $f\in E(M)$
with $f[1]>e[1]$ such that $e$ and $f$ form pattern $P$.
\end{definition}

Note that if we have a matching $M$ whose every edge is left-$P$-free,
or whose every edge is right-$P$-free, then $M$ itself is $P$-free.

Now we are ready to complete the proof of \cref{eq:ES-3-4}.
\begin{proof}[Proof of the lower bound on $\ramsey{4}{n}$ in \cref{thm:ES-3-4}]
We use the notions of weak patterns and signatures introduced in
\cref{sec:ramsey-lower}. There are 35 different 4-patterns,
including eight non-collectable patterns (which are all given names
in \cref{table: 4-patterns}).

\begin{table}
\centering \setlength{\tabcolsep}{7mm} % \resizebox{\textwidth}{7mm}
\begin{tabular}{|c|c|}
\hline 
Pattern name & Representation\tabularnewline
\hline 
$\psi(\line\line\line)$ & $|\AA\AA|\BB\BB|$\tabularnewline
\hline 
$P_{\line\line\line}^{*}$ & $\AAA\BA\BBB$\tabularnewline
\hline 
$P_{\line\line\line}^{**}$ & $\AA\AB\BA\BB$\tabularnewline
\hline 
$P_{\line\line\line}^{***}$ & $\AA\blockB\AA\BBB$\tabularnewline
\hline 
$P_{\line\line\line}^{****}$ & $\AA\BA\BA\BB$\tabularnewline
\hline 
$\psi(\line\line\stack)$ & $|\AAA\BBB|\BA|$\tabularnewline
\hline 
$P_{\line\line\stack}^{*}$ & $\AA\BA\BB\BA$\tabularnewline
\hline 
$\psi(\line\line\wave)$ & $|\AAA\BBB|\AB|$\tabularnewline
\hline 
$P_{\line\line\wave}^{*}$ & $\AA\BA\BB\AB$\tabularnewline
\hline 
$\psi(\line\stack\line)$ & $|\AA\BB|\BB\AA|$\tabularnewline
\hline 
$\psi(\line\stack\stack)$ & $|\AA\BB|\BA|\AB|$\tabularnewline
\hline 
$\psi(\line\stack\wave)$ & $|\AA\BB|\BA|\BA|$\tabularnewline
\hline 
$\psi(\line\wave\line)$ & $|\AA\BB|\AA\BB|$\tabularnewline
\hline 
$\psi(\line\wave\stack)$ & $|\AA\BB|\AB|\BA|$\tabularnewline
\hline 
$\psi(\line\wave\wave)$ & $|\AA\BB|\AB|\AB|$\tabularnewline
\hline 
$\psi(\stack\line\line)$ & $|\AB|\BBB\AAA|$\tabularnewline
\hline 
$P_{\stack\line\line}^{*}$ & $\AB\BB\AB\AA$\tabularnewline
\hline 
$\psi(\stack\line\stack)$ & $|\AB|\BB\AA|\AB|$\tabularnewline
\hline 
\end{tabular}\enskip{}%
\begin{tabular}{|c|c|}
\hline 
Pattern name & Representation\tabularnewline
\hline 
$\psi(\stack\line\wave)$ & $|\AB|\BB\AA|\BA|$\tabularnewline
\hline 
$\psi(\stack\stack\line)$ & $|\AB|\BA|\AA\BB|$\tabularnewline
\hline 
$\psi(\stack\stack\stack)$ & $|\AB|\BA|\AB|\BA|$\tabularnewline
\hline 
$\psi(\stack\stack\wave)$ & $|\AB|\BA|\AB|\AB|$\tabularnewline
\hline 
$\psi(\stack\wave\line)$ & $|\AB|\BA|\BB\AA|$\tabularnewline
\hline 
$\psi(\stack\wave\stack)$ & $|\AB|\BA|\BA|\AB|$\tabularnewline
\hline 
$\psi(\stack\wave\wave)$ & $|\AB|\BA|\BA|\BA|$\tabularnewline
\hline 
$\psi(\wave\line\line)$ & $|\AB|\AAA\BBB|$\tabularnewline
\hline 
$P_{\wave\line\line}^{*}$ & $\AB\AA\BA\BB$\tabularnewline
\hline 
$\psi(\wave\line\stack)$ & $|\AB|\AA\BB|\BA|$\tabularnewline
\hline 
$\psi(\wave\line\wave)$ & $|\AB|\AA\BB|\AB|$\tabularnewline
\hline 
$\psi(\wave\stack\line)$ & $|\AB|\AB|\BB\AA|$\tabularnewline
\hline 
$\psi(\wave\stack\stack)$ & $|\AB|\AB|\BA|\AB|$\tabularnewline
\hline 
$\psi(\wave\stack\wave)$ & $|\AB|\AB|\BA|\BA|$\tabularnewline
\hline 
$\psi(\wave\wave\line)$ & $|\AB|\AB|\AA\BB|$\tabularnewline
\hline 
$\psi(\wave\wave\stack)$ & $|\AB|\AB|\AB|\BA|$\tabularnewline
\hline 
$\psi(\wave\wave\wave)$ & $|\AB|\AB|\AB|\AB|$\tabularnewline
\hline 
\multicolumn{1}{c}{} & \multicolumn{1}{c}{}\tabularnewline
\end{tabular}

\caption{All 35 different 4-patterns. The collectable patterns are named as
$\psi(W)$ for their corresponding weak pattern $W$ (as in \cref{sec:ramsey-lower}). The
non-collectable patterns are given ad-hoc names with superscript ``$*$''s.}
\label{table: 4-patterns}
\end{table}

Define
\begin{align*}
\mathcal{P}_{1}&=\phi^{-1}(\line\line\line)=\{\psi(\line\line\line),P_{\line\line\line}^{*},&&\!\!\!\!\!P_{\line\line\line}^{**},P_{\line\line\line}^{***},P_{\line\line\line}^{****}\},\quad\mathcal{P}_{2}=\{\psi(\line\line\wave),P_{\line\line\wave}^{*},\psi(\line\wave\line),\psi(\wave\line\line),P_{\wave\line\line}^{*}\},\\
\mathcal{P}_{3}&=\{\psi(\line\line\stack),\psi(\line\stack\line),\psi(\stack\line\line)\},&&\mathcal{P}_{4}=\{P_{\line\line\stack}^{*},P_{\stack\line\line}^{*}\},\quad \mathcal{P}_{5}=\{\psi(\line\wave\wave),\psi(\wave\line\wave),\psi(\wave\wave\line)\},\\
\mathcal{P}_{6}&=\{\psi(\line\stack\stack),\psi(\stack\line\stack),\psi(\stack\stack\line)\}, 
&&\mathcal{P}_{7}=\{\psi(\line\wave\stack),\psi(\line\stack\wave),\psi(\stack\line\wave),\psi(\stack\wave\line),\psi(\wave\line\stack),\psi(\wave\stack\line)\}.
\end{align*}
For each $i\in\{1,\dots,7\}$ we write $\preceq_{i}$ to denote the
relation $\preceq_{\mathcal{P}_{i}}$(see \cref{def:calP-clique} for the definition of $\preceq_{\mc{P}}$). 
Some of these relations give
rise to posets, as follows.
\begin{compactenum}
\item Consider any matching $M$. By the proof of \cref{lem:key-ES}(A):
\begin{compactitem}
\item $\preceq_{1}$ is always a partial order (with corresponding signature
$(3,\emptyset)$);
\item $\preceq_{2}$ is a partial order if $M$ is $\mathcal{P}_{1}$-free
(with corresponding signature $(2,\emptyset)$);
\item $\preceq_{5}$ is a partial order if $M$ is $\mathcal{P}_{1}\cup\mathcal{P}_{2}$-free
(with corresponding signature $(1,\emptyset)$).
\end{compactitem}
\item In any matching, $\preceq_{3}$ is a partial order (this can be observed
by direct case-checking, analogous to $\mathcal{P}_{3}$ in the proof for $r=3$ case of \cref{thm:ES-3-4}; the patterns in $\mathcal{P}_{3}$ can be interpreted as
``generalised nestings'' where one of the two edges of the pattern
is fully contained between two consecutive vertices of the other edge).
\end{compactenum}
Also, according to \cref{lemma: large weak clique implies large clique}, %recall from the proof of \cref{lem:key-ES}(B) that\zj{why?}:
\begin{compactenum}\addtocounter{enumi}{2}
\item  for $i\in\{1,2,5\}$ (which each correspond
to specific signatures), every $\mathcal{P}_{i}$-clique $M$ of size
$n$ has $L(M)\ge n/3$.
\end{compactenum}
Now, consider any matching $M=M_{1}$ of size $n$. We show how to
find a $P$-clique of size at least $n^{1/15}/4$, proceeding in stages.

\medskip\noindent\textit{Stage 1. }By Mirsky's theorem applied to
$\preceq_{1}$ on $M_{1}$ (recalling (1)), there is either a $\mathcal{P}_{1}$-clique
of size at least $n^{1/15}$, or a $\mathcal{P}_{1}$-free submatching
$M_{2}$ of size at least $n^{14/15}$. In the former case, we are
done by (3); in the latter case move on to Stage 2.

\medskip\noindent\textit{Stage 2. }By Mirsky's theorem applied to
$\preceq_{2}$ on $M_{2}$ (recalling (1)), there is either a $\mathcal{P}_{2}$-clique
of size at least $n^{1/15}$, or a $(\mathcal{P}_{1}\cup\mathcal{P}_{2})$-free
submatching $M_{3}$ of size at least $n^{13/15}$. In the former
case, we are done by (3); in the latter case move on to Stage 3.

\medskip\noindent\textit{Stage 3. }By Mirsky's theorem applied to
$\preceq_{3}$ on $M_{3}$ (recalling (2)), there is either a $\mathcal{P}_{3}$-clique
of size at least $n^{1/15}$, or a $(\mathcal{P}_{1}\cup\mathcal{P}_{2}\cup\mathcal{P}_{3})$-free
submatching $M_{4}$ of size at least $n^{12/15}$. In the former
case, we observe that $\mathcal{P}_{3}$ is left-dominated, so we
are done by \cref{lem:dominated-pigeonhole}; in the latter case move on to stage 4.

\medskip\noindent\textit{Stage 4. }In this stage we wish to eliminate
the non-collectable patterns $P_{\line\line\stack}^{*}$ and $P_{\stack\line\line}^{*}$
in $\mathcal{P}_{4}$ without seriously decreasing the size of our
matching. List all the edges of $M_{4}$ as $e_{1},\dots,e_{m}$ with
$m=|M_{4}|\ge n^{12/15}$ and $e_{1}[1]<\dots<e_{m}[1]$. We claim
that for any $1\le i<j<k\le m$, at most one of the pairs $(e_{i},e_{j})$
$(e_{j},e_{k})$ can form $P_{\line\line\stack}^{*}$. Indeed, if
both pairs were to form $P_{\line\line\stack}^{*}$, then $(e_{i},e_{k})$
would also form $P_{\line\line\stack}^{*}$ by \cref{lemma: composition of two weak patterns} (recalling that
$M_{4}$ is $\psi(\line\line\stack)$-free). But then, $e_{i},e_{j},e_{k}$
would form a $P_{\line\line\stack}^{*}$-clique, which is impossible
as $P_{\line\line\stack}^{*}$ is not collectable.

By the above claim, every edge of $M_{4}$ is left-$P_{\line\line\stack}^{*}$-free
or right-$P_{\line\line\stack}^{*}$-free. So, by the pigeonhole principle
we can find a submatching $M_{5}'$ of at least $n^{12/15}/2$ edges
which are either all left-$P_{\line\line\stack}^{*}$-free, or all
right-$P_{\line\line\stack}^{*}$-free, meaning that $M_{5}'$ itself
is $P_{\line\line\stack}^{*}$-free.

Repeating all this analysis with the pattern $P_{\stack\line\line}^{*}$,
we can find a matching $M_{5}$ of size at least $n^{12/15}/4$ that
is $(\mc P_1\cup\dots\cup\mathcal{P}_{4})$-free.

We have now eliminated all non-collectable patterns, so we will refer to each pattern by its corresponding weak pattern (e.g., we simply write $\line\stack\stack$ instead of $\psi(\line\stack\stack)$). Move on to stage 5.

\medskip\noindent\textit{Stage 5. } 
By Mirsky's theorem applied to
$\preceq_{5}$ on $M_{5}$ (recalling (1)), there is either a $\mathcal{P}_{5}$-clique
of size at least $n^{1/15}$, or a $(\mathcal{P}_{1}\cup\dots\cup\mathcal{P}_{5})$-free
submatching $M_{6}$ of size at least $n^{11/15}/4$. In the former
case, we are done by (3); in the latter case move on to Stage 6.

\medskip\noindent\textit{Stage 6. }The analysis now starts to get
a bit more intricate, and more tedious casework becomes necessary.
Recall that $\mathcal{P}_{6}=\{\line\stack\stack,\stack\line\stack,\stack\stack\line\}.$
It is not necessarily the case that $\preceq_{6}$ is a poset; in
\cref{table: 5-11-13} we tabulate the possible compositions of patterns in $\mathcal{P}_{6}$
(among patterns that can be present in $M_{6}$). These can all be determined by direct case-checking. Crucially, we have
\[
(\line\stack\stack\circ\stack\stack\line)\setminus(\mathcal{P}_{1}\cup\dots\cup\mathcal{P}_{5})=\emptyset.
\]
That is to say, every edge in $M_{6}$ is left-$\line\stack\stack$-free
or right-$\stack\stack\line$-free. So, by the pigeonhole principle
we can find a submatching $M_{6}'$ of size at least $n^{11/15}/8$
which is $\line\stack\stack$-free or $\stack\stack\line$-free.

Suppose that $M_{6}'$ is $\line\stack\stack$-free (the $\stack\stack\line$-free
case is handled basically symmetrically), and let $\mathcal{P}_{6}'=\{\stack\line\stack,\stack\stack\line\}$.
\cref{table: 5-11-13} shows that $\preceq_{\mathcal{P}_{6}'}$ is a poset on $M_{6}'$
(consider the sub-table induced by $\stack\line\stack$ and
$\stack\stack\line$). So, by Mirsky's theorem we can find a
$\mathcal{P}_{6}'$-clique of size $n^{1/15}$ or a $(\mathcal{P}_{1}\cup\dots\cup\mathcal{P}_{6})$-free
submatching $M_{7}$ of size at least $n^{10/15}/8$. In the former
case, \cref{table: 5-11-13} shows that $\mathcal{P}_{6}'$ is right-dominated and we
are done by \cref{lem:dominated-pigeonhole}. In the latter case, move on to stage 7.

\begin{table}
\centering %
\begin{tabular}{|c|c|c|c|}
\hline 
\diagbox{$P$}{$Q$} & $\line\stack\stack$ & $\stack\line\stack$ & $\stack\stack\line$\tabularnewline
\hline 
$\line\stack\stack$ & $\line\stack\stack$ & $\line\stack\stack$ & $\emptyset$\tabularnewline
\hline 
$\stack\line\stack$ & $\stack\line\stack$ & $\stack\line\stack$ & $\stack\stack\line$\tabularnewline
\hline 
$\stack\stack\line$ & ? & $\stack\line\stack$ & $\stack\stack\line$\tabularnewline
\hline 
\end{tabular}\caption{Possible compositions of the patterns in $\mathcal{P}_{6}$. The cell
indexed by $(P,Q)$ indicates the patterns in $P\circ Q$, apart from
the patterns in $\mathcal{P}_{1}\cup\dots\cup\mathcal{P}_{5}$. For
each $P,Q\in\mathcal{P}_{6}$, there is only a single pattern in $(P\circ Q)\setminus(\mathcal{P}_{1}\cup\dots\cup\mathcal{P}_{5})$,
unless $(P,Q)=(\protect\line\protect\stack\protect\stack,\protect\stack\protect\stack\protect\line)$
or $(P,Q)=(\protect\stack\protect\stack\protect\line,\protect\line\protect\stack\protect\stack)$.
(In the former case there are no valid possibilities, and in the latter
case there are several (including some not in $\mathcal{P}_{6}$),
which we do not need to carefully enumerate.}
\label{table: 5-11-13}
\end{table}

\medskip\noindent\textit{Stage 7. }Recall that $\mathcal{P}_{7}=\{\line\wave\stack,\line\stack\wave,\stack\line\wave,\stack\wave\line,\wave\line\stack,\wave\stack\line\}.$
We tabulate in \cref{table: 6-8-12-16-20-22} the possible compositions of patterns in $\mathcal{P}_{7}$
(among patterns that can be present in $M_{7}$). Note in particular
that $\preceq_{P}$ is a partial order for each $P\in\mathcal{P}_{7}$.

\begin{table}
\centering %
\begin{tabular}{|c|c|c|c|c|c|c|}
\hline 
\diagbox{$P$}{$Q$} & $\line\stack\wave$ & $\line\wave\stack$ & $\stack\line\wave$ & $\stack\wave\line$ & $\wave\line\stack$ & $\wave\stack\line$\tabularnewline
\hline 
$\line\stack\wave$ & $\line\stack\wave$ & $\line\stack\wave,\line\wave\stack$ & $\line\stack\wave$ & $\emptyset$ & $\line\stack\wave,\line\wave\stack$ & $\emptyset$\tabularnewline
\hline 
$\line\wave\stack$ & $\line\stack\wave,\line\wave\stack$ & $\line\wave\stack$ & $\line\stack\wave$ & $\line\stack\wave,\line\wave\stack$ & $\emptyset$ & $\line\stack\wave,\line\wave\stack$\tabularnewline
\hline 
$\stack\line\wave$ & $\stack\line\wave$ & $?$ & $\stack\line\wave$ & $\emptyset$ & $?$ & $\emptyset$\tabularnewline
\hline 
$\stack\wave\line$ & $\stack\wave\line,\wave\stack\line$ & $\stack\wave\line,\wave\stack\line$ & $\emptyset$ & $\stack\wave\line$ & $\wave\stack\line$ & $\stack\wave\line,\wave\stack\line$\tabularnewline
\hline 
$\wave\line\stack$ & $\emptyset$ & $\emptyset$ & $?$ & $?$ & $\wave\line\stack$ & $\wave\line\stack$\tabularnewline
\hline 
$\wave\stack\line$ & $\emptyset$ & $\emptyset$ & $\stack\wave\line,\wave\stack\line$ & $\stack\wave\line,\wave\stack\line$ & $\wave\stack\line$ & $\wave\stack\line$\tabularnewline
\hline 
\end{tabular}\caption{Possible compositions of the patterns in $\mathcal{P}_{7}$. The cell
indexed by $(P,Q)$ indicates the patterns in $P\circ Q$, apart from
the patterns in $\mathcal{P}_{1}\cup\dots\cup\mathcal{P}_{6}$. There
are ten pairs $(P,Q)$ for which there are no valid compositions,
and four cells in which there are many valid compositions (which we
do not need to carefully enumerate).}
\label{table: 6-8-12-16-20-22}
\end{table}

Note that there are ten different pairs of patterns $P,Q\in\mathcal{P}_{7}$
such that 
\[
(P\circ Q)\setminus(\mathcal{P}_{1}\cup\dots\cup\mathcal{P}_{6})=\emptyset.
\]
(Call such pairs ``special pairs''). For each special pair $(P,Q)$,
we know that every edge in $M_{7}$ is either left-$P$-free or right-$Q$
free. By the pigeonhole principle, we can find a submatching $M_{7}'$
of size at least $(n^{10/15}/8)/2^{10}$ such that for each special
pair $(P,Q)$, our submatching $M_{7}'$ is either $P$-free or $Q$-free.
Let $\mathcal{P}_{7}'$ be the collection of all patterns present
in $M_{7}'$. By carefully considering all cases, we can check that
either $|\mathcal{P}_{7}'|\le2$, or $\mathcal{P}_{7}'=\{\line\stack\wave,\line\wave\stack,\stack\line\wave\}$,
or $\mathcal{P}_{7}'=\{\stack\wave\line,\wave\stack\line,\wave\line\stack\}$.
The second and third cases are basically symmetric to each other, so we
just describe how to handle the first two cases.
\begin{compactitem}
\item If $|\mathcal{P}_{7}'|\le2$, then we can apply Mirsky's theorem once or
twice to obtain a $P$-clique of size at least $n^{1/15}/4$ (for
some $P\in\mathcal{P}_{7}'$) or a $(\mathcal{P}_{1}\cup\dots\cup\mathcal{P}_{7})$-free
submatching $M_{8}$ of size at least $n^{8/15}/2^{9}$. In the former
case we are done; in the latter case move on to stage 8.
\item If $\mathcal{P}_{7}'=\{\line\stack\wave,\line\wave\stack,\stack\line\wave\}$,
then let $\mathcal{P}_{7}''=\{\line\stack\wave,\stack\line\wave\}$.
Observe from \cref{table: 6-8-12-16-20-22} that $\mathcal{P}_{7}''$ is left-dominating and that
$\preceq_{\mathcal{P}_{7}''}$ is a partial order. Applying Mirsky's
theorem twice, $M_{7}'$ has a $\line\wave\stack$-clique of
size at least $n^{1/15}/4$ or a $\mathcal{P}_{7}''$-clique of size
at least $n^{1/15}/2$ or a $(\mathcal{P}_{1}\cup\dots\cup\mathcal{P}_{7})$-free
submatching $M_{8}$ of size at least $n^{8/15}/2^{10}$. In the first
case we are immediately done, in the second case we are done by \cref{lem:dominated-pigeonhole},
and in the third case we move on to stage 8.
\end{compactitem}

\medskip\noindent\textit{Stage 8. }Now, we have a matching $M_{8}$
of size at least $n^{8/15}/2^{10}\ge n^{8/15}/4^{8}$ which contains
only the eight 4-partite 4-patterns. By the proof of \cref{lem:key-ES}(A), for each 4-partite 4-pattern $P$, the relation $\preceq_{P}$
is a partial order (each pattern corresponds to a signature of the form $(0,S)$
for $S\subseteq\{1,2,3\}$). So, by Mirsky's theorem we have $L(M_{8})\ge n^{1/15}/4$,
and we are done.
\end{proof}